\documentclass[10pt]{amsart}
\usepackage{amsmath,amssymb,latexsym,cancel,rotating}
\usepackage{graphicx,mathrsfs,color,fancyhdr,amsthm}

\usepackage[all]{xy}
\usepackage{geometry}
\usepackage{hyperref}
\usepackage{cite}
\usepackage{setspace}
\newtheorem{theorem}{Theorem}[section]
\newtheorem{lemma}[theorem]{Lemma}
\newtheorem{corollary}[theorem]{Corollary}
\newtheorem{definition}[theorem]{Definition}
\newtheorem{proposition}[theorem]{Proposition}

\newtheorem{remark}[theorem]{Remark}

\def\cC{{\mathcal C}}

\def\bbN{{\mathbb N}} \def\bbZ{{\mathbb Z}} \def\bbQ{{\mathbb Q}}
  
\def\bbC{{\mathbb C}}  \def\bbP{{\mathbb P}}

     \def\bfU{{\bf U}}
  
 \def\leq{\leqslant} \def\geq{\geqslant}

\def\Hom{\mbox{\rm Hom}} 
    
\def\Ext{\mbox{\rm Ext}}  
\def\dim{\mbox{\rm dim}}  
\def\coker{\mbox{\rm coker}}

\def\bfV{{\mathbf V}}

\def\supp{{\rm supp}}

\def\bfV{\mathbf{V}}
\def\bfW{\mathbf{W}}
\def\bfE{\mathbf{E}}
\def\bfG{\mathbf{G}}

\def\bfU{\mathbf{U}}

\def\bfM{\mathbf{M}}

\newcommand\mk{{\mathcal{K}}}
\newcommand\mn{{\mathcal{N}}}

\newcommand\mm{{\mathcal{M}}}

\newcommand\me{{\mathcal{E}}}
\newcommand\mf{{\mathcal{F}}}
\newcommand\mi{{\mathcal{I}}}
\newcommand\mo{{\mathcal{O}}}
\newcommand\mP{{\mathcal{P}}}
\newcommand\mq{{\mathcal{Q}}}

\newcommand\ml{{\mathcal{L}}}

\newcommand\mv{{\mathcal{V}}}

\begin{document}

\title[]{Lusztig sheaves and integrable highest weight modules in the symmetrizable case}

\author{Yixin Lan, Yumeng Wu, Jie Xiao}
\address{Max Planck Institute for Mathematics}
\email{lanyixin@amss.ac.cn (Y. Lan)}
\address{Beijing International Center for Mathematical Research, Peking University, Beijing 100871, P. R. China}
\email{wuym25@pku.edu.cn (Y. Wu)}
\address{School of Mathematical Sciences, Beijing Normal University, Beijing 100875, P. R. China}
\email{jxiao@bnu.edu.cn(J.Xiao)}



\bibliographystyle{abbrv}

\begin{abstract}
This paper continues the work of \cite{fang2023lusztigsheavesintegrablehighest} and \cite{fang2023lusztigsheavestensorproducts}. For a symmetrizable generalized Cartan matrix $C$ and the corresponding quantum group $\mathbf{U}$, we consider an associated quiver $Q$ equipped with an admissible automorphism $a$. We construct a category $\widetilde{\mathcal{Q}/\mathcal{N}}$ obtained from localizations of Lusztig sheaves for the corresponding framed and $2$-framed quivers with automorphism. The Grothendieck groups of these categories realize the integrable highest weight module $L(\lambda)$ and the tensor product $L(\lambda_1)\otimes L(\lambda_2)$ of integrable highest weight $\mathbf{U}$-modules. After quotienting by traceless objects, Lusztig sheaves yield the signed canonical bases of $L(\lambda)$ and $L(\lambda_1)\otimes L(\lambda_2)$. As applications, we recover symmetrizable crystal structures on Nakajima quiver varieties, Nakajima tensor product varieties, and Lusztig nilpotent varieties of preprojective algebras.
\end{abstract}

\maketitle

\setcounter{tocdepth}{1}\tableofcontents
\begin{spacing}{1.5}
\section{Introduction}
\subsection{Symmetrizable Cartan data and quantum groups}
Let $C=(c_{ij})_{i,j\in I'}=DB$ be a symmetrizable generalized Cartan
matrix, where $D=\mathbf{diag}(\frac{1}{s_i})_{i\in I'}$ is diagonal
and $B=(b_{ij})_{i,j\in I'}$ is symmetric. This matrix determines a
bilinear form $\langle\ ,\ \rangle : \bbZ[I'] \times \bbZ[I']
\rightarrow \bbZ$ by $\langle i,j\rangle = c_{ij}$.
The corresponding quantum group $\mathbf{U}$ is the
$\mathbb{Q}(v)$-algebra generated by $E_i, F_i, K_{\nu}$ for
$i\in I', \nu\in \bbZ[I']$, subject to the following relations:
\begin{itemize}
  \item {\rm{(a)}} $K_0 = 1$, and $K_{\nu}K_{\nu'} = K_{\nu+\nu'}$
  for any $\nu,\nu' \in \bbZ[I']$;
  \item {\rm{(b)}} $K_{\nu}E_i = v^{\langle \nu,i \rangle} E_i K_{\nu}$
  for $i \in I', \nu \in \bbZ[I']$;
  \item {\rm{(c)}} $K_{\nu}F_i = v^{-\langle \nu,i \rangle} F_i K_{\nu}$
  for $i \in I', \nu \in \bbZ[I']$;
  \item {\rm{(d)}} $E_i F_j - F_j E_i = \delta_{ij}
  \frac{\tilde{K}_{i} - \tilde{K}_{-i}}{v_i - v_i^{-1}}$;
  \item {\rm{(e)}} $\sum\limits_{p+q=1-c_{ij}} (-1)^p
  E_i^{(p)} E_j E_i^{(q)} = 0$ for $i\ne j$;
  \item {\rm{(f)}} $\sum\limits_{p+q=1-c_{ij}} (-1)^p
  F_i^{(p)} F_j F_i^{(q)} = 0$ for $i\ne j$.
\end{itemize}
Here we use the notation $v_i := v^{s_i}$.
For $\nu = \sum\limits_{i\in I'} \nu_i i \in \bbZ[I']$, we denote
$\tilde{K}_{\nu} := \prod\limits_{i\in I'} K_{s_i \nu_i i}$, and
$E_i^{(n)} := E_i^n/[n]^!_i$, $F_i^{(n)} := F_i^n/[n]^!_i$ are the
divided powers of $E_i$ and $F_i$, where
$[n]_i := \frac{v_i^n - v_i^{-n}}{v_i - v_i^{-1}}$,
$[n]^!_i := \prod_{s=1}^n [s]_i$, and $[0]^!_i=1$.

Fix a weight lattice $X$ with linearly independent simple roots
$\alpha_i\in X$, simple coroots
$h_i\in\operatorname{Hom}_{\bbZ}(X,\bbZ)$, and fundamental weights
$\beta_i\in X$, satisfying
\[
  \langle h_i,\alpha_j\rangle=c_{ij},
  \qquad
  \langle h_i,\beta_j\rangle=\delta_{ij}.
\]
For $\nu=\sum_{i\in I'}\nu_i i\in\bbZ[I']$ and $\lambda\in X$, extend
the pairing notation by
\[
  \langle\nu,\lambda\rangle
     :=\sum_{i\in I'}\nu_i\langle h_i,\lambda\rangle.
\]
This agrees with the preceding bilinear form after identifying its
second copy of $\bbZ[I']$ with the root lattice
$\bigoplus_{i\in I'}\bbZ\alpha_i\subseteq X$.
Write $X^+=\{\lambda\in X\mid\langle i,\lambda\rangle\geq0
\text{ for all }i\in I'\}$.

Let $M$ be a $\mathbf{U}$-module with an $X$-graded weight-space
decomposition $M = \oplus_{\lambda \in X} M^{\lambda}$ such that,
for any $m \in M^{\lambda}$ and $\nu\in\bbZ[I']$, one has
$K_{\nu}m = v^{\langle \nu,\lambda \rangle}m$, and
\[
  E_iM^{\lambda}\subseteq M^{\lambda+\alpha_i},
  \qquad
  F_iM^{\lambda}\subseteq M^{\lambda-\alpha_i}.
\]
We say that $M$ is integrable if, for every $m \in M$, there exists
$n_0 \in \bbN$ such that $E_i^{(n)}m=0$ and $F_i^{(n)}m=0$ for all
$i\in I'$ and $n>n_0$. We refer to \cite{Jantzen,L01} for details.

For a dominant weight $\lambda\in X^+$ (that is,
$\langle i, \lambda \rangle \in \bbN$ for all $i \in I'$), the
irreducible integrable highest weight module $L(\lambda)$ is defined by
\[
L(\lambda) := \mathbf{U}/ \left(
  \sum\limits_{i\in I'} \mathbf{U} E_i
  + \sum\limits_{\nu \in \bbZ[I']}
      \mathbf{U} (K_{\nu} - v^{\langle \nu,\lambda \rangle})
  + \sum\limits_{i \in I'}
      \mathbf{U} F_i^{\langle i,\lambda \rangle +1}
\right).
\]
We denote the image of $1$ by $v_\lambda$ and assign it weight $\lambda$.
The $\mathbf{U}$-module structure on $L(\lambda_1)\otimes L(\lambda_2)$
is induced by the comultiplication on $\mathbf{U}$.	
\subsection{Lusztig's sheaves for quivers with automorphisms}\label{quiver a}	In \cite[Chapters 11, 12, and 14]{lusztig2010introduction}, Lusztig considers a finite quiver $Q=(I,H,\Omega)$ with an admissible automorphism $a$ attached to a given Cartan matrix. 

A finite quiver $Q=(I,H,\Omega)$ consists of finite sets $I$ and $H$, together with a subset $\Omega\subset H$, where $I$ is the set of vertices, $H$ is the set of all oriented arrows $s(h) \xrightarrow{h} t(h)$, and $H$ is the disjoint union of $\Omega$ and $\bar{\Omega}$. Here $\bar{\quad}:H\rightarrow H$ sends an arrow to the same edge with the opposite orientation. Such a subset $\Omega$ is called an orientation of $Q$.

For a finite quiver $(I,H,\Omega)$ without loops, an admissible automorphism $a$ consists of two permutations $a:I\rightarrow I$ and $a:H\rightarrow H$ satisfying the following conditions:\\
{\rm{(a)}} $a(s(h))=s(a(h))$ and $a(t(h))=t(a(h))$ for any $h\in H$;\\
{\rm{(b)}} $s(h)$ and $t(h)\in I$ belong to different $a$-orbits for any $h\in H$.

Given a symmetrizable generalized Cartan matrix $C=(c_{ij})_{i,j\in I'}=DB$, there exists, in general non-uniquely, a finite quiver $Q=(I,H,\Omega)$ with an admissible automorphism $a$ such that the set of $a$-orbits in $I$ is in bijection with $I'$, the order of the $a$-orbit corresponding to $i \in I'$ is $s_i$, and the number of arrows between the $a$-orbits corresponding to $i,j \in I'$ is $c_{i,j}s_i=c_{j,i}s_j$. For such a quiver with the automorphism $a$, Lusztig considers a full subcategory $\mathcal{Q}=\coprod\limits_{\mathbf{V}} \mathcal{Q}_{\mathbf{V}}$ of semisimple perverse sheaves on the moduli space $\coprod\limits_{\mathbf{V}}\mathbf{E}_{\mathbf{V},\Omega}$ of quiver representations. Via the periodic functor $a^{\ast}$, Lusztig has constructed a certain graded additive category $\tilde{\mathcal{Q}}$ from $\mathcal{Q}$, consisting of pairs $(B,\phi)$ of semisimple complexes $B$ and morphisms $\phi: a^{\ast} B \rightarrow B$. He has also defined the induction and restriction functors for the category $\tilde{\mathcal{Q}}$. (See \cite[Chapter 9 and 12]{lusztig2010introduction} or Section 2.1.) The (generalized) Grothendieck group of $\tilde{\mathcal{Q}}$ has a bialgebra structure, which is canonically isomorphic to the positive (or negative) part $\mathbf{U}^{\pm}$ of $\mathbf{U}$. Moreover, certain objects of  $\tilde{\mathcal{Q}}$ provide an integral basis of $\mathbf{U}^{\pm}$, which is called the signed basis by Lusztig. 

When the generalized Cartan matrix $C$ is symmetric, or equivalently, the admissible automorphism $a$ is trivial, the simple perverse sheaves in $\mathcal{Q}$ form a basis of $\mathbf{U}^{\pm}$, which is called the canonical basis and shares many remarkable properties.

Lusztig also defines induction and restriction functors as follows. For a fixed subspace $\bfW\subset \bfV$, $\dim \bfW=\dim \bfV_2$, $\bfV\cong \bfV_1\oplus\bfV_2$, $P$ is the stabilizer of $\bfW$ in $G_{\bfV}$, $U$ is the unipotent radical of $P$ and $F=\{x\in \bfE_{\bfV}|x(\bfW)\subset \bfW\}$.

$$\bfE_{\bfV_1}\times \bfE_{\bfV_2}\xleftarrow{p_1}G_{\bfV}\times_{U}F\xrightarrow{p_2}G_{\bfV}\times_{P}F\xrightarrow{p_3}\bfE_{\bfV},$$

where $p_1(g,x)=(x|_{\bfV/\bfW},x|_\bfW)$, $p_2$ is the quotient map, and $p_3(g,x)=g(x)$.

$$\bfE_{\bfV_1}\times \bfE_{\bfV_2}\xleftarrow{\kappa}F\xrightarrow{\iota}\bfE_{\bfV},$$

where $\iota$ is the embedding and $\kappa(x)=(x|_{\bfV/\bfW},x|_{\bfW})$.

Lusztig has defined $\mathbf{Ind}^{\bfV}_{\bfV_1,\bfV_2}=(p_3)_!(p_2)_{\flat}(p_1)^*[d_1-d_2]$, where the relative dimension of $p_i$ is $d_i$, $(p_2)_{\flat}$ is the inverse functor of $p_2^*$ and $\mathbf{Res}_{\bfV_1,\bfV_2}^{\bfV}=(\kappa)_!(\iota)^*[d_1-d_2-2\dim G_{\bfV}/P]$.

\subsection{Localizations of Lusztig's sheaves and other categorical realizations}

When the generalized Cartan matrix $C$ is symmetric, inspired by \cite{2007A,zheng2014categorification,2014Tensor}, the authors of \cite{fang2023lusztigsheavesintegrablehighest} and \cite{fang2023lusztigsheavestensorproducts} have considered certain localizations of Lusztig's category $\mathcal{Q}$ to realize the irreducible integrable highest weight module $L(\lambda)$ and tensor product $L(\lambda_1)\otimes L(\lambda_2)$.

For a dominant weight $\lambda_1,\lambda_2$, they have considered the moduli space $\mathbf{E}_{\mathbf{V},\mathbf{W}^{\bullet},\Omega}$ of a framed quiver and certain thick subcategory $\mathcal{N}_{\mathbf{V}}$ of the $G_{\mathbf{V}}$-equivariant derived category $\mathcal{D}^{b}_{G_{\mathbf{V}} } (\mathbf{E}_{\mathbf{V},\mathbf{W}^{\bullet},\Omega} )$ of constructible $\overline{\mathbb{Q}}_{l}$-sheaves. They have defined functors $\mathcal{E}_{i}$ and $\mathcal{F}_{i} , i \in I$ for the Verdier quotient $\mathcal{D}^{b}_{G_{\mathbf{V}} } (\mathbf{E}_{\mathbf{V},\mathbf{W}^{\bullet},\Omega} )/\mathcal{N}_{\mathbf{V}}$ and showed that Lusztig's sheaves $\mathcal{Q}_{\mathbf{V},\mathbf{W}^{\bullet}}$ for the framed quiver are preserved by these functors $\mathcal{E}_{i}$ and $\mathcal{F}_{i}$ (up to isomorphisms in localizations). (See \cite{fang2023lusztigsheavesintegrablehighest} or Section 2.2.) The Grothendieck group of $\coprod\limits_{\mathbf{V}}\mathcal{Q}_{\mathbf{V},\mathbf{W}^{\bullet}}/\mathcal{N}_{\mathbf{V}}$ becomes a $\mathbf{U}$-module with the action of functors $\mathcal{E}_{i}$ and $\mathcal{F}_{i}$, and is isomorphic to $L(\lambda)$ canonically. Moreover, the nonzero simple perverse sheaves in $\mathcal{Q}_{\mathbf{V},\mathbf{W}^{\bullet}}/\mathcal{N}_{\mathbf{V}}$ form the canonical basis of $L(\lambda_1)\otimes L(\lambda_2)$.

\subsection{Perverse sheaf realizations of integrable highest weight modules and their tensor products}

Since the work of Zheng~\cite{2007A,zheng2014categorification}, there has been a sequence of developments in the categorification of integrable highest weight modules. For example, the authors of~\cite{CKL} construct geometric $\mathfrak{g}$-actions on coherent sheaves over quiver varieties, while the authors of~\cite{BL} consider geometric $\mathfrak{g}$-actions on quantized quiver varieties. Webster also uses both D-modules~\cite{WG} and diagrammatic approaches~\cite{webster2015canonical} to construct similar categorical actions. Kashiwara and Kang, in~\cite{Seok2012Categorification}, have constructed highest weight modules via KLR algebras, including the symmetrizable cases. 

The present paper provides a geometric construction of highest weight modules and their tensor products in the symmetrizable case, within the framework of Lusztig~\cite[Chapter 12]{lusztig2010introduction}. This work continues the program initiated by Fang, Lan, and Xiao~\cite{fang2023lusztigsheavestensorproducts, fang2023lusztigsheavesintegrablehighest}.

From the perspective of sheaf complexes, it is natural to ask how to generalize the construction of \cite{fang2023lusztigsheavesintegrablehighest} to the symmetrizable case. More precisely, given a symmetrizable generalized Cartan matrix $C$, our goal is to realize the irreducible highest weight module $L(\lambda)$ of $\mathbf{U}$ associated with $C$ via sheaves on moduli spaces of a quiver with automorphism. 

In this article, we apply Lusztig's method of periodic functors to the localizations for $N$-framed quivers introduced in \cite{fang2023lusztigsheavestensorproducts}, and obtain a category $\widetilde{\mathcal{Q}/\mathcal{N}}$. Its Grothendieck group is a $\mathbf{U}$-module, which is canonically isomorphic to $L(\lambda)$ of the symmetrizable quantum group associated to $C$ and, in the $2$-framed case, to $L(\lambda_1)\otimes L(\lambda_2)$. See Theorem \ref{4.6} and \ref{span} for details. 

Moreover, under the canonical isomorphism, the objects $(B,\phi)$ in $\widetilde{\mathcal{Q}/\mathcal{N}}$ naturally give the signed basis of $L(\lambda)$, and this basis is almost orthogonal with respect to a contravariant geometric pairing. See Proposition \ref{bilinear} for details. The signed basis of $L(\lambda_1)\otimes L(\lambda_2)$ obtained in the same way coincides with Lusztig's signed basis for tensor products of highest weight modules. 

As an application, following \cite{fang2023lusztigsheavestensorproducts} we give a new proof of the Yang--Baxter relation for the operators $\mathbf{R}_{ij}$ arising from the functor $\Delta'\circ (\Delta^{\prime,\vee})^{-1}$. We also deduce a symmetrizable crystal structure on Nakajima's quiver varieties in \cite{N94} and \cite{N98}, Nakajima's tensor product varieties in \cite{Ntensor},\cite{Ntensor2} and \cite{Malkin2003Tensor} and Lusztig's nilpotent varieties \cite{semicb} of preprojective algebras. See Proposition \ref{Yang}, Theorems \ref{4.15} and \ref{5.7} and their corollaries for details. 

\subsection{Organization of the paper}
In Section 2, we recall the notation and results from \cite{fang2023lusztigsheavesintegrablehighest}. In Section 3, we define our category	$\widetilde{\mathcal{Q}/\mathcal{N}}$ and its functors, and prove the commutation relations of these functors. The proofs of these commutation relations differ substantially from those in the symmetric case. In Section 4, we determine the module structure of the Grothendieck group of $\widetilde{\mathcal{Q}/\mathcal{N}}$, and construct the signed bases of $L(\lambda)$ and $L(\lambda_1)\otimes L(\lambda_2)$. As applications, in Section 5, we verify the Yang--Baxter equation and deduce crystal structures on quiver varieties, tensor product varieties and nilpotent varieties in the symmetrizable case.
\subsection*{Acknowledgements} Y. Lan is partially supported by the National Natural Science Foundation of China [Grant No. 1288201] and National Natural Science Foundation of China [Grant No. 12471030],
Yumeng Wu and J. Xiao are supported by National Natural Science Foundation of China [Grant No. 12031007]
and National Natural Science Foundation of China [Grant No. 12471030]. This paper is a continuation of a collaborative work \cite{fang2023lusztigsheavesintegrablehighest} with Jiepeng Fang, and we are very grateful to him for many helpful discussions.

\section{The category $\mq_{\bfV, \bfW^{\bullet}}/\mn_{\bfV}$ and functors $\mathcal{E}^{(n)}_{i},\mathcal{F}^{(n)}_{i}$}

\subsection{Lusztig's sheaves}\label{Lusztig sheaf}

In this subsection, we recall the category of Lusztig sheaves for quivers. Let $\mathbf{k}$ be an algebraically closed field with $\operatorname{char}(\mathbf{k})=p>0$. Given a quiver $Q=(I,H,\Omega)$ and an $I$-graded $\mathbf{k}$-space $\bfV=\bigoplus\limits_{i \in I} \bfV_{i}$ of dimension vector $\nu$, the affine variety $\bfE_{\bfV,\Omega}$ is defined by 
$$ \bfE_{\bfV,\Omega}=\bigoplus\limits_{h\in\Omega} \Hom(\bfV_{s(h)},\bfV_{t(h)} ). $$ The algebraic group $G_{\bfV}=\prod\limits_{i \in I}GL(\bfV_{i})$ acts naturally on $\bfE_{\bfV,\Omega}$ by change of basis. We denote the $G_{\bfV}$-equivariant derived category of constructible $\overline{\bbQ}_{l}$-sheaves on $\bfE_{\bfV,\Omega}$ by $\mathcal{D}^{b}_{G_{\bfV}}(\bfE_{\bfV,\Omega})$. We denote by $[n]$ the shift functor by $n$ and by $\mathbf{D}$ the Verdier duality functor.

Let $\mathcal{S}$ be the set of finite sequences $\boldsymbol{\nu}=(\nu^{1},\nu^{2},\cdots,\nu^{s})$ of dimension vectors such that each $\nu^{l}=a_{l}i_{l}$ for some $a_{l} \in \bbN_{\geqslant 1}$ and $i_{l} \in I$. If $\sum\limits_{1\leqslant l \leqslant s} \nu^{l} =\nu$, we say $\boldsymbol{\nu}$ is a flag type of $\nu$ or $\bfV$. 

For a flag type $\boldsymbol{\nu}$ of $\bfV$, the flag variety $\mathcal{F}_{\boldsymbol{\nu},\Omega}$ is the smooth variety which consists of pairs $(x,f)$, where $x \in \bfE_{\bfV,\Omega}$ and $f=(\bfV=\bfV^{s} \subseteq \bfV^{s-1} \subseteq \cdots \subseteq \bfV^{0}=0 )$ is a filtration of the $I$-graded space such that $x(\bfV^{l})\subseteq \bfV^{l}$ and the dimension vector of $\bfV^{l-1}/\bfV^{l}=\nu^{l}$ for any $l$. There is a proper map $\pi_{\boldsymbol{\nu},\Omega}: \mathcal{F}_{\boldsymbol{\nu},\Omega} \rightarrow \bfE_{\bfV,\Omega}; (x,f) \mapsto x. $ Hence by the decomposition theorem in \cite{BBD}, the complex $L_{\boldsymbol{\nu}}= (\pi_{\boldsymbol{\nu},\Omega})_{!} \bar{\mathbb{Q}}_{l}[\dim \mathcal{F}_{\boldsymbol{\nu},\Omega}]$ is a semisimple complex on $\mathbf{E}_{\mathbf{V},\Omega}$, where $\bar{\mathbb{Q}}_{l}$ is the constant sheaf on $\mathcal{F}_{\boldsymbol{\nu},\Omega}$. 
\begin{definition}
	Let $\mathcal{P}_{\bfV}$ be the set consisting of those simple perverse sheaves $L$ in $\mathcal{D}^{b}_{G_{\bfV}}(\bfE_{\bfV,\Omega})$ such that $L$ is a direct summand (up to shifts) of $L_{\boldsymbol{\nu}}$ for some flag type $\boldsymbol{\nu}$ of $\bfV$. Let $\mq_{\bfV}$ be the full subcategory of $\mathcal{D}^{b}_{G_{\bfV}}(\bfE_{\bfV,\Omega})$, which consists of
	finite direct sums of shifted simple perverse sheaves in $\mathcal{P}_{\bfV}$. Following \cite{OSNotes}, we call $\mq_{\bfV}$ the category of Lusztig sheaves for $Q$.
\end{definition}

In order to realize the tensor product of $N$ irreducible integrable highest weight $\mathbf{U}$-modules, we introduce the $N$-framed quiver. The $N$-framed quiver $\tilde{Q}^{(N)}$ of $Q=(I,H,\Omega)$ is defined by $\tilde{Q}^{(N)} = (I^{(N)}, \tilde{H}^{(N)}, \tilde{\Omega}^{(N)})$. Here $I^{(N)}=I \cup I^{1} \cup \cdots \cup I^{N}$ such that each $I^{k}$ is a copy of $I$, and we denote the $k$-th copy of $i \in I$ by $i^{k} \in I^{k}$, and $\tilde{\Omega}^{(N)}$ is the set of oriented arrows $\Omega \cup \{i \rightarrow i^k,i \in I, 1\leqslant k \leqslant N\}$, and $\tilde{H}^{(N)}=\tilde{\Omega}^{(N)} \cup \bar{\tilde{\Omega}}^{(N)} $. Here $\bar{\quad }: \tilde{H}^{(N)} \rightarrow \tilde{H}^{(N)}$ is the involution defined by taking the opposite orientation. When $N=1$, the $N$-framed quiver $\tilde{Q}^{(N)}$ is exactly the framed quiver considered in \cite{fang2023lusztigsheavesintegrablehighest} and \cite{zheng2014categorification}.

Given graded spaces $\bfW^{1}, \bfW^{2}, \cdots, \bfW^{N}$, with each $\bfW^{k}$ supported on $I^{k}$, the moduli space of the $N$-framed quiver is defined by $$\bfE_{\bfV, \bfW^{\bullet},\Omega} := \bigoplus\limits_{i \in I,1\leqslant k \leqslant N} \Hom(\bfV_i, \bfW^{k}_{i^{k}}) \oplus \bfE_{\bfV,\Omega}.$$ The algebraic group $G_{\bfV}$ acts canonically on $\bfE_{\bfV, \bfW^{\bullet},\Omega}$ and we consider the $G_{\bfV}$-equivariant derived category $\mathcal{D}^{b}_{G_{\bfV}}(\bfE_{\bfV, \bfW^{\bullet},\Omega})$. 

Assume that $I=\{i_{1},i_{2},\cdots,i_{t} \}$ and $d^{k}_{i_{l}}=\dim \bfW^{k}_{i^{k}_{l}}$ for each $l$ and $k$. We take $\boldsymbol{e}^{k}=(d^{k}_{i_{1}}i^{k}_{1},d^{k}_{i_{2}}i^{k}_{2} ,\cdots,d^{k}_{i_{t}}i^{k}_{t})$ as a flag type of $\bfW^{k}$. For any flag types $\boldsymbol{\nu}^{1},\boldsymbol{\nu}^{2}, \cdots, \boldsymbol{\nu}^{N} $, the flag type $\boldsymbol{\nu}^{1}\boldsymbol{e}^{1}\boldsymbol{\nu}^{2}\boldsymbol{e}^{2} \cdots \boldsymbol{\nu}^{N}\boldsymbol{e}^{N}$ is a flag type of $\bfV\oplus \bigoplus\limits_{1\leqslant k \leqslant N} \bfW^{k}$ for some $\bfV$. In particular, the sheaf $L_{\boldsymbol{\nu}^{1}\boldsymbol{e}^{1}\boldsymbol{\nu}^{2}\boldsymbol{e}^{2} \cdots \boldsymbol{\nu}^{N}\boldsymbol{e}^{N} }$ is a semisimple complex in $\mathcal{D}^{b}_{G_{\bfV}}(\bfE_{\bfV, \bfW^{\bullet},\Omega})$.

\begin{definition}
	We define the category $\mq_{\bfV, \bfW^{\bullet}}$ of Lusztig sheaves for $\tilde{Q}^{(N)}$ to be the full subcategory of $\mathcal{D}^{b}_{G_{\bfV}}(\bfE_{\bfV, \bfW^{\bullet},\Omega})$ consisting of finite direct sums of shifted simple direct summands of those $L_{\boldsymbol{\nu}^{1}\boldsymbol{e}^{1}\boldsymbol{\nu}^{2}\boldsymbol{e}^{2} \cdots \boldsymbol{\nu}^{N}\boldsymbol{e}^{N} }$, where the $\boldsymbol{\nu}^{k}$ run over all flag types.
	
\end{definition}

\subsection{Localizations}\label{local}
In this subsection, we introduce the localization of Lusztig's sheaves in \cite{fang2023lusztigsheavesintegrablehighest} and \cite{fang2023lusztigsheavestensorproducts}. Choose an orientation $\Omega_{i}$ such that $i\in I$ is a source in $\Omega_{i}$. Let $\bfE_{\bfV, \bfW^{\bullet}, i}^0$ be the open subset of $\bfE_{\bfV, \bfW^{\bullet},\Omega_{i}}$ defined by

 $$\bfE_{\bfV, \bfW^{\bullet}, i}^0=\{ x \in \bfE_{\bfV, \bfW^{\bullet},\Omega_{i}} | \dim \ker(\bigoplus\limits_{h \in \tilde{\Omega}^{(N)}_{i}, s(h) = i} x_h) = 0 \}.$$
 We also denote its complement by $\bfE_{\bfV, \bfW^{\bullet}, i}^{\geqslant 1}$. Let $\mathcal{N}_{\mathbf{V},\bfW^{\bullet},i}$ be the thick subcategory of $\mathcal{D}^{b}_{G_{\bfV}}(\bfE_{\bfV, \bfW^{\bullet},\Omega_{i}}),$ consisting of complexes supported on $\bfE_{\bfV, \bfW^{\bullet}, i}^{\geqslant 1}$. Since $\bfW^{\bullet}$ is fixed throughout, we write $\mathcal{N}_{\mathbf{V},i}$ for $\mathcal{N}_{\mathbf{V},\bfW^{\bullet},i}$ when there is no ambiguity.
 
 Recall that the Fourier-Deligne transform $\mathcal{F}_{\Omega_{i},\Omega}: \mathcal{D}^{b}_{G_{\bfV}}(\bfE_{\bfV, \bfW^{\bullet},\Omega_{i}}) \rightarrow \mathcal{D}^{b}_{G_{\bfV}}(\bfE_{\bfV, \bfW^{\bullet},\Omega})$ induces derived equivalence between different orientations. (See \cite[Section 2.2]{fang2023lusztigsheavesintegrablehighest}, \cite[Chapter 10]{Abook}, or \cite[Chapter 10]{lusztig2010introduction}.) We can consider the thick subcategory $\mathcal{F}_{\Omega_{i},\Omega}(\mathcal{N}_{\bfV,i})$ of $\mathcal{D}^{b}_{G_{\bfV}}(\bfE_{\bfV, \bfW^{\bullet},\Omega})$ for each $i$, and let $\mn_{\bfV}$ be the thick subcategory of $\mathcal{D}^{b}_{G_{\bfV}}(\bfE_{\bfV, \bfW^{\bullet},\Omega})$ generated by $\mathcal{F}_{\Omega_{i},\Omega}(\mathcal{N}_{\bfV,i}), i\in I$.
 
 \begin{definition}
 	Let $\mathcal{D}^{b}_{G_{\bfV}}(\bfE_{\bfV, \bfW^{\bullet},\Omega})/\mn_{\bfV}$ be the Verdier quotient of $\mathcal{D}^{b}_{G_{\bfV}}(\bfE_{\bfV, \bfW^{\bullet},\Omega})$ with respect to the thick subcategory $\mn_{\bfV}$, the localization $\mq_{\bfV, \bfW^{\bullet}}/\mn_{\bfV}$ is defined to be the full subcategory of $\mathcal{D}^{b}_{G_{\bfV}}(\bfE_{\bfV, \bfW^{\bullet},\Omega})/\mn_{\bfV}$, which consists of objects isomorphic to those of $\mq_{\bfV, \bfW^{\bullet}}$ in $\mathcal{D}^{b}_{G_{\bfV}}(\bfE_{\bfV, \bfW^{\bullet},\Omega})/\mn_{\bfV}$.
 \end{definition}
 The category $\mathcal{D}^{b}_{G_{\bfV}}(\bfE_{\bfV, \bfW^{\bullet},\Omega})/\mn_{\bfV}$ inherits a perverse $t$-structure from $\mathcal{D}^{b}_{G_{\bfV}}(\bfE_{\bfV, \bfW^{\bullet},\Omega})$, and the Verdier Duality $\mathbf{D}$ of $\mathcal{D}^{b}_{G_{\bfV}}(\bfE_{\bfV, \bfW^{\bullet},\Omega})$ also acts on $\mathcal{D}^{b}_{G_{\bfV}}(\bfE_{\bfV, \bfW^{\bullet},\Omega})/\mn_{\bfV}$ and $\mq_{\bfV, \bfW^{\bullet}}/\mn_{\bfV}$.
 
 If we denote the localization functor by $L: \mathcal{D}^{b}_{G_{\bfV}}(\bfE_{\bfV, \bfW^{\bullet},\Omega}) \rightarrow \mathcal{D}^{b}_{G_{\bfV}}(\bfE_{\bfV, \bfW^{\bullet},\Omega})/\mn_{\bfV}$, then $L$ restricts to an additive functor $L:\mq_{\bfV, \bfW^{\bullet}} \rightarrow \mq_{\bfV, \bfW^{\bullet}}/\mn_{\bfV}$.

\subsection{Functors of localizations}

In this subsection, we recall the definition of $\me^{(n)}_{i}$ and $\mf^{(n)}_{i}$ in \cite{fang2023lusztigsheavesintegrablehighest}.
\subsubsection{The functor $\me^{(n)}_i$}
For $\nu, \nu' \in \bbN[I]$ such that $\nu' = \nu - ni$, take graded spaces $\bfV, \bfV'$ of dimension vectors $\nu$ and $\nu'$ respectively. Define $\dot{\bfE}_{\bfV, \bfW^{\bullet}, i} = \bigoplus\limits_{h \in \Omega_{i}, s(h) \not= i} \Hom(\bfV_{s(h)}, \bfV_{t(h)}) \oplus \bigoplus\limits_{j \not= i, j \in I,1\leqslant k\leqslant N} \Hom(\bfV_j, \bfW^{k}_{j^{k}})$ and denote $\tilde{\nu_i} = \sum\limits_{h \in \Omega_{i},s(h)=i} \nu_{t(h)} + \sum\limits_{1\leqslant k \leqslant N} \dim \bfW^{k}_{i^{k}}$. Then we consider the following diagram
\[
\xymatrix{
	\mathbf{E}_{\mathbf{V},\mathbf{W}^{\bullet},\Omega_{i}}
	&
	& \mathbf{E}_{\mathbf{V}',\mathbf{W}^{\bullet},\Omega_{i}} \\
	\mathbf{E}^{0}_{\mathbf{V},\mathbf{W}^{\bullet},i} \ar[d]_{\phi_{\mathbf{V},i}} \ar[u]^{j_{\mathbf{V},i}}
	&
	& \mathbf{E}^{0}_{\mathbf{V}',\mathbf{W}^{\bullet},i} \ar[d]^{\phi_{\mathbf{V}',i}} \ar[u]_{j_{\mathbf{V}',i}} \\
	\dot{\mathbf{E}}_{\mathbf{V},\mathbf{W}^{\bullet},i} \times \mathbf{Gr}(\nu_i, \tilde{\nu}_{i})
	& \dot{\mathbf{E}}_{\mathbf{V},\mathbf{W}^{\bullet},i} \times \mathbf{Fl}(\nu'_{i},\nu_{i},\tilde{\nu}_{i}) \ar[r]^{q_{2}} \ar[l]_{q_{1}}
	& \dot{\mathbf{E}}_{\mathbf{V},\mathbf{W}^{\bullet},i} \times \mathbf{Gr}(\nu'_{i}, \tilde{\nu}_{i});
}
\]
where $\mathbf{Gr}(\nu_i, \tilde{\nu}_{i})$ is the Grassmannian consisting of $\nu_{i}$-dimensional subspaces of $\tilde{\nu}_{i}$-dimensional space $(\bigoplus\limits_{s(h)=i,h \in \Omega_{i}}\mathbf{V}_{t(h)})\oplus \bigoplus\limits_{1\leqslant k \leqslant N}\mathbf{W}^{k}_{i^{k}}$, and
\begin{equation*}
	\mathbf{Fl}(\nu_{i}-n,\nu_{i},\tilde{\nu}_{i})=\{ \mathbf{S}_{1}\subset \mathbf{S}_{2} \subset (\bigoplus\limits_{s(h)=i}\mathbf{V}_{t(h)}) \oplus \bigoplus\limits_{1\leqslant k \leqslant N}\mathbf{W}^{k}_{i^{k}} )|{\rm{dim}} \mathbf{S}_{1} = \nu_{i}-n, {\rm{dim}}\mathbf{S}_{2}=\nu_{i} \}.
\end{equation*}
is the flag variety. The morphisms are defined by
\begin{align*}
	\phi_{\mathbf{V},i}:\mathbf{E}^{0}_{\mathbf{V},\mathbf{W}^{\bullet},i} \longrightarrow \dot{\mathbf{E}}_{\mathbf{V},\mathbf{W}^{\bullet},i} \times \mathbf{Gr}(\nu_i, \tilde{\nu}_{i});
	x \mapsto (\dot{x}, {\rm{Im}} (\bigoplus \limits_{h \in \tilde{\Omega}_{i}, s(h)=i} x_{h} ) ),
\end{align*}where $\dot{x}=(x_h)_{s(h)\not=i},$
and $j_{\bfV,i},j_{\bfV',i}$ are open inclusions, $q_{1},q_{2}$ are natural projections
\begin{equation*}
	q_{1}(\dot{x}, \mathbf{S}_{1},\mathbf{S}_{2})=(\dot{x},\mathbf{S}_{2}),
\end{equation*}
\begin{equation*}
	q_{2}(\dot{x},\mathbf{S}_{1},\mathbf{S}_{2})=(\dot{x},\mathbf{S}_{1}).
\end{equation*}

\begin{definition}
 The functor $\tilde{\me}^{(n)}_{i}:\mathcal{D}^{b}_{G_{\bfV}}(\bfE_{\bfV, \bfW^{\bullet},\Omega_{i}}) \rightarrow \mathcal{D}^{b}_{G_{\bfV'}}(\bfE_{\bfV', \bfW^{\bullet},\Omega_{i}})$ is defined by	$$\tilde{\me}^{(n)}_i := (j_{\bfV', i})_! (\phi_{\bfV', i})^* (q_2)_! (q_1)^* (\phi_{\bfV, i})_{\flat} (j_{\bfV, i})^* [-n\nu_i],$$
 where $f_{\flat}$ denotes the inverse of $f^{\ast}$ for a principal bundle $f$. We also denote $\me^{(1)}_{i}$ by $\me_i$. For a general fixed orientation $\Omega$, the functor $\me^{(n)}_{i}:\mathcal{D}^{b}_{G_{\bfV}}(\bfE_{\bfV, \bfW^{\bullet},\Omega}) \rightarrow \mathcal{D}^{b}_{G_{\bfV'}}(\bfE_{\bfV', \bfW^{\bullet},\Omega})$ is defined by 
 $$\me^{(n)}_i = \mf_{\Omega_{i}, \Omega} \tilde{\me}^{(n)}_i \mf_{\Omega, \Omega_{i}}.$$
\end{definition}

By \cite[Proposition 3.18 and Corollary 3.19]{fang2023lusztigsheavesintegrablehighest}, the functor $\me^{(n)}_{i}$ induces a functor $$\me^{(n)}_{i}:\mathcal{D}^{b}_{G_{\bfV}}(\bfE_{\bfV, \bfW^{\bullet},\Omega})/\mn_{\bfV} \rightarrow \mathcal{D}^{b}_{G_{\bfV'}}(\bfE_{\bfV', \bfW^{\bullet},\Omega})/\mn_{\bfV'}, $$
and restricts to a functor
$\me^{(n)}_{i}:\mq_{\bfV, \bfW^{\bullet}}/\mn_{\bfV} \rightarrow \mq_{\bfV',\bfW^{\bullet}}/\mn_{\bfV'}. $ Observe that $\me^{(n)}_{i}$ is independent of the choice of $\Omega_{i}$ and $\Omega$ up to Fourier-Deligne transforms. (See \cite[Lemma 3.6]{fang2023lusztigsheavesintegrablehighest} for a proof.)

\subsubsection{The functor $\mf^{(n)}_{i}$}
For graded spaces $\mathbf{V}, \mathbf{V}''$ of dimension vectors $\nu,\nu''$ respectively, such that $ni+\nu''=\nu$.

Let $\mathbf{E}'_{\mathbf{V},\mathbf{W}^{\bullet},\Omega}$ be the variety consisting of $(x,\mathbf{S}, \rho)$, where $x \in \mathbf{E}_{\mathbf{V},\mathbf{W}^{\bullet},\Omega}$ and $\mathbf{S}$ is a subspace of $\mathbf{V}\oplus \bigoplus\limits_{1\leqslant k \leqslant N} \mathbf{W}^{k}$ such that $x(\mathbf{S})\subseteq \mathbf{S}$, and $\rho:\mathbf{S} \simeq \mathbf{V}''\oplus\bigoplus\limits_{1\leqslant k \leqslant N} \mathbf{W}^{k}$ is a linear isomorphism of graded spaces such that $\rho|_{\mathbf{W}^{k}}=id_{\mathbf{W}^{k}}$ for each $k$. Observe that it is equivalent to forget $\mathbf{W}^{k}$ and just take a linear isomorphism $\mathbf{S}\cap \mathbf{V} \simeq \mathbf{V}''$, hence we sometimes also say $\rho$ is a linear isomorphism between $\mathbf{S}$ and $\mathbf{V}''$. Let $\mathbf{E}''_{\mathbf{V},\mathbf{W}^{\bullet},\Omega}$ be the variety consisting of $(x,\mathbf{S})$ with the same conditions as above. Consider the following diagram
\begin{center}
	$ \mathbf{E}_{\mathbf{V}'',\mathbf{W}^{\bullet},\Omega}\xleftarrow{p_{1}} \mathbf{E}'_{\mathbf{V},\mathbf{W}^{\bullet},\Omega} \xrightarrow{p_{2}} \mathbf{E}''_{\mathbf{V},\mathbf{W}^{\bullet},\Omega} \xrightarrow{p_{3}} \mathbf{E}_{\mathbf{V},\mathbf{W}^{\bullet},\Omega},$
\end{center}
where the morphisms are defined by
\begin{equation*}
		p_{1}(x,\mathbf{S}, \rho)=(\rho_{\ast}(x|_{\mathbf{S}}));~
		p_{2}(x,\mathbf{S}, \rho) =(x,\mathbf{S});~
		p_{3}(x,\mathbf{S})=x.
\end{equation*}
Observe that $p_{1}$ is smooth with connected fiber, $p_{2}$ is a principal $G_{\bfV''}$ bundle and $p_{3}$ is proper.
\begin{definition}
	The functor $\mf^{(n)}_{i}:\mathcal{D}^{b}_{G_{\bfV''}}(\bfE_{\bfV'', \bfW,\Omega}) \rightarrow \mathcal{D}^{b}_{G_{\bfV}}(\bfE_{\bfV, \bfW^{\bullet},\Omega})$ is defined by 
	$$\mf^{(n)}_i = (p_{3})_{!}(p_{2})_{\flat}(p_{1})^{\ast}[d_{1}-d_{2}],$$
	where $d_{1}$ and $d_{2}$ are the relative dimensions of $p_{1}$ and $p_{2}$ respectively. We also denote $\mf^{(1)}_{i}$ by $\mf_i$.
\end{definition}

Observe that $\mf^{(n)}_i$ is isomorphic to Lusztig's induction functor $\mathbf{Ind}^{\mathbf{V}\oplus\bfW}_{ni,\mathbf{V''\oplus\bfW}}(\overline{\mathbb{Q}}_{l}\boxtimes -)$ (See \cite[Section 3]{fang2023lusztigsheavesintegrablehighest}), so 
\begin{equation}\label{ind}
	\mf^{(n)}_i(L_{\boldsymbol{\nu}^{1}\boldsymbol{e}^{1}\boldsymbol{\nu}^{2}\boldsymbol{e}^{2} \cdots \boldsymbol{\nu}^{N}\boldsymbol{e}^{N} })= L_{ni,\boldsymbol{\nu}^{1}\boldsymbol{e}^{1}\boldsymbol{\nu}^{2}\boldsymbol{e}^{2} \cdots \boldsymbol{\nu}^{N}\boldsymbol{e}^{N} }.
\end{equation}

By \cite[Lemma 3.7]{fang2023lusztigsheavesintegrablehighest}, the functor $\mf^{(n)}_{i}$ induces a functor $$\mf^{(n)}_{i}:\mathcal{D}^{b}_{G_{\bfV''}}(\bfE_{\bfV'', \bfW^{\bullet},\Omega})/\mn_{\bfV''} \rightarrow \mathcal{D}^{b}_{G_{\bfV}}(\bfE_{\bfV, \bfW^{\bullet},\Omega})/\mn_{\bfV}, $$
and restricts to a functor
$\mf^{(n)}_{i}:\mq_{\bfV'', \bfW^{\bullet}}/\mn_{\bfV''} \rightarrow \mq_{\bfV,\bfW^{\bullet}}/\mn_{\bfV}. $

\section{The category $\widetilde{\mq_{\bfV,\bfW^{\bullet}}/\mn_{\bfV}}$ and functors $\mathcal{E}^{(n)}_{\underline{i}},\mathcal{F}^{(n)}_{\underline{i}}$}

\subsection{Periodic functor}\label{periodic functor}

In this subsection, we review the definition of a periodic functor and refer to \cite[Chapter 11]{lusztig2010introduction} for further details. Let $o$ be a fixed positive integer.

	 Let $\cC$ be a $\overline{\bbQ}_l$-linear additive category. A periodic functor on $\cC$ is a linear functor $a^{\ast}:\cC\rightarrow \cC$ such that $(a^{*})^{o}$ is the identity functor on $\cC$.

\begin{definition}\label{3.1}
 Let $a^{\ast}$ be a periodic functor on $\cC$, we define the additive category $\tilde{\cC}$ as follows:\\
	$\bullet$ Its objects are pairs $(A,\varphi)$, where $A\in \cC$ and $\phi:a^*(A)\rightarrow A$ is an isomorphism in $\cC$ such that the composition
	$$A=a^{*o}(A)\xrightarrow{a^{*(o-1)(\varphi)}}a^{*(o-1)}(A)\rightarrow....\rightarrow a^*(A)\xrightarrow{\varphi}A$$
	is the identity morphism on $A$.\\
	$\bullet$ For any $(A,\varphi),(A',\varphi')\in \tilde{\cC}$, the morphism space
	$$\Hom_{\tilde{\cC}}((A,\varphi),(A',\varphi'))=\{f\in \Hom_{\cC}(A,A')|f\varphi=\varphi'(a^*(f))\}.$$ 
	$\bullet$ The direct sum of $(A,\varphi),(A',\varphi')\in \tilde{\cC}$ is defined naturally by $(A\oplus A',\varphi\oplus \varphi')$.
\end{definition}

\begin{definition}
 An object $(A,\phi)\in\tilde{\cC}$ is called traceless if there exists an object $B\in \cC$ and an integer $t\geqslant 2$ dividing $o$ such that $a^{*t}(B)\cong B$, $A\cong B\oplus a^*(B)\oplus...\oplus a^{*(t-1)}(B)$ and $\varphi:a^*(A)\rightarrow A$ corresponds to the isomorphism $a^*(B)\oplus a^{*2}(B)\oplus...\oplus a^{*t}(B)$ taking $a^{*s}(B)$ onto $a^*(a^{*s-1}(B))$ for $1\leqslant s\leqslant t-1$ and taking $a^{*t}(B)$ onto $B$, giving a permutation between the direct summands of $A$ and $a^*A$.
\end{definition} 

\begin{definition}
	We say $A$ and $B$ in $\tilde{\cC}$ are isomorphic modulo traceless objects if there exist traceless objects $C$ and $D$ such that $A \oplus C \cong B \oplus D.$
\end{definition}

\begin{lemma}[{\cite[Section 11.1.3]{lusztig2010introduction}}]\label{split criterion}
	Let $(A,\varphi),(A',\varphi'),(A'',\varphi'')$ be objects in $\tilde{\cC}$ and $i':(A',\varphi')\rightarrow (A,\varphi),\ p'':(A,\varphi)\rightarrow (A'',\varphi'')$ be morphisms in $\tilde{\cC}$, if there exist morphisms $i'':A''\rightarrow A,\ p':A\rightarrow A'$ in $\cC$ such that $$p'i'=1_{A'},\ p'i''=0,\ p''i'=0,\ p''i''=1_{A''},\ i'p'+i''p''=1_A,$$
 then $(A,\varphi)\cong(A',\varphi')\oplus(A'',\varphi'')$ in $\tilde{\cC}$.
\end{lemma}

\subsection{Localization and its functors for a quiver with automorphism}

 For a given generalized Cartan matrix $C=DB$, we construct, using \cite[Proposition 14.1.2]{lusztig2010introduction}, a finite quiver $Q=(I,H,\Omega)$ with an admissible automorphism $a$. Let $o$ be the order of $a$. The automorphism $a$ naturally induces an admissible automorphism of the $N$-framed quiver $\tilde{Q}^{(N)}$. We denote the set of $a$-orbits of $I$ by $\underline{I}=I/\langle a \rangle$, and denote the set of $a$-invariant dimension vectors by $\bbN I^a$. Then any $\underline{i} \in \underline{I}$ determines a dimension vector $\underline{i}= \sum\limits_{i \in \underline{i}}i$ in $\bbN I^a$. We also denote the $k$-th copy of $\underline{i}$ in $\tilde{Q}^{(N)}$ by $\underline{i}^{k}$.

Let $\mathcal{S}'$ be the set of finite sequences $\boldsymbol{\nu}=(\nu^{1},\nu^{2},\cdots,\nu^{s})$ of dimension vectors such that each $\nu^{l}=a_{l}\underline{i}_{l}$ for some $a_{l} \in \bbN_{\geqslant 1}$ and $\underline{i}_{l} \in \underline{I} \subset \bbN I^a$. For $\nu \in \bbN I^a$, we say $\boldsymbol{\nu} \in \mathcal{S}'$ is an $a$-flag type of $\nu$ if $\sum\limits_{1\leqslant l \leqslant s} \nu^{l} =\nu$. We may also define the semisimple complex $L_{\boldsymbol{\nu}}$ for $\boldsymbol{\nu} \in \mathcal{S}'$ as in Section 2.1.

Given a $I^{k}$-graded space $\bfW^{k}$ with dimension vectors $\omega^{k} \in (\bbN I^{k})^{a}$ for each $k$, we assume $\omega^{k}=\sum\limits_{\underline{i} \in \underline{I}} e^{k}_{\underline{i}}\underline{i}^{k}$. Assume that $I'=\{\underline{i}_{1},\underline{i}_{2},\cdots, \underline{i}_{m} \}$. We fix an $a$-flag type $\boldsymbol{e}^{k}=(e^{k}_{\underline{i}_{1}}\underline{i}^{k}_{1},e^{k}_{\underline{i}_{2}}\underline{i}^{k}_{2},\cdots,e^{k}_{\underline{i}_{m}}\underline{i}^{k}_{m})$. For any flag types $\boldsymbol{\nu}^{1},\boldsymbol{\nu}^{2},\cdots,\boldsymbol{\nu}^{N} $, there is an induced flag type $\boldsymbol{\nu}^{1}\boldsymbol{e}^{1}\boldsymbol{\nu}^{2}\boldsymbol{e}^{2}\cdots \boldsymbol{\nu}^{N}\boldsymbol{e}^{N} $ of $\bfV \oplus\bigoplus\limits_{1\leqslant k \leqslant N}\bf W^{k}$ for some $\bfV$. Moreover, if these $\boldsymbol{\nu}^{k}$ are $a$-flag types, so is $\boldsymbol{\nu}^{1}\boldsymbol{e}^{1}\boldsymbol{\nu}^{2}\boldsymbol{e}^{2}\cdots \boldsymbol{\nu}^{N}\boldsymbol{e}^{N} $. We also denote $L_{\boldsymbol{\nu}^{1}\boldsymbol{e}^{1}\boldsymbol{\nu}^{2}\boldsymbol{e}^{2}\cdots \boldsymbol{\nu}^{N}\boldsymbol{e}^{N}}$ by $L_{\boldsymbol{\nu}^{1},\boldsymbol{\nu}^{2},\cdots,\boldsymbol{\nu}^{N} }$.

\begin{definition}
	Let $\bfV$ be an $I$-graded space with dimension vector $\nu \in \bbN I$.\\
	\rm{(1)} Let $\mathcal{P}_{\bfV,\bfW^{\bullet}}$ be the set consisting of those simple perverse sheaves $L$ in $\mathcal{D}^{b}_{G_{\bfV}}(\bfE_{\bfV,\bfW^{\bullet},\Omega})$ such that $L$ is a direct summand (up to shifts) of $L_{\boldsymbol{\nu}^{1},\boldsymbol{\nu}^{2},\cdots,\boldsymbol{\nu}^{N}}$ for some flag type $\boldsymbol{\nu}$ of $\bfV$. Let $\mq_{\bfV,\bfW^{\bullet}}$ be the full subcategory of $\mathcal{D}^{b}_{G_{\bfV}}(\bfE_{\bfV,\bfW^{\bullet},\Omega})$, which consists of
	finite direct sums of shifted simple perverse sheaves in $\mathcal{P}_{\bfV,\bfW^{\bullet}}$.\\
	\rm{(2)} The localization $\mq_{\bfV, \bfW^{\bullet}}/\mn_{\bfV}$ is defined to be the full subcategory of $\mathcal{D}^{b}_{G_{\bfV}}(\bfE_{\bfV, \bfW^{\bullet},\Omega})/\mn_{\bfV}$, which consists of objects isomorphic to those of $\mq_{\bfV, \bfW^{\bullet}}$ in $\mathcal{D}^{b}_{G_{\bfV}}(\bfE_{\bfV, \bfW^{\bullet},\Omega})/\mn_{\bfV}$.
\end{definition}

\begin{definition}\label{E_i}
	For $\underline{i}\in\underline{I}$ and $I$-graded spaces $\bfV,\bfV'$ with dimension vectors $\nu,\nu'=\nu-n\underline{i}$, the functor $\me^{(n)}_{\underline{i}}:\mathcal{D}^{b}_{G_{\bfV}}(\bfE_{\bfV, \bfW^{\bullet},\Omega})/\mn_{\bfV} \rightarrow \mathcal{D}^{b}_{G_{\bfV'}}(\bfE_{\bfV', \bfW^{\bullet},\Omega})/\mn_{\bfV'}$ is defined by 
	$$ \me^{(n)}_{\underline{i}}=\prod\limits_{i \in \underline{i}} \me^{(n)}_i. $$ 
\end{definition}

\begin{definition}\label{F_i}
	For $\underline{i}\in\underline{I}$ and $I$-graded spaces $\bfV,\bfV''$ with dimension vectors $\nu,\nu''=\nu-n\underline{i}$, the functor $\mf^{(n)}_{\underline{i}}:\mathcal{D}^{b}_{G_{\bfV''}}(\bfE_{\bfV'', \bfW^{\bullet},\Omega})/\mn_{\bfV''} \rightarrow \mathcal{D}^{b}_{G_{\bfV}}(\bfE_{\bfV, \bfW^{\bullet},\Omega})/\mn_{\bfV}$ is defined by 
	$$ \mf^{(n)}_{\underline{i}}=\prod\limits_{i \in \underline{i}} \mf^{(n)}_i. $$ 
\end{definition}

Since the automorphism $a$ is admissible, there are no arrows between vertices in $\underline{i}$; hence the functors $\me^{(n)}_{i}$ (or $\mf^{(n)}_{i}$), $i\in \underline{i}$ commute with each other. Thus the above compositions of functors are well-defined.

\begin{proposition}
	The functors $\mf^{(n)}_{\underline{i}}$ and $\me^{(n)}_{\underline{i}}$ restrict to functors between $\mq_{\bfV, \bfW^{\bullet}}/\mn_{\bfV}$ and $\mq_{\bfV, \bfW^{\bullet}}/\mn_{\bfV}$ for $\bfV$, $\bfV''$ with dimension vectors $\nu,\nu''=\nu-n\underline{i}$, 
	$$\me^{(n)}_{\underline{i}}:\mq_{\bfV, \bfW^{\bullet}}/\mn_{\bfV} \rightarrow \mq_{\bfV', \bfW^{\bullet}}/\mn_{\bfV'}, $$
	$$\mf^{(n)}_{\underline{i}}:\mq_{\bfV'', \bfW^{\bullet}}/\mn_{\bfV''} \rightarrow \mq_{\bfV, \bfW^{\bullet}}/\mn_{\bfV}. $$
\end{proposition}
\begin{proof}
	It suffices to show that for any $\boldsymbol{\nu}^{k} \in \mathcal{S}$, the functors $\me^{(n)}_{\underline{i}}$ and $\mf^{(n)}_{\underline{i}}$ send $L_{\boldsymbol{\nu}^{1},\boldsymbol{\nu}^{2},\cdots,\boldsymbol{\nu}^{N}}$ to a finite direct sum of those shifted $L_{\boldsymbol{\nu'}^{1},\boldsymbol{\nu'}^{2},\cdots,\boldsymbol{\nu'}^{N}}$ with $\boldsymbol{\nu'}^{k} \in \mathcal{S}$.

	By Equation (\ref{ind}), it follows that 
	$ \mf^{(n)}_{\underline{i}}L_{\boldsymbol{\nu}^{1},\boldsymbol{\nu}^{2},\cdots,\boldsymbol{\nu}^{N}} =L_{n\underline{i},\boldsymbol{\nu}^{1},\boldsymbol{\nu}^{2},\cdots,\boldsymbol{\nu}^{N}}, $ hence $\mf^{(n)}_{\underline{i}}:\mq_{\bfV'', \bfW^{\bullet}}/\mn_{\bfV''} \rightarrow \mq_{\bfV, \bfW^{\bullet}}/\mn_{\bfV}$ is well-defined.
	
	 By \cite[Corollary 3.20]{fang2023lusztigsheavesintegrablehighest}, $\me^{(n)}_{\underline{i}}$ commutes with $\mf_{\underline{j}}$ if $\underline{i}\neq \underline{j}$. Otherwise, $\me^{(n)}_{\underline{i}}\mf_{\underline{i}}$ and $\mf_{\underline{i}}\me^{(n)}_{\underline{i}}$ differ by a direct sum of some shifts of $\mathbf{Id}$. Regarding the vertices $I^{k}$ as unframed vertices, it also follows that $\me^{(n)}_{\underline{i}}$ commutes with the left multiplication by $L_{\boldsymbol{e}^{k}}$.
	 By induction on the lengths of the sequences $\boldsymbol{\nu}^{k}$, we prove that $\me^{(n)}_{\underline{i}}:\mq_{\bfV, \bfW^{\bullet}}/\mn_{\bfV} \rightarrow \mq_{\bfV'', \bfW^{\bullet}}/\mn_{\bfV''} $ is well-defined.
\end{proof}

Observe that when $a(\nu)=\nu$, the automorphism $a$ induces a natural map $a: \bfE_{\bfV,\bfW^{\bullet},\Omega}\rightarrow \bfE_{\bfV,\bfW^{\bullet},\Omega}$, then $a^{\ast}$ is a periodic functor on $\mathcal{D}^{b}_{G_{\bfV}}(\bfE_{\bfV, \bfW^{\bullet},\Omega})$. Since $a^{\ast}(\mn_{\bfV,i}) \subseteq\mn_{\bfV,a^{-1}(i)}$, $a^{\ast}$ preserves $\mn_{\bfV}$ and induces a periodic functor on $\mathcal{D}^{b}_{G_{\bfV}}(\bfE_{\bfV, \bfW^{\bullet},\Omega})/\mn_{\bfV}$. Since $a^{\ast}$ also preserves $\mq_{\bfV, \bfW^{\bullet}}$, it restricts to a periodic functor on $\mq_{\bfV, \bfW^{\bullet}}/\mn_{\bfV}$. 

When $\nu\not=a(\nu)$, we have $a^{\ast}\mathcal{D}^{b}_{G_{\bfV}}(\bfE_{\bfV, \bfW^{\bullet},\Omega})=\mathcal{D}^{b}_{G_{a^{-1}(\bfV})}(\bfE_{a^{-1}(\bfV), \bfW^{\bullet},\Omega})$ and $a^{\ast}\mn_{\bfV}=\mn_{a^{-1}\bfV}$(here $a$ acts on $\bfV$ as a permutation of $\bfV_i$), thus $a^{\ast}$ is a periodic functor on $\bigoplus_{\bfV}\mq_{\bfV, \bfW^{\bullet}}/\mn_{\bfV}$. 
\begin{definition}
	We define $\widetilde{\bigoplus_{\bfV}\mq_{\bfV, \bfW^{\bullet}}/\mn_{\bfV}}$ to be the additive category obtained by applying Definition \ref{3.1} to the periodic functor $a^{\ast}$ on $\bigoplus_{\bfV}\mq_{\bfV, \bfW^{\bullet}}/\mn_{\bfV}$.  The Verdier duality is defined by $\mathbf{D}(L,\phi) =(\mathbf{D}L,\mathbf{D}(\phi)^{-1}).$
\end{definition}

 Since $\me^{(n)}_i a^{\ast} \cong a^{\ast} \me^{(n)}_{a^{-1}(i)}$, we know that $\me^{(n)}_{\underline{i}} a^{\ast} \cong a^{\ast} \me^{(n)}_{\underline{i}}.$ Similarly, $\mf^{(n)}_{\underline{i}} a^{\ast} \cong a^{\ast} \mf^{(n)}_{\underline{i}}.$ 
For $\phi: a^{\ast} L \rightarrow L$, we denote the composition $a^{\ast}\me^{(n)}_{\underline{i}}L \cong \me^{(n)}_{\underline{i}}a^{\ast}L \xrightarrow{\me^{(n)}_{\underline{i}}(\phi)} \me^{(n)}_{\underline{i}}L $ by $\me^{(n)}_{\underline{i}}\phi$, and denote the composition $a^{\ast}\mf^{(n)}_{\underline{i}}L \cong \mf^{(n)}_{\underline{i}}a^{\ast}L \xrightarrow{\mf^{(n)}_{\underline{i}}(\phi)} \mf^{(n)}_{\underline{i}}L $ by $\mf^{(n)}_{\underline{i}}\phi$.

\begin{definition}
	The functors $\me^{(n)}_{\underline{i}},\mf^{(n)}_{\underline{i}}: \widetilde{\bigoplus\limits_{\bfV}\mathcal{Q}_{\mathbf{V},\mathbf{W}^{\bullet}}/\mathcal{N}_{\mathbf{V}}} \rightarrow \widetilde{\bigoplus\limits_{\bfV}\mathcal{Q}_{\mathbf{V},\mathbf{W}^{\bullet}}/\mathcal{N}_{\mathbf{V}}}$ are defined respectively by
	$$\me^{(n)}_{\underline{i}}((L,\phi))=(\me^{(n)}_{\underline{i}}L,\me^{(n)}_{\underline{i}}\phi), $$
	$$ \mf^{(n)}_{\underline{i}}((L,\phi))=(\mf^{(n)}_{\underline{i}}L,\mf^{(n)}_{\underline{i}}\phi). $$
	For $\bfV$ with $a(\nu)=\nu$, the functor $\mk_{\underline{i}},\mk_{-\underline{i}}:\widetilde{\mathcal{Q}_{\mathbf{V},\mathbf{W}^{\bullet}}/\mathcal{N}_{\mathbf{V}}} \rightarrow \widetilde{\mathcal{Q}_{\mathbf{V},\mathbf{W}^{\bullet}}/\mathcal{N}_{\mathbf{V}}}$ are defined by $$\mk_{\underline{i}} =\mathbf{Id} [s_{i}(\tilde{\nu}_{i}-2\nu_{i})],~ \mk_{-\underline{i}} =\mathbf{Id} [-s_{i}(\tilde{\nu}_{i}-2\nu_{i})],$$
	where $s_{i}$ is the order of $\underline{i}$ and $\tilde{\nu_i} = \sum\limits_{h \in \Omega_{i},s(h)=i} \nu_{t(h)} + \sum\limits_{1\leqslant k \leqslant N}\dim \bfW^{k}_{i^{k}}$ as before.
\end{definition}

\subsection{Commutation relations of functors}

In this subsection, we study the commutation relations of our functors, and we assume throughout that the dimension vectors are $a$-stable. 
\subsubsection{Relations of divided powers}

Let $\underline{i}$ be the $a$-orbit of $i \in I$, let $s_i$ be the order of $\underline{i}$, and choose an orientation $\Omega$ such that every vertex in $\underline{i}$ is a source in $\Omega$.
\begin{proposition}\label{dp}
	For any $(L, \phi)$ in $\widetilde{\mathcal{Q}_{\mathbf{V},\mathbf{W}^{\bullet}}/\mathcal{N}_{\mathbf{V}}}$, there are isomorphisms modulo traceless objects
	$$
		\me_{\underline{i}}^{(n-1)} \me_{\underline{i}}(L, \phi)\cong \me_{\underline{i}} \me_{\underline{i}}^{(n-1)}(L, \phi) \cong \bigoplus_{r=0}^{n-1} \me_{\underline{i}}^{(n)}[-2s_i r + s_i(n-1)](L, \phi),
	$$
	$$\mf_{\underline{i}}^{(n-1)} \mf_{\underline{i}}(L, \phi) \cong
	\mf_{\underline{i}} \mf_{\underline{i}}^{(n-1)}(L, \phi) \cong \bigoplus_{r=0}^{n-1} \mf_{\underline{i}}^{(n)}[-2s_i r + s_i(n-1)](L, \phi),
	$$
	where $\me_{\underline{i}}^{(0)}$ and $\mf_{\underline{i}}^{(0)}$ are defined to be the identity functor $\mathbf{Id}$.
\end{proposition}

\begin{proof}
	Since the functor $\mf^{(n)}_{i}$ is isomorphic to Lusztig's induction functor, the second relation follows from the proof of \cite[Lemma 12.3.4]{lusztig2010introduction}. It remains to prove the first relation.

	Let $\bfV, \bfV', \bfV''$ have dimensions $\nu$, $\nu' = \nu - (n-1)\underline{i}$, and $\nu'' = \nu - n\underline{i}$, respectively. Let $\bfE_{\bfV,\bfW^{\bullet},\underline{i}}^0$ be the open subset of $\bfE_{\bfV,\bfW^{\bullet},\Omega}$ where $x$ satisfies $\dim \ker(\bigoplus\limits_{h \in \tilde{\Omega}^{(N)}, s(h) = i} x_h) = 0$ for any $i \in \underline{i}$. We also define
	$$
	\dot{\bfE}_{\bfV,\bfW^{\bullet},\underline{i}} = \bigoplus_{h \in \Omega, s(h) \notin \underline{i}} \Hom(\bfV_{s(h)}, \bfV_{t(h)}) \oplus \bigoplus_{j \notin \underline{i}, j \in I,1\leqslant k\leqslant N} \Hom(\bfV_j, \bfW^{k}_{j^{k}}),
	$$
	and denote
	$
	\tilde{\nu}_i = \sum_{h \in H, s(h) = i} \nu_{t(h)} +\sum\limits_{1\leqslant k \leqslant N} \dim \bfW^{k}_{i^{k}}.
	$
	We consider the following diagram
	\[
	\xymatrix{
		\bfE_{\bfV, \bfW^{\bullet},\Omega} \ar@{<-}[d]_{j_{\bfV,\underline{i}}} & & \bfE_{\bfV'', \bfW^{\bullet},\Omega} \ar@{<-}[d]_{j_{\bfV'',\underline{i}}}\\
		\bfE_{\bfV, \bfW^{\bullet}, \underline{i}}^0 \ar@{->}[d]_{\phi_{\bfV,\underline{i}}} & & \bfE_{\bfV'', \bfW^{\bullet}, \underline{i}}^0 \ar@{->}[d]_{\phi_{\bfV'',\underline{i}}} \\
		\dot{\bfE}_{\bfV, \bfW^{\bullet}, \underline{i}} \times \mathbf{Gr}^{s_i}(\nu_i, \tilde{\nu}_i) \ar@{<-}[r]_{q_1} & 
		\dot{\bfE}_{\bfV, \bfW^{\bullet}, \underline{i}} \times \mathbf{Fl}^{s_i}(\nu_i - n, \nu_i, \tilde{\nu}_i) \ar@{->}[r]_{q_2} & 
		\dot{\bfE}_{\bfV, \bfW^{\bullet}, \underline{i}} \times \mathbf{Gr}^{s_i}(\nu_i - n, \tilde{\nu}_i),
	}
	\]
	where the maps $j_{\bfV,\underline{i}}$ are open inclusions, and $\phi_{\bfV,\underline{i}}$ is defined by
	\begin{align*}
		\phi_{\mathbf{V},\underline{i}}:\mathbf{E}^{0}_{\mathbf{V},\mathbf{W}^{\bullet},\underline{i}} \longrightarrow \dot{\mathbf{E}}_{\mathbf{V},\mathbf{W}^{\bullet},\underline{i}} \times \mathbf{Gr}^{s_{i}}(\nu_i, \tilde{\nu}_{i});
		x \mapsto (\dot{x}, ({\rm{Im}} (\bigoplus \limits_{h \in \tilde{\Omega}, s(h)=i} x_{h} )_{i \in \underline{i}} ) ),
	\end{align*}
	$q_{1},q_{2}$ are the natural projections, which are smooth and proper. Then the functor $\me_{\underline{i}}^{(n)}$ is isomorphic to $(j_{\bfV'', \underline{i}})_! (\phi_{\bfV'', \underline{i}})^* (q_2)_! (q_1)^* (\phi_{\bfV, \underline{i}})_{\flat} (j_{\bfV, \underline{i}})^* [-ns_{i}\nu_i]$.
	
	Similarly, the functor $\me_{\underline{i}}^{(n-1)}\me_{\underline{i}}$ is isomorphic to $$(j_{\bfV'', \underline{i}})_! (\phi_{\bfV'', \underline{i}})^* (q''_2)_! (q''_1)^* (\phi_{\bfV', \underline{i}})_{\flat} (j_{\bfV', \underline{i}})^*(j_{\bfV', \underline{i}})_! (\phi_{\bfV', \underline{i}})^* (q'_2)_! (q'_1)^* (\phi_{\bfV, \underline{i}})_{\flat} (j_{\bfV, \underline{i}})^* [-ns_{i}\nu_i+s_{i}(n-1)],$$ where the morphisms are defined in the following diagram:
	\[
	\xymatrix{
		\bfE_{\bfV, \mathbf{W}^{\bullet},\Omega} \ar@{<-}[d]_{j_{\bfV,\underline{i}}} & & \bfE_{\bfV', \mathbf{W}^{\bullet},\Omega} \ar@{<-}[d]_{j_{\bfV',\underline{i}}} & & \bfE_{\bfV'', \mathbf{W}^{\bullet},\Omega} \ar@{<-}[d]_{j_{\bfV'',\underline{i}}} \\
		\bfE^0_{\bfV, \mathbf{W}^{\bullet}, \underline{i}} \ar@{->}[d]_{\phi_{\bfV,\underline{i}}} & & \bfE^0_{\bfV', \mathbf{W}^{\bullet}, \underline{i}} \ar@{->}[d]_{\phi_{\bfV',\underline{i}}} & & \bfE^0_{\bfV'', \mathbf{W}^{\bullet}, \underline{i}} \ar@{->}[d]_{\phi_{\bfV'',\underline{i}}} \\
		\txt{$\dot{\bfE}_{\bfV, \mathbf{W}^{\bullet}, \underline{i}}$\\ $\times$\\ $\mathbf{Gr}^{s_i}(\nu_i, \tilde{\nu}_i)$} 
		\ar@{<-}[r]^-{q_1'} & 
		\txt{$\dot{\bfE}_{\bfV, \mathbf{W}^{\bullet}, \underline{i}}$\\ $\times$\\ $\mathbf{Fl}^{s_i}(\nu'_i, \nu_i, \tilde{\nu}_i)$} 
		\ar@{->}[r]^-{q_2'} & 
		\txt{$\dot{\bfE}_{\bfV, \mathbf{W}^{\bullet}, \underline{i}}$\\ $\times$\\ $\mathbf{Gr}^{s_i}(\nu'_i, \tilde{\nu}_i)$} 
		\ar@{<-}[r]^-{q_1''} & 
		\txt{$\dot{\bfE}_{\bfV, \mathbf{W}^{\bullet}, \underline{i}}$\\ $\times$\\ $\mathbf{Fl}^{s_i}(\nu''_i, \nu'_i , \tilde{\nu}_i)$} 
		\ar@{->}[r]^-{q_2''} & 
		\txt{$\dot{\bfE}_{\bfV, \mathbf{W}^{\bullet}, \underline{i}}$\\ $\times$\\ $\mathbf{Gr}^{s_i}(\nu''_i , \tilde{\nu}_i)$}.
	}
	\]
	Since $(\phi_{\bfV', \underline{i}})_{\flat} (j_{\bfV', \underline{i}})^*(j_{\bfV', i})_! (\phi_{\bfV', i})^* \cong \mathbf{Id}$, we can reduce to the following diagram
	\[
	\xymatrix{
			{\dot{\bfE}_{\bfV, \mathbf{W}^{\bullet}, \underline{i}} \times \mathbf{Gr}^{s_i}(\nu'_i, \tilde{\nu}_i)} 
			& {\dot{\bfE}_{\bfV, \mathbf{W}^{\bullet}, \underline{i}} \times \mathbf{Fl}^{s_i}(\nu''_i, \nu'_i, \tilde{\nu}_i)} 
			\ar[l]_{q_1''} \ar[r]^{q_2''} 
			& {\dot{\bfE}_{\bfV, \mathbf{W}^{\bullet}, \underline{i}} \times \mathbf{Gr}^{s_i}(\nu''_i, \tilde{\nu}_i)} 
			\\
			{\dot{\bfE}_{\bfV, \mathbf{W}^{\bullet}, \underline{i}} \times \mathbf{Fl}^{s_i}(\nu'_i, \nu_i, \tilde{\nu}_i)} 
			\ar[u]^{q_2'} \ar[d]_{q_1'}
			& {\dot{\bfE}_{\bfV, \mathbf{W}^{\bullet}, \underline{i}} \times \mathbf{Fl}^{s_i}(\nu''_i, \nu'_i, \nu_i, \tilde{\nu}_i)} 
			\ar[u]^{r_1} \ar[l]_{r_2} \ar[rd]^{\pi} 
			& \\
			{\dot{\bfE}_{\bfV, \mathbf{W}^{\bullet}, \underline{i}} \times \mathbf{Gr}^{s_i}(\nu_i, \tilde{\nu}_i)} 
			& 
			& {\dot{\bfE}_{\bfV, \mathbf{W}^{\bullet}, \underline{i}} \times \mathbf{Fl}^{s_i}(\nu''_i, \nu_i, \tilde{\nu}_i)}. 
			\ar[uu]_{q_2} \ar[ll]^{q_1}
		}
	\]
	Observe that $\pi$ is a locally trivial fibration whose fiber is the $s_i$-fold product of $\bbP^{n-1}$. Thus $(q_2'')_! (q_1'')^* (q_2')_! (q_1')^* \cong (q_2)_! \pi_! \pi^* (q_1)^* \cong (q_2)_! (\pi_! \overline{\bbQ}_{l} \otimes (q_1)^* (-)) \cong (q_2)_! ( \mathbf{H}^{\ast}( (\bbP^{n-1})^{s_{i}} ) \otimes (q_1)^* (-))$. 
	
	We can identify $\mathbf{H}^{\ast}( (\bbP^{n-1})^{s_{i}} ) $ with $\mathbf{H}^{\ast}( (\bbP^{n-1}) )^{\otimes s_{i}} $ via the K\"unneth formula and let $\overline{\bbQ}_{l}[-2r_1 - 2r_2 - \cdots - 2r_{s_i}] \subseteqq \mathbf{H}^{\ast}( (\bbP^{n-1})^{s_{i}} ) $ be the direct summand contributed by $\mathbf{H}^{2r_{1}}( \bbP^{n-1} )[-2r_{1}] \otimes \mathbf{H}^{2r_{2}}( \bbP^{n-1} )[-2r_{2}] \cdots \otimes \mathbf{H}^{2r_{s_{i}}}( \bbP^{n-1} )[-2r_{s_{i}}]. $ Then for a morphism $\phi: a^{\ast}L \rightarrow L$, the morphism
	$$
	(q_2)_! (\pi_! \overline{\bbQ}_{l} \otimes (q_1)^* \phi): (q_2)_! (\pi_! \overline{\bbQ}_{l} \otimes (q_1)^* a^* L) \rightarrow (q_2)_! (\pi_! \overline{\bbQ}_{l} \otimes (q_1)^* L)
	$$
	sends $(q_2)_!(\overline{\bbQ}_{l}[-2r_1 - 2r_2 - \cdots - 2r_{s_i}] \otimes (q_1)^*L)$ to $(q_2)_!(\overline{\bbQ}_{l}[-2r_2 - 2r_3 - \cdots - 2r_{s_i} - 2r_1] \otimes (q_1)^*L)$, where $r_k \in \{0, 1, 2, \dots, n-1\}$. If $r_{s'}\neq r_{s''}$ for some $s',s''$, $\phi$ acts by a cyclic permutation on the direct summands corresponding to $\{r_{k}, 0\leqslant k \leqslant s_{i}-1 \}$, hence the direct sum of these direct summands contributes to a traceless object. As a result, only those $\overline{\bbQ}_{l}[-2r_{s_{i}]}$ with $0\leqslant r \leqslant n-1$ contribute to the non-traceless part. This proves the assertion for $\me_{\underline{i}}^{(n-1)}\me_{\underline{i}}$. The proof for $\me_{\underline{i}}\me_{\underline{i}}^{(n-1)}$ is similar.

\end{proof}

\subsubsection{Relations of $\me_{\underline{i}}^{(n)}$ and $\mf_{\underline{j}}$}
Let $\underline{i}$ and $\underline{j}$ be the $a$-orbits of $i,j \in I$, respectively, and let $s_i$ be the order of $\underline{i}$.
\begin{proposition}\label{dq1}
	For any $(L, \phi)$ in $\widetilde{\mathcal{Q}_{\mathbf{V},\mathbf{W}^{\bullet}}/\mathcal{N}_{\mathbf{V}}}$, if $\underline{i} \neq \underline{j}$, there is an isomorphism
  \begin{equation}\label{c1}
  	 \me_{\underline{i}}^{(n)} \mf_{\underline{j}}(L, \phi) \cong \mf_{\underline{j}} \me_{\underline{i}}^{(n)}(L, \phi).
  \end{equation}
  If $\underline{i}=\underline{j}$, set $m = n + \tilde{\nu_i} - 2\nu_i - 1$ and there is an isomorphism modulo traceless objects
  \begin{equation}
  	\begin{split}
  			 &\me_{\underline{i}}^{(n)} \mf_{\underline{i}}(L, \phi) \oplus \bigoplus_{k=0}^{-m-1} \me_{\underline{i}}^{(n-1)}[(-m-1-2k) s_i](L, \phi)\\
  		\cong &\mf_{\underline{i}} \me_{\underline{i}}^{(n)}(L, \phi) \oplus \bigoplus_{k=0}^{m-1} \me_{\underline{i}}^{(n-1)}[(m-1-2k) s_i](L, \phi).
  	\end{split}
  \end{equation}
 .
\end{proposition}

\begin{proof}
 If $\underline{i} \neq \underline{j}$, by \cite[Lemma 3.14 and Corollary 3.20]{fang2023lusztigsheavesintegrablehighest} the functors $\me_{i}^{(n)} $ and $\mf_{j}$ commute with each other on those $\mathcal{D}^{b}_{G_{\bfV}}(\bfE_{\bfV, \bfW^{\bullet},\Omega})/\mn_{\bfV}$, so their products $\me_{\underline{i}}^{(n)} $ and $\mf_{\underline{j}}$ also commute on those $\mathcal{D}^{b}_{G_{\bfV}}(\bfE_{\bfV, \bfW^{\bullet},\Omega})/\mn_{\bfV}$. Hence they induce commuting functors on $\widetilde{\mathcal{Q}_{\mathbf{V},\mathbf{W}^{\bullet}}/\mathcal{N}_{\mathbf{V}}}$. The equation (\ref{c1}) is proved.

Now, for $\underline{i} = \underline{j}$, we consider $\nu''', \nu', \nu'' \in \bbN[I]^{a}$, and take $\bfV''', \bfV', \bfV''$ of dimensions $\nu''' = \nu + \underline{i}, \nu' = \nu - (n-1)\underline{i}, \nu'' = \nu - n\underline{i}$, respectively. Choose an orientation $\Omega$ such that all vertices $i \in \underline{i}$ are sources in $\Omega$. Then we can draw the following diagrams for $\me_{\underline{i}}^{(n)}\mf_{\underline{i}}$,
\[
\xymatrix{
\bfE_{\bfV, \bfW^{\bullet},\Omega} & \bfE'_{\bfV''', \bfW^{\bullet},\Omega} \ar[l]_{p_1} \ar[r]^{p_2} & \bfE''_{\bfV''', \bfW^{\bullet},\Omega} \ar[r]^{p_3} & \bfE_{\bfV''', \bfW^{\bullet},\Omega} \\
\bfE^{0}_{\bfV, \bfW^{\bullet}, \underline{i}} \ar[u]^{j_{\bfV, \underline{i}}} \ar[d]_{\phi_{\bfV, \underline{i}}} & \bfE^{0,'}_{\bfV''', \bfW^{\bullet}, \underline{i}} \ar[l] \ar[u] \ar[r] \ar[d] & \bfE^{0,''}_{\bfV''', \bfW^{\bullet}, \underline{i}} \ar[u] \ar[r] \ar[d] & \bfE^{0}_{\bfV''', \bfW^{\bullet}, \underline{i}} \ar[u]^{j_{\bfV''', \underline{i}}} \ar[d]_{\phi_{\bfV''', \underline{i}}} \\
\txt{$\dot{\bfE}_{\bfV, \bfW^{\bullet},\underline{i}}$\\ $\times$\\ $\mathbf{Gr}^{s_i}(\nu_i, \tilde{\nu}_i)$} & \txt{$\dot{\bfE}_{\bfV''', \bfW^{\bullet},\underline{i}}$\\ $\times$\\ $\mathbf{Fl}^{s_i}(\nu_i, \nu_i + 1, \tilde{\nu}_i)$} \ar[l]^{q_1} \ar@{=}[r] & \txt{$\dot{\bfE}_{\bfV''', \bfW^{\bullet},\underline{i}}$\\ $\times$\\ $\mathbf{Fl}^{s_i}(\nu_i, \nu_i + 1, \tilde{\nu}_i)$} \ar[r]_{q_2} & \txt{$\dot{\bfE}_{\bfV''', \bfW^{\bullet},\underline{i}}$\\ $\times$\\ $\mathbf{Gr}^{s_i}(\nu_i + 1, \tilde{\nu}_i)$},
}
\]
\[
\xymatrix{
\bfE_{\bfV''', \bfW^{\bullet},\Omega} \ar@{<-}[d]_{j_{\bfV''', \underline{i}}} & & \bfE_{\bfV', \bfW^{\bullet},\Omega} \ar@{<-}[d]_{j_{\bfV', \underline{i}}} \\
\bfE_{\bfV''', \bfW^{\bullet}, \underline{i}}^0 \ar@{->}[d]_{\phi_{\bfV''', \underline{i}}} & & \bfE_{\bfV', \bfW^{\bullet}, \underline{i}}^0 \ar@{->}[d]_{\phi_{\bfV', \underline{i}}} \\
\txt{$\dot{\bfE}_{\bfV, \bfW^{\bullet}, \underline{i}}$\\ $\times$\\ $\mathbf{Gr}^{s_i}(\nu_i + 1, \tilde{\nu}_i)$} \ar@{<-}[r]_-{q_1'} & \txt{$\dot{\bfE}_{\bfV, \bfW^{\bullet}, \underline{i}}$\\ $\times$ \\ $\mathbf{Fl}^{s_i}(\nu_i - n + 1, \nu_i + 1, \tilde{\nu}_i)$} \ar@{->}[r]_-{q_2'} & \txt{$\dot{\bfE}_{\bfV, \bfW^{\bullet}, \underline{i}}$\\ $\times$\\ $\mathbf{Gr}^{s_i}(\nu_i - n + 1, \tilde{\nu}_i)$},
}
\]
where $\bfE^{0,''}_{\bfV''', \bfW^{\bullet}, \underline{i}}$ is the preimage of $\bfE^{0}_{\bfV''', \bfW^{\bullet}, \underline{i}}$ under $p_{3}$ and $\bfE^{0,'}_{\bfV''', \bfW^{\bullet}, \underline{i}}$ is the preimage of $\bfE^{0,''}_{\bfV''', \bfW^{\bullet}, \underline{i}}$ under $p_{2}$. Furthermore, the lower-left square is Cartesian. Then the functor
$\me_{\underline{i}}^{(n)} \mf_{\underline{i}}$ is isomorphic to $$(j_{\bfV', \underline{i}})_! (\phi_{\bfV', \underline{i}})^* (q_2')_! (q_1')^* (q_2)_! (q_1)^* (\phi_{\bfV, \underline{i}})_{\flat} (j_{\bfV, \underline{i}})^* [s_i(\nu_i + \tilde{\nu}_i - n \nu_i - n)].$$

Similarly, consider the following diagrams
\[
\xymatrix{
\bfE_{\bfV, \bfW^{\bullet},\Omega} \ar@{<-}[d]_{j_{\bfV, \underline{i}}} & & \bfE_{\bfV'', \bfW^{\bullet},\Omega} \ar@{<-}[d]_{j_{\bfV'', \underline{i}}} \\
\bfE_{\bfV, \bfW^{\bullet}, \underline{i}}^0 \ar@{->}[d]_{\phi_{\bfV, \underline{i}}} & & \bfE_{\bfV'', \bfW^{\bullet}, \underline{i}}^0 \ar@{->}[d]_{\phi_{\bfV'', \underline{i}}} \\
\txt{$\dot{\bfE}_{\bfV, \bfW^{\bullet}, \underline{i}}$\\ $\times$ \\$\mathbf{Gr}^{s_i}(\nu_i, \tilde{\nu}_i)$} \ar@{<-}[r]_-{\tilde{q_1}'} & \txt{$\dot{\bfE}_{\bfV, \bfW^{\bullet}, \underline{i}}$\\ $\times$ \\$\mathbf{Fl}^{s_i}(\nu_i - n, \nu_i, \tilde{\nu}_i)$} \ar@{->}[r]_-{\tilde{q_2}'} & \txt{$\dot{\bfE}_{\bfV, \bfW^{\bullet}, \underline{i}}$\\ $\times$\\ $\mathbf{Gr}^{s_i}(\nu_i - n, \tilde{\nu}_i)$};
}
\]

\[
\xymatrix{
\bfE_{\bfV'', \bfW^{\bullet},\Omega} & \bfE'_{\bfV', \bfW^{\bullet},\Omega} \ar[l]_{p_1} \ar[r]^{p_2} & \bfE''_{\bfV', \bfW^{\bullet},\Omega} \ar[r]^{p_3} & \bfE_{\bfV', \bfW^{\bullet},\Omega} \\
\bfE^{0}_{\bfV'', \bfW^{\bullet}, \underline{i}} \ar[u]^{j_{\bfV'', \underline{i}}} \ar[d]_{\phi_{\bfV'', \underline{i}}} & \bfE^{0,'}_{\bfV', \bfW^{\bullet}, \underline{i}} \ar[l] \ar[u] \ar[r] \ar[d] & \bfE^{0,''}_{\bfV', \bfW^{\bullet}, \underline{i}} \ar[u] \ar[r] \ar[d] & \bfE^{0}_{\bfV', \bfW^{\bullet}, \underline{i}} \ar[u]^{j_{\bfV', \underline{i}}} \ar[d]_{\phi_{\bfV', \underline{i}}} \\
\txt{$\dot{\bfE}_{\bfV'', \bfW^{\bullet},\underline{i}}$\\ $\times$\\ $\mathbf{Gr}^{s_i}(\nu''_i, \tilde{\nu}_i)$} & \txt{$\dot{\bfE}_{\bfV', \bfW^{\bullet},\underline{i}}$ \\ $\times$ \\ $\mathbf{Fl}^{s_i}(\nu''_i , \nu'_i, \tilde{\nu}_i)$} \ar[l]^{\tilde{q_1}} \ar@{=}[r] & \txt{$\dot{\bfE}_{\bfV', \bfW^{\bullet},\underline{i}}$\\ $\times$ \\$\mathbf{Fl}^{s_i}(\nu''_i , \nu'_i, \tilde{\nu}_i)$} \ar[r]_{\tilde{q_2}} & \txt{$\dot{\bfE}_{\bfV', \bfW^{\bullet},\underline{i}}$ \\ $\times$ \\ $\mathbf{Gr}^{s_i}(\nu'_i, \tilde{\nu}_i)$}.
}
\]
Then the functor $\mf_{\underline{i}} \me_{\underline{i}}^{(n)}$ is isomorphic to $$ (j_{\bfV', \underline{i}})_! (\phi_{\bfV', \underline{i}})^* (\tilde{q_2})_! (\tilde{q_1})^* (\tilde{q_2}')_! (\tilde{q_1}')^* (\phi_{\bfV, \underline{i}})_{\flat} (j_{\bfV, \underline{i}})^* [s_i(\nu_i + \tilde{\nu}_i - n \nu_i - n)].$$

We now only need to study the relation between $(q_2')_! (q_1')^* (q_2)_! (q_1)^*$ and $(\tilde{q_2})_! (\tilde{q_1})^* (\tilde{q_2}')_! (\tilde{q_1}')^*$. Let $X_{1}$ be the pullback of $q_2, q_1'$, and let $X_{2}$ be the pullback of $\tilde{q_2}', \tilde{q_1}$. There are varieties 
$$X_1 := \dot{\bfE}_{\bfV, \bfW^{\bullet},\underline{i}} \times \{(W,W' \subset W''' \subset \tilde{W} \mid \dim W' = \nu_i', \dim W = \nu_i, \dim W''' = \nu_i''', \dim \tilde{W} = \tilde{\nu}_i \}^{s_i},$$ 
$$X_2 := \dot{\bfE}_{\bfV, \bfW^{\bullet},\underline{i}} \times \{(W'' \subset W,W' \subset \tilde{W} \mid \dim W' = \nu_i', \dim W = \nu_i, \dim W'' = \nu_i'', \dim \tilde{W} = \tilde{\nu}_i \}^{s_i}.$$
Then $(q_2')_! (q_1')^* (q_2)_! (q_1)^* \cong (\pi_2)_! r_! r^* (\pi_1)^*$, where $\pi_{1},\pi_{2}$ and $r$ are the natural projections in the following diagram
\[
\xymatrix{
\txt{$\dot{\bfE}_{\bfV, \bfW^{\bullet},\underline{i}}$ \\ $\times$ \\ $\mathbf{Gr}^{s_i}(\nu_i, \tilde{\nu}_i)$} 
& \txt{$\dot{\bfE}_{\bfV''', \bfW^{\bullet},\underline{i}}$ \\ $\times$ \\ $\mathbf{Fl}^{s_i}(\nu_i, \nu_i + 1, \tilde{\nu}_i)$} 
\ar[l]_{q_1} \ar[r]^{q_2} 
& \txt{$\dot{\bfE}_{\bfV''', \bfW^{\bullet},\underline{i}}$\\ $\times$ \\ $\mathbf{Gr}^{s_i}(\nu_i + 1, \tilde{\nu}_i)$} \\
& X_1 \ar[r] \ar[u] \ar[ld]_{r} 
& \txt{$\dot{\bfE}_{\bfV, \bfW^{\bullet}, \underline{i}}$ \\ $\times$ \\ $\mathbf{Fl}^{s_i}(\nu_i - n + 1, \nu_i + 1, \tilde{\nu}_i)$} 
\ar[u]^{q_1'} \ar[d]_{q_2'} \\
\txt{$\dot{\bfE}_{\bfV, \bfW^{\bullet},\underline{i}}$ \\ $\times$ \\ $\mathbf{Gr}^{s_i}(\nu_i, \tilde{\nu}_i)$ \\ $\times$ \\ $\mathbf{Gr}^{s_i}(\nu_i - n + 1, \tilde{\nu}_i)$} 
\ar[uu]^{\pi_1} \ar[rr]^{\pi_2} 
& 
& \txt{$\dot{\bfE}_{\bfV, \bfW^{\bullet}, \underline{i}}$ \\ $\times$ \\ $\mathbf{Gr}^{s_i}(\nu_i - n + 1, \tilde{\nu}_i)$}.
}
\]

Similarly, $(\tilde{q_2})_! (\tilde{q_1})^* (\tilde{q_2}')_! (\tilde{q_1}')^* \cong (\pi_2)_! r'_! r'^* (\pi_1)^*$, where $\pi_{1},\pi_{2}$ and $r'$ are the natural projections in the following diagram
\[
	\xymatrix{
		\txt{$\dot{\bfE}_{\bfV,\bfW^{\bullet},\underline{i}}$ \\ $\times$ \\ $\mathbf{Gr}^{s_i}(\nu_i, \tilde{\nu}_i)$} 
		& \txt{$\dot{\bfE}_{\bfV''',\bfW^{\bullet},\underline{i}}$ \\ $\times$ \\ $\mathbf{Fl}^{s_i}(\nu_i - n, \nu_i, \tilde{\nu}_i)$} 
		\ar[l]_{\tilde{q_1}'} \ar[r]^{\tilde{q_2}'} 
		& \txt{$\dot{\bfE}_{\bfV''',\bfW^{\bullet},\underline{i}}$ \\ $\times$ \\ $\mathbf{Gr}^{s_i}(\nu_i - n, \tilde{\nu}_i)$} \\
		& X_2 \ar[r] \ar[u] \ar[ld]_{r'} 
		& \txt{$\dot{\bfE}_{\bfV,\bfW^{\bullet},\underline{i}}$ \\ $\times$ \\ $\mathbf{Fl}^{s_i}(\nu_i - n, \nu_i - n+1, \tilde{\nu_i})$} 
		\ar[u]^{\tilde{q_1}} \ar[d]_{\tilde{q_2}} \\
		\txt{$\dot{\bfE}_{\bfV,\bfW^{\bullet},\underline{i}}$ \\ $\times$ \\ $\mathbf{Gr}^{s_i}(\nu_i, \tilde{\nu}_i)$ \\ $\times$ \\ $\mathbf{Gr}^{s_i}(\nu_i - n+1, \tilde{\nu_i})$} 
		\ar[uu]^{\pi_1} \ar[rr]^{\pi_2} 
		& 
		& \txt{$\dot{\bfE}_{\bfV,\bfW^{\bullet},\underline{i}}$ \\ $\times$ \\ $\mathbf{Gr}^{s_i}(\nu_i - n+1, \tilde{\nu_i})$}.
	}
\]

By projection formula, we have $(\pi_2)_! r_! r^* (\pi_1)^* \cong (\pi_2)_! (r_! \overline{\bbQ}_{l} \otimes (\pi_1)^* (-))$, and $(\pi_2)_! r'_! r'^* (\pi_1)^* \cong (\pi_2)_! (r'_!\overline{\bbQ}_{l} \otimes (\pi_1)^* (-))$. We need to study the differences between the semisimple complexes $r'_!\overline{\bbQ}_{l}$ and $r_! \overline{\bbQ}_{l} $. (They are semisimple since $X_{1},X_{2}$ are smooth and $r,r'$ are proper.) It suffices to calculate restrictions of these complexes on a stratification.

Now we introduce the subsets $X_1^{\geq t}$ and $X_2^{\geq t}$ as follows. Define a map $\phi_{1}: \{(W,W' \subset W''' \subset \tilde{W} \mid \dim W' = \nu_i', \dim W = \nu_i, \dim W''' = \nu_i''', \dim \tilde{W} = \tilde{\nu}_i \} \rightarrow \{0,1\}$ by setting $\phi_{1}( (W,W', W''',\tilde{W}))=1 $ if $W' \subset W$ and $\phi_{1}( (W,W', W''',\tilde{W}))=0 $ otherwise. It induces a map $\Phi_{1}:X_{1} \rightarrow \{k\in \mathbb{N}| 0\leqslant k \leqslant s_{i} \}$ by $\Phi_{1}= \sum\limits_{1\leqslant l \leqslant s_{i}} \phi_{1,l},$ where $\phi_{1,l}$ is defined as $\phi_{1}$ on the $l$-th $\{(W,W' \subset W''' \subset \tilde{W} \mid \dim W' = \nu_i', \dim W = \nu_i, \dim W''' = \nu_i''', \dim \tilde{W} = \tilde{\nu}_i \}$ component in $X_{1}$. Then $X_1^{\geq t}$ is defined to be the preimage of $\{t,t+1,\cdots ,s_{i}\}$ under $\Phi_{1}$ and is closed in $X_1$. Denote $X_1^{\geq t}-X_1^{\geq t+1}$ by $X^{t}_{1}$. Similarly, we define $\phi_{2}: \{(W'' \subset W,W' \subset \tilde{W} \mid \dim W' = \nu_i', \dim W = \nu_i, \dim W'' = \nu_i'', \dim \tilde{W} = \tilde{\nu}_i \} \rightarrow \{0,1\}$ by setting $\phi_{2}( (W'',W, W',\tilde{W}))=1 $ if $W' \subset W$ and $\phi_{2}( (W'',W, W',\tilde{W}))=0 .$ Otherwise. The map $\phi_{2}$ induces a map $\Phi_{2}= \sum\limits_{1\leqslant l \leqslant s_{i}} \phi_{2,l}$ on $X_{2}$, and we define $X_2^{\geq t}$ to be the preimage of $\{t,t+1,\cdots ,s_{i}\}$ under $\Phi_{2}$ and is closed in $X_2$. Denote $X_2^{\geq t}-X_2^{\geq t+1}$ by $X^{t}_{2}$.

Denote $Y_t := r(X_1^{\geq t}) - r(X_1^{\geq t+1})$, then $
r_!\overline{\bbQ}_{l} = \bigoplus\limits_{0\leqslant t \leqslant s_{i}} F_{t},$
where each $F_{t}$ is the direct summand of $
r_!\overline{\bbQ}_{l}$ such that any simple direct summand (up to shift) $L_{t}$ of $F_{t}$ satisfies $\supp L_{t} \subset r(X_1^{\geq t})$ and $\supp L_{t} \cap Y_t \neq \emptyset$. Similarly, denote $Z_t := r(X_2^{\geq t}) - r(X_2^{\geq t+1})$. There is also a decomposition $
r'_!\overline{\bbQ}_{l} = \bigoplus\limits_{0\leqslant t \leqslant s_{i}} F'_{t},$ where the simple direct summands (up to shift) $L'_{t}$ of $F'_{t}$ satisfy $\supp L'_{t} \subset r'(X_2^{\geq t})$ and $\supp L'_{t} \cap Z_t \neq \emptyset$.

\textbf{(1)Case $t \neq 0,s_{i}$ :}
 Let $j_t: Y_t \rightarrow r(X_1^{\geq t})$. Then any simple perverse sheaf of the form $L_{t}$ satisfies $L_{t} \cong (j_t)_{*!} (j_t)^* L_{t}$. Observe that if $t \neq 0,s_{i}$, the connected components of $Y_t$ have nontrivial $a$-orbits. The simple perverse sheaf $(j_t)^* L_{t}$ is supported on a unique connected component and $a^{\ast}$ sends it to another simple perverse sheaf supported on another connected component of $Y_{t}$. We may assume $o(L) \neq 1$ is the least number such that $(a^{\ast})^{o(L)} L_{t}\cong L_{t}$. By the same reasoning as in the proof of Proposition \ref{dp}, after taking composition with $\otimes$ and $(\pi_{2})_{!}$, the direct sum of these $L_{t}\oplus a^{\ast} L_{t} \oplus (a^{\ast})^{2} L_{t}\cdots \oplus (a^{\ast})^{o(L)-1} L_{t}$ is acted by $\phi$ as a cyclic permutation, hence contributes to a traceless object. Since $F_{t}$ is a direct sum of objects of the form $L_{t}\oplus a^{\ast} L_{t} \oplus (a^{\ast})^{2} L_{t}\cdots \oplus (a^{\ast})^{o(L)-1} L_{t}$, we know that $F_{t}$ also contributes to traceless objects if $t\neq 0,s_{i}$. Similarly, those $F'_{t}$ also contribute to traceless objects if $t\neq 0,s_{i}$. 

\textbf{(2)Case $t = 0$ :}
 When $t=0$, $r$ and $r'$ restrict to the same isomorphism on the open subset $Z_0 = Y_0$ of the images, since in this case $W'''= W+W'$ and $W''=W\cap W'$ are uniquely determined by $W,W'$. Then $(j_{0})^{\ast}F_{0} \cong r_{!}(j^{1}_{0})^{\ast}\bar{\mathbb{Q}} \cong  r'_{!}(j^{2}_{0})^{\ast}\bar{\mathbb{Q}} \cong (j_{0})^{\ast}F'_{0}$, where $j^{1}_{0},j^{2}_{0}$ are the inclusion of $X^{0}_{1},X^{0}_{2}$ respectively. In particular, $F_{0} \cong F'_{0}$.

\textbf{(3)Case $t = s_{i}$ :}
Denote $j_{s_{i}}: Y_{s_i}=Z_{s_{i}} \rightarrow \dot{\bfE}_{\bfV, \bfW^{\bullet},\underline{i}} \times \mathbf{Gr}^{s_i}(\nu_i, \tilde{\nu}_i) \times \mathbf{Gr}^{s_i}(\nu_i - n+1, \tilde{\nu_i})$, then $(j_{s_{i}})^{\ast}r_{!} \overline{\bbQ}_{l} \cong r_{!}(\overline{\mathbb{Q}}_{l}|_{ X^{s_{i}}_{1} } ). $ Since the restriction of $r$ on $X^{s_{i}}_{1}$ is a locally trivial fibration and the fiber is $s_{i}$-times product of the projective space $\bbP^{\tilde{\nu_i}-\nu_{i}-1}$, we have $$(j_{s_{i}})^{\ast}r_{!} \overline{\bbQ}_{l} \cong \bigoplus\limits_{0 \leq r_1, \dots, r_{s_i} \leq \tilde{\nu_{i}}-\nu_i -1} \overline{\bbQ}_{l}|_{Y_{s_{i}}} [-2r_1 - 2r_2 - \dots - 2r_{s_i}]$$ with respect to the the K\"unneth formula. Similarly, 
$$(j_{s_{i}})^{\ast}r'_{!} \overline{\bbQ}_{l} \cong \bigoplus\limits_{0 \leq r'_1, \dots, r'_{s_i} \leq \nu_i -n} \overline{\bbQ}_{l}|_{Y_{s_{i}}} [-2r'_1 - 2r'_2 - \dots - 2r'_{s_i}].$$

By an argument similar to that in the proof of Proposition \ref*{dp}, only the direct summands with $r_1 = r_2 = \dots = r_{s_i}$ and $r'_1 = r'_2 = \dots = r'_{s_i}$ contribute to non-traceless elements. Observe that for $n\geq 2$, $(\pi_2)_!( \overline{\bbQ}_{l}|_{Y_{s_{i}}} \otimes (\pi_1)^*(-)) \cong (\bar{q}_{2})_{!}(\bar{q}_{1})^{\ast} $, where $\bar{q}_{1},\bar{q}_{2}$ are the natural projections appearing in the definition of $\mathcal{E}^{(n-1)}_{\underline{i}}$
$$ \bar{q}_{1}:\dot{\mathbf{E}}_{\mathbf{V},\mathbf{W}^{\bullet},\underline{i}} \times \mathbf{Fl}^{s_{i}}(\nu_{i}-n+1,\nu_{i},\tilde{\nu}_{i}) \rightarrow \dot{\mathbf{E}}_{\mathbf{V},\mathbf{W}^{\bullet},\underline{i}} \times \mathbf{Gr}^{s_{i}}(\nu_{i},\tilde{\nu}_{i}),$$
$$ \bar{q}_{2}:\dot{\mathbf{E}}_{\mathbf{V},\mathbf{W}^{\bullet},\underline{i}} \times \mathbf{Fl}^{s_{i}}(\nu_{i}-n+1,\nu_{i},\tilde{\nu}_{i}) \rightarrow \dot{\mathbf{E}}_{\mathbf{V},\mathbf{W}^{\bullet},\underline{i}} \times \mathbf{Gr}^{s_{i}}(\nu_{i}-n+1,\tilde{\nu}_{i}),$$
we can obtain
\begin{equation*}
	\begin{split}
		&\me_{\underline{i}}^{(n)} \mf_{\underline{i}}(L, \phi) \oplus \bigoplus_{k=0}^{-m-1} \me_{\underline{i}}^{(n-1)}[(-m-1-2k) s_i](L, \phi)\\
		\cong &\mf_{\underline{i}} \me_{\underline{i}}^{(n)}(L, \phi) \oplus \bigoplus_{k=0}^{m-1} \me_{\underline{i}}^{(n-1)}[(m-1-2k) s_i](L, \phi)
	\end{split}
\end{equation*}
modulo traceless objects. For $n=1$, since $(j_{\bfV,i})_!(j_{\bfV,i})^*\cong \mathbf{Id}$ on $\mathcal{D}^{b}_{G_{\bfV}}(\bfE_{\bfV, \bfW^{\bullet},\Omega})/\mn_{\bfV}$, we can also get the proof.
\end{proof}

\section{The realization of $\mathbf{U}$-modules and their signed bases}\label{Grothendieck g}
In this section, we determine the structure of the Grothendieck group of $\widetilde{\mathcal{Q}_{\mathbf{V},\mathbf{W}^{\bullet}}/\mathcal{N}_{\mathbf{V}}} $.

\subsection{Key Lemma} 
 We first recall the key inductive lemma from \cite{fang2023lusztigsheavesintegrablehighest} and \cite{fang2023lusztigsheavestensorproducts}, which generalizes Lusztig's lemma in \cite{L01}. Assume $\Omega$ is an orientation such that $i$ is a sink and $\mathbf{V}$ is a graded space with dimension vectors $\nu \in \mathbb{N}I$, and each $\bfW^{k}$ is a graded space with dimension vectors $\omega^{k} \in \mathbb{N}I^{k}$. The variety $\bfE_{\bfV, \bfW^{\bullet},\Omega}$ has a stratification $^{r}\bfE_{\bfV, \bfW^{\bullet},i}, 0\leqslant r \leqslant \nu_{i}$, where each $^{r}\bfE_{\bfV, \bfW^{\bullet},\Omega}$ is defined by 
 $$^{r}\bfE_{\bfV, \bfW^{\bullet}, i}=\{ x \in \bfE_{\bfV, \bfW^{\bullet},\Omega} | \dim \coker(\sum\limits_{h \in \tilde{\Omega}^{(N)}, t(h) = i} x_h) = r \}.$$ 
 Then for each simple perverse sheaf $L$, there exists a unique $t=:t_{i}(L) $ such that $\supp L \cap {^{t}\bfE}_{\bfV, \bfW^{\bullet}, i}$ is dense in $\supp L$.

 By \cite[Corollary 4.5]{fang2023lusztigsheavesintegrablehighest} and \cite[Proposition 4.1]{fang2023lusztigsheavestensorproducts}, we have the following lemma.
 \begin{lemma} 
 	For a simple perverse sheaf $L$ in $\mq_{\bfV, \bfW^{\bullet}}$ with $t_{i}(L)=r$, take $\bfV'$ with dimension vector $\nu -ri$, then there exists a unique simple perverse sheaf $K$ in $\mq_{\bfV', \bfW^{\bullet}}$ such that $t_{i}(K)=0$ and $$\mf^{(r)}_{i} K \cong L \oplus \bigoplus\limits_{t_{i}(L')>r}L'[n_{L'}].$$
 	With the notation above, we denote by $\pi_{i,r}(L)=K,$ then $\pi_{i,r}$ is a bijection between the set $\{L \in \mq_{\bfV, \bfW^{\bullet}}|$ $L$ is a simple perverse sheaf with $t_{i}(L)=r\}$ and the set $\{K \in \mq_{\bfV', \bfW^{\bullet}}|$ $K$ is a simple perverse sheaf with $t_{i}(K)=0\}$.
 \end{lemma}

 Now we apply the lemma above to a quiver with automorphism $a$. Since $t_{a^{-1}(i)}(L)=t_i(a^{\ast}L),$ if $L\cong a^{\ast}L,$ $t_i(L)$ is invariant as $i$ varies within a single orbit.
 \begin{corollary}\label{ind1}
 		With the notation above, we assume $L$ is a simple perverse sheaf in $\mq_{\bfV, \bfW^{\bullet}}$ with $t_{i}(L)=r$ and $a^{\ast} L \cong L.$ Take $\bfV'$ with dimension vector $\nu -r\underline{i}.$ There exists a unique simple perverse sheaf $K$ in $\mq_{\bfV', \bfW^{\bullet}}$ such that $t_{i}(K)=0,a^{\ast} K \cong K$ and the following equation, $$\mf^{(r)}_{\underline{i}} K \cong L \oplus \bigoplus\limits_{i'\in \underline{i}, t_{i'}(L')>r,L'\in \mq_{\bfV', \bfW^{\bullet}}}L'[n_{L'}].$$
 \end{corollary}
 
\begin{proof}
  Repeatedly applying the above lemma,  it follows that there exists a simple perverse sheaf $K$ such that $t_{i}(K)=0$ and $$\mf^{(r)}_{\underline{i}} K \cong L \oplus \bigoplus\limits_{i'\in \underline{i}, t_{i'}(L')>r}L'[n_{L'}].$$ 
  By the uniqueness of $K$, we must have $a^{\ast}K \cong K$.
\end{proof} 
With the notation above, we denote by $\pi_{\underline{i},r}(L)=K,$ then $\pi_{\underline{i},r}$ is a bijection between the set $\{L \in \mq_{\bfV, \bfW^{\bullet}}|$ $L$ is a nonzero simple perverse sheaf with $a^{\ast} L \cong L$ and $t_{i}(L)=r\}$ and the set $\{K \in \mq_{\bfV', \bfW^{\bullet}}|$ $K$ is a nonzero simple perverse sheaf with $a^{\ast} K \cong K$ and $t_{i}(K)=0\}$.

Similarly, by assuming that $i$ is a source, we define $t^{\ast}_{i}(L)$ for any simple perverse sheaf $L,$ changing $\dim\coker(\sum\limits_{h \in \tilde{\Omega}^{(N)}, t(h) = i} x_h)$ in the definition of $t_i(L)$ to $\dim\ker(\bigoplus\limits_{h \in \tilde{\Omega}^{(N)}, s(h) = i} x_h).$ We also 
denote the right multiplication by $\overline{\bbQ_l}$ by $\mf_{\underline{i}}^{(r),\vee}$. Then we have the following proposition, which is an analogue of \cite[Proposition 4.2]{fang2023lusztigsheavestensorproducts}.  We do not assume $L$ in $\mq_{\bfV, \bfW^{\bullet}}$, because $\mf_{\underline{i}}^{(r),\vee}$ does not send objects of $\mq_{\bfV',\bfW^{\bullet}}$ to those of $\mq_{\bfV,\bfW^{\bullet}}$.

 \begin{corollary}\label{ind1'}
	With the notation above, we assume $L$ is a simple perverse sheaf with $t^{\ast}_{i}(L)=r$ and $a^{\ast} L \cong L$, and take $\bfV'$ with dimension vector $\nu -r\underline{i}$, then there exists a unique simple perverse sheaf $K$ such that $t^{\ast}_{i}(K)=0,a^{\ast} K \cong K$ and $$\mf^{(r),\vee}_{\underline{i}} K \cong L \oplus \bigoplus\limits_{i'\in \underline{i}, t^*_{i'}(L')>r}L'[n_{L'}].$$
\end{corollary}

\subsection{The Grothendieck group}

Now let $\mathcal{A}=\mathbb{Z}[v,v^{-1}],$ $\mathcal{O}=\mathbb{Z}[\zeta]$ and $\mathcal{O}'=\mathcal{O}[v,v^{-1}]$, where $\zeta$ is a primitive $o$-th root of unity and $o$ is the order of the automorphism $a$ on the quiver $Q$. There is a $\mathbb{Z}$-linear involution $\bar{\ }$ on $\mathcal{O}'$ defined by $\bar{v}=v^{-1}$ and $\bar{\zeta}=\zeta^{-1}$.

Let $_{\mathcal{O}'}\mathbf{U}={\mathcal{O}'}\otimes_{\mathcal{A}} {}_{\mathcal{A}}\mathbf{U}$ generated by Lusztig's divided powers and those $K_{\pm i'}$ and $_{\mathcal{O}'}L(\lambda)$ be the $_{\mathcal{O}'}\mathbf{U}$-module generated by the highest weight vector. We can also define $_{\mathcal{A}}\mathbf{U}$ and $_{\mathcal{A}}L(\lambda)$ in the same way.

Let $V$ have dimension vector $\nu\in(\mathbb N I)^a$, and let
$W^k$ have dimension vector $\omega^k\in(\mathbb N I^k)^a$ for
$1\leq k\leq N$. Identify the set of $a$-orbits in $I$ with $I'$.
For $i'\in I'$ and any vertex $i$ in the corresponding orbit, write
\[
  \nu_{i'}=\nu_i,
  \qquad
  \omega^k_{i'}=\omega^k_{i^k}=\dim W^k_{i^k},
\]
where $i^k$ is the copy of $i$ in $I^k$.
These numbers are independent of the chosen representative $i$.
The framing $W^k$ determines the dominant weight
\[
  \lambda^k=\sum_{i'\in I'}\omega^k_{i'}\beta_{i'}\in X^+,
  \qquad
  \langle h_{i'},\lambda^k\rangle=\omega^k_{i'}.
\]
The weight associated with $(\nu,\omega^\bullet)$ is
\[
  \operatorname{wt}(\nu,\omega^\bullet)
    =\sum_{k=1}^N\lambda^k
       -\sum_{i'\in I'}\nu_{i'}\alpha_{i'}\in X.
\]
Consequently,
\[
  \bigl\langle h_{i'},\operatorname{wt}(\nu,\omega^\bullet)\bigr\rangle
    =\sum_{k=1}^N\omega^k_{i'}
       -\sum_{j'\in I'}c_{i'j'}\nu_{j'}.
\]
In particular, dimension vectors contribute combinations of simple
roots, whereas framing vectors determine combinations of fundamental
weights.
\begin{definition}
	Define $\mk(\omega^{\bullet})=\bigoplus\limits_{\nu \in \bbN[I]^{a}} \mk(\nu,\omega^{\bullet}) $ and the Grothendieck group $\mk(\nu,\omega^{\bullet})$ of $\widetilde{\mathcal{Q}_{\mathbf{V},\mathbf{W}^{\bullet}}/\mathcal{N}_{\mathbf{V}}} $ is defined to be the $\mathcal{O}'$-module spanned by $[(L,\phi)]$, $(L,\phi)$ in $\widetilde{\mathcal{Q}_{\mathbf{V},\mathbf{W}^{\bullet}}/\mathcal{N}_{\mathbf{V}}} $, subject to the following relations
	$$[(L_1, \phi_1)] + [(L_2, \phi_2)] = [(L_1 \oplus L_{2}, \phi_1 \oplus \phi_{2})]; $$
	$$[(L, \phi)[n]] = v^n [(L, \phi)];$$
	$$[(L, t \phi)] = t [(L, \phi)] \quad \text{for} \quad t \in \overline{\bbQ}_l \quad \text{with} \quad t^{o}=1;$$
	$$[(L, \phi)]=0 \quad \text{for a traceless object $(L,\phi)$}.$$ 
	
	Similarly, we define $\mv(\omega^{\bullet})=\bigoplus\limits_{\nu \in \bbN[I]^{a}} \mv(\nu,\omega^{\bullet}) $ to be the $\mathcal{O}'$-module spanned by $[(L,\phi)]$, $(L,\phi)$ in $\widetilde{\mathcal{Q}_{\mathbf{V},\mathbf{W}^{\bullet}}}$ subject to the same relations.
\end{definition}

Observe that if $a(\nu) \neq \nu$, then $\mathcal{Q}_{\mathbf{V},\mathbf{W}^{\bullet}}$ always contributes traceless elements, hence we only need to consider dimension vectors $\nu \in \bbN[I]^{a}$.

If each $\boldsymbol{\nu}^{k}$ is an $a$-flag type, there is a canonical choice $\phi_{0}$ for $L_{\boldsymbol{\nu}^{1}\boldsymbol{e}^{1}\boldsymbol{\nu}^{2}\boldsymbol{e}^{1}\cdots\boldsymbol{\nu}^{N}\boldsymbol{e}^{N}}$ by \cite{lusztig2010introduction}, given by 
$$ \phi_{0} :a^{\ast} \pi_{\boldsymbol{\nu}^{1}\boldsymbol{e}^{1}\boldsymbol{\nu}^{2}\boldsymbol{e}^{1}\cdots\boldsymbol{\nu}^{N}\boldsymbol{e}^{N},\tilde{\Omega}^{(N)}}\overline{\bbQ}_{l} = \pi_{\boldsymbol{\nu}^{1}\boldsymbol{e}^{1}\boldsymbol{\nu}^{2}\boldsymbol{e}^{1}\cdots\boldsymbol{\nu}^{N}\boldsymbol{e}^{N},\tilde{\Omega}^{(N)}} a^{\ast}\overline{\bbQ}_{l} \cong  \pi_{\boldsymbol{\nu}^{1}\boldsymbol{e}^{1}\boldsymbol{\nu}^{2}\boldsymbol{e}^{1}\cdots\boldsymbol{\nu}^{N}\boldsymbol{e}^{N},\tilde{\Omega}^{(N)}}\overline{\bbQ}_{l}.$$
We will show that the images of these $(L_{\boldsymbol{\nu}^{1}\boldsymbol{e}^{1}\boldsymbol{\nu}^{2}\boldsymbol{e}^{1}\cdots\boldsymbol{\nu}^{N}\boldsymbol{e}^{N}},\phi_{0})$ span the whole Grothendieck group in the next subsections.

The following proposition is an analogue of \cite[Proposition 12.5.2]{lusztig2010introduction}.
\begin{proposition}\label{4.5}
	Take a graded space $\bfV$ such that its dimension vector $\nu \in \mathbb{N}I^{a}$. Given a nonzero simple perverse sheaf $L$ in $\mathcal{Q}_{\mathbf{V}, \mathbf{W}^{\bullet}}/\mathcal{N}_{\mathbf{V}}$ such that $a^{\ast} L \cong L$, then there exists $\phi$ such that $(L,\phi)$ belongs to $\widetilde{\mathcal{Q}_{\mathbf{V}, \mathbf{W}^{\bullet}}/\mathcal{N}_{\mathbf{V}}} $ and $(\mathbf{D}L,\mathbf{D}(\phi)^{-1}) \cong (L,\phi)$ in $\widetilde{\mathcal{Q}_{\mathbf{V}, \mathbf{W}^{\bullet}}/\mathcal{N}_{\mathbf{V}}} $. In particular, $ \mathbf{D}L \cong L$. Moreover, such $\phi$ is unique if $o$ is odd and is unique up to $\pm 1$ if $o$ is even. 
	
\end{proposition}

\begin{proof}
	Since an isomorphism $\phi:a^{\ast} L \rightarrow L$ in $\mathcal{Q}_{\mathbf{V}, \mathbf{W}^{\bullet}}$ naturally induces an isomorphism $\phi:a^{\ast} L \rightarrow L$ in $\mathcal{Q}_{\mathbf{V}, \mathbf{W}^{\bullet}}/\mathcal{N}_{\mathbf{V}} $, it suffices to find $\phi$ in $\mathcal{Q}_{\mathbf{V}, \mathbf{W}^{\bullet}}$ such that $(L,\phi)$ satisfies the same properties in $\widetilde{\mathcal{Q}_{\mathbf{V}, \mathbf{W}^{\bullet}}} $.

	We prove the proposition by induction on $N.$ When $N=1,$ if $t_i(L)=0$ for any $i\in I,$ $L=L_{\boldsymbol{e_1}}.$ There exists $\phi,$ satisfying $\mathbf{D}(L_{\boldsymbol{e_1}},\phi)\cong (L_{\boldsymbol{e_1}},\phi).$
	
	Given $a^{\ast} L \cong L$ with $\bfV$ nonzero, there exists an $i$ such that $t_{i}(L)=r>0$. Then there exists a unique $K$ with $t_{i}(K)=0$ and $a^{\ast}K \cong K$ such that
	$$\mf^{(r)}_{\underline{i}} K \cong L \oplus \bigoplus\limits_{t_{i}(L')>r,a^*L'\cong L'}L'[n_{L'}]\oplus C,$$ where $C$ is a direct sum of traceless elements.
	
	By induction on the dimension vector of $\bfV$ and descending induction on $r$, we may assume that there exist $\phi'$ for $K$ and $\phi_{L'}$ for those $a$-fixed $L'$ with $t_{i}(L')>r$ such that $(K,\phi')$ and $(L',\phi_{L'})$ satisfy the condition $(\mathbf{D}K,\mathbf{D}(\phi')^{-1}) \cong (K,\phi'), (\mathbf{D}L',\mathbf{D}(\phi_{L'})^{-1}) \cong (L',\phi_{L'}).$ 
	
	There exists $\phi,$ such that the following equation holds in $\mv(\nu,\omega)$, $$\mf^{(r)}_{\underline{i}}[(K,\phi')]=[(L,\phi)]+ \sum\limits_{t_{i}(L')>r} P_{L'} [(L',\phi_{L'})] $$ with $P_{L'} \in \mathcal{O}'$. Apply $\mf^{(r)}_{\underline{i}}$ to $\mathbf{D}[(K,\phi')]=[(K,\phi')]$, we obtain 
	$$\mathbf{D}[(L,\phi)]+ \sum\limits_{t_{i}(L')>r} \bar{P}_{L'} \mathbf{D}[(L',\phi_{L'})] = [(L,\phi)]+ \sum\limits_{t_{i}(L')>r} P_{L'} [(L',\phi_{L'})],$$
	then $\mathbf{D}[(L,\phi)]- [(L,\phi)] = \sum\limits_{t_{i}(L')>r}(P_{L'} \bar{P}_{L'}) \mathbf{D}[(L',\phi_{L'})] $.
	
	Observe that those $[(L,\phi)]$ and $[(L',\phi_{L'})]$ are contained in an $\mathcal{O}'$ basis of $\mv(\nu,\omega)$ and $\mathbf{D}[(L,\phi)]$ must be a multiple of one of those elements. Since $t_{i}(L) \neq t_{i}(L')$, $\mathbf{D}L$ can never be isomorphic to the other $L' \cong \mathbf{D}L'.$ The equation $\mathbf{D}[(L,\phi)]- [(L,\phi)] = \sum\limits_{t_{i}(L')>r}(P_{L'} \bar{P}_{L'}) \mathbf{D}[(L',\phi_{L'})] $ forces $\mathbf{D}L \cong L$ and $\mathbf{D}[(L,\phi)]= [(L,\phi)]$. This proves the existence of $\phi$. 

    We claim that for any simple perverse sheaf $a^*L\cong L$ as a direct summand of $(L_{\boldsymbol{\nu}^{1}\boldsymbol{e}^{1}\cdots \boldsymbol{e}^{N-1}\boldsymbol{\nu}^{N}},\phi_0),$ up to shift, there exists $\phi$ satisfying $\mathbf{D}(L,\phi)\cong (L,\phi).$
Applying Corollary \ref{ind1'}, if $t_i^*(L)=r,$ there exist $(K,\phi')$ and $\phi$ satisfying $\mf^{(r),\vee}_{\underline{i}}(K,\phi')=(L,\phi)\oplus \bigoplus(L',\phi')[n_{L',\phi'}]\oplus N,$ where $t_i^*(L')>r$ and $N$ is a traceless object. $(K,\phi'), (L',\phi')$ are direct summands up to shift of some flag sheaf complex $(L_{\boldsymbol{\nu'}^{1}\boldsymbol{e}^{1}\cdots \boldsymbol{e}^{N-1}\boldsymbol{\nu'}^{N}},\phi_0)$ and $(K,\phi')$ is stable under Verdier duality. Then $(L,\phi)$ is stable under Verdier duality. By induction on $\dim \bfV$ and descending induction on $t_i^*(L),$ applying the same methods as in the case $N=1,$ it is reduced to the case $\boldsymbol{\nu^N}=\emptyset,$ which is the case $N-1.$

Considering the projection map $\pi_{\bfV,\bfW^N}:\bfE_{\bfV,\bfW^{\bullet},\Omega}\rightarrow \bfE_{\bfV,\bfW^{\bullet-N},\Omega},$ where $\bfW^{\bullet-N}$ is the set of graded spaces $\bfW^1,\cdots,\bfW^{N-1},$ $\pi_{\bfV,\bfW^N}^*[\sum_{i\in I}\dim \bfV_i\dim\bfW^N_{i^N}]L_{\boldsymbol{\nu}}=L_{\boldsymbol{\nu}\boldsymbol{e^N}},$ and $\pi_{\bfV,\bfW^N}^*[\sum_{i\in I}\dim \bfV_i\dim\bfW^N_{i^N}]$ preserves simple perverse sheaf. Since $\pi_{\bfV,\bfW^N}^*[\sum_{i\in I}\dim \bfV_i\dim\bfW^N_{i^N}]$ communites with $a^*$ and $\mathbf{D},$ if $(L,\phi)$ satisfying $\mathbf{D}(L,\phi)\cong (L,\phi),$ $\mathbf{D}(\pi_{\bfV,\bfW^N}^*[\sum_{i\in I}\dim \bfV_i\dim\bfW^N_{i^N}](L,\phi))\cong \pi_{\bfV,\bfW^N}^*[\sum_{i\in I}\dim \bfV_i\dim\bfW^N_{i^N}](L,\phi).$ Hence, for any simple perverse sheaf $a^*L\cong L$ as a direct summand of $(L_{\boldsymbol{\nu}^{1}\boldsymbol{e}^{1}\cdots \boldsymbol{e}^{N-1}\boldsymbol{\nu}^{N}\boldsymbol{e^N}},\phi_0),$ up to shift, there exists $\phi,$ $\mathbf{D}(L,\phi)\cong (L,\phi).$

	If $\eta \phi$ is another such isomorphism with $\eta^{o}=1$, then applying $\mathbf{D}$ gives $\eta^{-1}=\eta$. Hence $\eta=1$ if $o$ is odd and $\eta = \pm 1$ if $o$ is even. The uniqueness follows.
\end{proof}

\subsection{The irreducible integrable highest weight modules}
In this subsection, we assume $N=1$. We denote $\mathcal{Q}_{\mathbf{V},\mathbf{W}^{\bullet}}$ by $\mathcal{Q}_{\mathbf{V},\mathbf{W}}$, and $\mk(\omega^{\bullet})$ by $\mk(\omega)$.

Observe that by \cite[11.1.8]{lusztig2010introduction}, $\mk(\nu,\omega)$ is a $\mathcal{O}'$-module with a basis corresponding to the set of nonzero simple perverse sheaves $L$ in $\mathcal{Q}_{\mathbf{V},\mathbf{W}}/\mathcal{N}_{\mathbf{V}}$ with $a^{\ast}L \cong L$. 

\begin{theorem}\label{4.6}
	If we identify the orbit $\underline{i} \in \underline{I}$ of $i$ with the element $i' \in I'$, then the functors $\me^{(n)}_{\underline{i}}, \mf^{(n)}_{\underline{i}},\mk_{\pm\underline{i}} $ act by $E^{(n)}_{i'},F^{(n)}_{i'},\tilde{K}_{\pm i'}$($\mathbf{id}[-2\nu_i+\tilde{\nu_i}]$ act by $K_i'$, $i\in \underline{i}$) on the Grothendieck group $\mk(\omega)$ such that $\mk(\omega)$ becomes an integrable $_{\mathcal{O}'}\mathbf{U}$-module and is canonically isomorphic to $_{\mathcal{O}'}L(\lambda)$, where $\lambda=\sum\limits_{i'\in I'}\omega_{i'}\beta_{i'}$.
\end{theorem}

\begin{proof}
	By Proposition \ref{dp} and \ref{dq1}, we know that $\mk(\omega)$ carries an $_{\mathcal{O}'}\mathbf{U}$ module structure. For any $x \in \mk(\nu,\omega)$, if $n > \nu_{i} $, it follows that $E^{(n)}_{i'}(x)=0$. If $n> \tilde{\nu_{i}},$ then $F^{(n)}_{i'}(x)$ belongs to $\mk(\nu',\omega)$, where $\nu'=\nu+n\underline{i}$ satisfies $\nu'_{i} > \tilde{\nu}'_{i}$. In this case, $\bfE^{\geqslant 1}_{\bfV, \bfW, \underline{i}} = \bfE_{\bfV, \bfW, \Omega} $, we obtain $\mk(\nu',\omega) =0$. Hence $\mk(\omega)$ is an integrable module.
	
	It suffices to show that $\mk(\omega)$ is a highest weight module. Let $\mathcal{M}(\nu,\omega)$ be the $\mathcal{O}'$ module spanned by $[(L_{\boldsymbol{\nu}\boldsymbol{e}},\phi_{0})], \boldsymbol{\nu} \in \mathcal{S}'. $ Since $	 \mf^{(n)}_{\underline{i}}(L_{\boldsymbol{\nu}\boldsymbol{e}},\phi_{0} )=(L_{ (n\underline{i},\boldsymbol{\nu}\boldsymbol{e})},\phi_{0}), $
	$\bigoplus\limits_{\nu \in \bbN[I]^{a}}\mathcal{M}(\nu,\omega)$ is generated by $(L_{ (\boldsymbol{e})},\phi_{0})$ under the action of $ \mf^{(n)}_{\underline{i}}, n \in \mathbb{N}, \underline{i} \in \underline{I}$.

	We claim that if $(L,\phi)$ satisfies Proposition \ref{4.5}, then $[(L,\phi)]$ belongs to $\mathcal{M}(\nu,\omega)$. If $\bfV=0$, it is trivial. Otherwise, we can find $i$ such that $t_{i}(L)=r>0.$ Then there exists a unique $K$ such that $t_{i}(K)=0$ and
	\begin{equation}\label{OZ}
		\mf^{(r)}_{\underline{i}}[(K,\phi')]=[(L,\phi)]+ \sum\limits_{t_{i}(L')>r,a^*L'\cong L'} P_{L'} [(L',\phi_{L'})]. 
	\end{equation} 
	By induction on dimension vector of $\bfV$ and descending induction on $r$, we may assume $[(K,\phi')]$ and $[(L',\phi_{L'})] $ belong to $\bigoplus\limits_{\nu \in \bbN[I]^{a}}\mathcal{M}(\nu,\omega)$. Since $\bigoplus\limits_{\nu \in \bbN[I]^{a}}\mathcal{M}(\nu,\omega)$ is closed under $\mf^{(r)}_{\underline{i}}$, it follows that $[(L,\phi)]$ also belongs to $\mathcal{M}(\nu,\omega)$.
	
	In particular, $\mk(\omega)=\bigoplus\limits_{\nu \in \bbN[I]^{a}}\mathcal{M}(\nu,\omega)$ is an integrable highest weight module with the highest weight vector $[(L_{\boldsymbol{e},\phi_{0}} )]$, hence it is canonically isomorphic to $_{\mathcal{O}'}L(\lambda)$.
\end{proof}

\subsection{The tensor products of irreducible integrable highest weight modules}

In this section, we prove that $\mk(\omega^{\bullet})=\bigoplus\limits_{\nu \in \bbN[I]^{a}} \mk(\nu,\omega^{\bullet}) $ is isomorphic to the tensor product $ _{\mathcal{O}'}L(\lambda^{1}) \otimes_{\mathcal{O}'} {_{\mathcal{O}'}L}(\lambda^{2}) \otimes_{\mathcal{O}'} \cdots \otimes_{\mathcal{O}'} {_{\mathcal{O}'}L}(\lambda^{N})$. 

\begin{proposition}
The $\mathcal{O}'$-module	$\mk(\omega^{\bullet})$ is equal to $\mo'-$spanned of $[(L_{\boldsymbol{\nu}^{1}\boldsymbol{e}^{1}\cdots \boldsymbol{\nu}^{N}\boldsymbol{e}^{N}},\phi_0)]$ where $\boldsymbol{\nu}^{i}$ is an $a$-flag type for $i=1,2,\cdots,N.$
\end{proposition}
\begin{proof}
Let $M(\omega^{\bullet})=\bigoplus\limits_{\nu \in \bbN[I]^{a}} M(\nu,\omega^{\bullet}) $ be the $\mo'$-module spanned by these $[(L_{\boldsymbol{\nu}^{1}\boldsymbol{e}^{1}\cdots \boldsymbol{\nu}^{N}\boldsymbol{e}^{N}},\phi_0)]$. It is easy to see that $M(\omega^{\bullet})\subset \mk(\omega^{\bullet}).$ We only need to prove that $[(L,\phi)]\in M(\omega^{\bullet})$ for $L\in \mP_{\bfV,\bfW^{\bullet}}$ and $a^{\ast}(L)\cong L.$ We will prove for $[(L,\phi)]\in \mv(\omega^{\bullet})$ before the localization by induction on $N.$

When $N=1,$ the proof is the same as in Theorem \ref{4.6}. If the proposition holds for $N=k,$ then when $N=k+1,$ we first consider $\mo'-$module $\mv'(\omega^{\bullet})$ generated by $[(L_{\boldsymbol{\nu}^{1}\boldsymbol{e}^{1}\cdots \boldsymbol{e}^{N-1}\boldsymbol{\nu}^{N}},\phi_0)].$ We claim that for any simple perverse sheaf $(L,\phi)$ appeared as a direct summand of $(L_{\boldsymbol{\nu}^{1}\boldsymbol{e}^{1}\cdots \boldsymbol{e}^{N-1}\boldsymbol{\nu}^{N}},\phi_0),$ up to shift, $[(L,\phi)]\in \mv'(\omega^{\bullet}).$

Here, we prove the claim. By Corollary \ref{ind1'}, if $(L,\phi)$ is a direct summand of $(L_{\boldsymbol{\nu}^{1}\boldsymbol{e}^{1}\cdots \boldsymbol{e}^{N-1}\boldsymbol{\nu}^{N}},\phi_0),$ up to shift, and $\boldsymbol{\nu_N}\not=\emptyset$, there exist $i\in I,$ $t_i^*(L)=r$ and $(K,\phi'),$ satisfying $\mf_{\underline{i}}^{(r),\vee}(K,\phi')=(L,\phi)\oplus \bigoplus(L',\phi')[n_{L',\phi'}]\oplus N,$ where $t_i^*(L')>r$ and $N$ is a traceless object. $(K,\phi'), (L',\phi')$ are direct summands up to shift of some flag sheaf complex $(L_{\boldsymbol{\nu'}^{1}\boldsymbol{e}^{1}\cdots \boldsymbol{e}^{N-1}\boldsymbol{\nu'}^{N}},\phi_0),$.

Then, by induction of $\dim \bfV$ and descending induction of $t^*(L),$ we reduce the case to $\boldsymbol{\nu_N}=\emptyset,$ which is the case $N=k.$ The claim is proved.

Applying map $\pi_{\bfV,\bfW^N}[\sum_{i\in I}\dim\bfV_i\dim \bfW^N_{i^N}]$ defined in the proof of Proposition \ref{4.5}, the proof of $N=k+1$ follows. This completes the proof.
\end{proof}

In what follows, we consider the case $N=2,$ and denote $\omega^{\bullet}=\omega^{1,0},$ when $\omega_2=0.$
\begin{corollary}
	 The $\mathcal{O}'$-module $\mk(\omega^{\bullet})$ is spanned by the elements of the form $[(L_{\boldsymbol{\nu}^{1}\boldsymbol{e}^{1}\boldsymbol{\nu}^{2}\boldsymbol{e}^{2}},\phi_0)]$, where each $\boldsymbol{\nu}^{i}$ is an $a$-flag type for $i=1,2$.
\end{corollary}

\begin{lemma}
	The $\mathcal{O}'$-module $\mk(\omega^{1,0})$ is isomorphic to $\mk(\omega^{1})$. In particular, it is isomorphic to $_{\mo'}L(\omega^{1})$ as $_{\mo'}\mathbf{U}$-module.
\end{lemma}
\begin{proof}
	Observe that if $\boldsymbol{\nu}^{2}$ is nonempty, the complex $L_{\boldsymbol{\nu}^{1}\boldsymbol{e}^{1}\boldsymbol{\nu}^{2}}$ belongs to $\mn_{\bfV}$, hence $[(L_{\boldsymbol{\nu}^{1}\boldsymbol{e}^{1}\boldsymbol{\nu}^{2}},\phi_0)]$ is zero in $\mk(\omega^{1,0})$. The statement follows by Theorem \ref{4.6}.
\end{proof}

In order to determine the module structure of $\mk(\omega^{\bullet})$, we consider the $\mo'$-linear map $\Delta$ induced by the restriction map $\bigoplus\limits_{\bfV^{1},\bfV^{2}} \mathbf{Res}^{\bfV\oplus\bfW^{\bullet}}_{\bfV^{1}\oplus\bfW^{1}, \bfV^{2}\oplus\bfW^{2}}$, 
$$ \Delta: \mv(\omega^{\bullet}) \rightarrow \mv(\omega^{1,0}) \otimes_{\mathcal{O}'} \mv(\omega^{2}). $$
\begin{proposition}\label{right}
	The map $\Delta$ induces a well-defined surjective $\mo'$-linear morphism 
	$$ \Delta: \mk(\omega^{\bullet}) \rightarrow \mk(\omega^{1,0}) \otimes_{\mathcal{O}'} \mk(\omega^{2}) \cong \mk(\omega^{1}) \otimes_{\mathcal{O}'} \mk(\omega^{2}).$$
\end{proposition}
\begin{proof}
First, we prove that $\Delta$ is well-defined. Let $\mi(\nu,\omega^{\bullet}), \mi(\nu,\omega^{1}), \mi(\nu,\omega^{2})$ respectively be the submodules of $\mv(\nu,\omega^{\bullet})$, $\mv(\nu,\omega^{1,0})$ and $\mv(\nu,\omega^{2})$ spanned by $[(L,\phi)]$ with $L$ in $\mn_{\bfV}$.
It suffices to prove that $\Delta(\mi(\omega^{\bullet})) \subseteq \mi(\nu,\omega^{1})\otimes_{\mo'}\mv(\omega^{2})+\mv(\omega^{1})\otimes\mi(\nu,\omega^{2})$. If $L \in \mn_{\bfV}$, there exist $i\in I$ such that $L$ is a direct summand of $L_{\boldsymbol{\nu},ri}$. By \cite[Lemma 12.3.3]{lusztig2010introduction}, $$\Delta([L_{\boldsymbol{\nu},ri},\phi_{0}])= \sum \limits_{r_{1}>0~or~r_{2}>0} [L_{\boldsymbol{\nu}'r_{1}i},\phi_{0}] \otimes  [L_{\boldsymbol{\nu}''r_{2}i},\phi_{0}] \subset \mi(\nu,\omega^{1})\otimes_{\mo'}\mv(\omega^{2})+\mv(\omega^{1})\otimes\mi(\nu,\omega^{2}).$$

Now we prove that $\Delta$ is surjective. By induction on the dimension of the flag type $ \boldsymbol{\nu}^{1}$, we prove that  $[L_{\boldsymbol{\nu}^{1}\boldsymbol{e}^{1}},\phi_{0}] \otimes [L_{\boldsymbol{\nu}^{2}\boldsymbol{e}^{2}},\phi_{0}]$ belongs to the image of $\Delta$. 

Since $\Delta([L_{\boldsymbol{e}^{1}\boldsymbol{\nu}^{2}\boldsymbol{e}^{2}},\phi_{0}]) \in [L_{\boldsymbol{e}^{1}},\phi_{0}] \otimes [L_{\boldsymbol{\nu}^{2}\boldsymbol{e}^{2}},\phi_{0}]+ \mi({\omega^{1}}) \otimes \mv(\omega^{2})$, the element $[L_{\boldsymbol{e}^{1}},\phi_{0}] \otimes [L_{\boldsymbol{\nu}^{2}\boldsymbol{e}^{2}},\phi_{0}]$ belongs to the image of $\Delta$.

In general, by \cite[Lemma 12.3.3]{lusztig2010introduction} $$\Delta([L_{\boldsymbol{\nu}^{1}\boldsymbol{e}^{1}\boldsymbol{\nu}^{2}\boldsymbol{e}^{2}},\phi_{0}])= \sum\limits_{\boldsymbol{\nu}'+\boldsymbol{\nu}''=\boldsymbol{\nu}^{1},\boldsymbol{\tau}'+\boldsymbol{\tau}''=\boldsymbol{\nu}^{2}} v^{N(\boldsymbol{\nu}',\boldsymbol{\nu}'',\boldsymbol{\tau}',\boldsymbol{\tau}'')}[L_{\boldsymbol{\nu}'\boldsymbol{e}^{1}\boldsymbol{\tau}'},\phi_{0}] \otimes [L_{\boldsymbol{\nu}'' \boldsymbol{\tau}''\boldsymbol{e}^{2}},\phi_{0}].$$

If $\boldsymbol{\tau}'$ is nonempty, the element $[L_{\boldsymbol{\nu}'\boldsymbol{e}^{1}\boldsymbol{\tau}'},\phi_{0}] \otimes [L_{\boldsymbol{\nu}'' \boldsymbol{\tau}''\boldsymbol{e}^{2}},\phi_{0}]$ is zero in $\mk(\omega^{1}) \otimes_{\mathcal{O}'} \mk(\omega^{2})$. If $\boldsymbol{\tau}'$ is empty and $\boldsymbol{\nu}' \neq \boldsymbol{\nu}^{1}$, then by induction hypothesis we can assume $[L_{\boldsymbol{\nu}'\boldsymbol{e}^{1}},\phi_{0}] \otimes [L_{\boldsymbol{\nu}'' \boldsymbol{\nu}^{2}\boldsymbol{e}^{2}},\phi_{0}] $ belongs to the image of $\Delta$. If $\boldsymbol{\tau}'$ is empty and $\boldsymbol{\nu}' = \boldsymbol{\nu}^{1}$,  $[L_{\boldsymbol{\nu}'\boldsymbol{e}^{1}\boldsymbol{\tau}'},\phi_{0}] \otimes [L_{\boldsymbol{\nu}'' \boldsymbol{\tau}''\boldsymbol{e}^{2}},\phi_{0}] = [L_{\boldsymbol{\nu}^{1}\boldsymbol{e}^{1}},\phi_{0}] \otimes [L_{\boldsymbol{\nu}^{2}\boldsymbol{e}^{2}},\phi_{0}].$ Hence $\Delta([L_{\boldsymbol{\nu}^{1}\boldsymbol{e}^{1}\boldsymbol{\nu}^{2}\boldsymbol{e}^{2}},\phi_{0}]) \in [L_{\boldsymbol{\nu}^{1}\boldsymbol{e}^{1}},\phi_{0}] \otimes [L_{\boldsymbol{\nu}^{2}\boldsymbol{e}^{2}},\phi_{0}] + {\rm{Im}} (\Delta).$

Since these $[L_{\boldsymbol{\nu}^{1}\boldsymbol{e}^{1}},\phi_{0}] \otimes [L_{\boldsymbol{\nu}^{2}\boldsymbol{e}^{2}},\phi_{0}]$ form a spanning set of $\mk(\omega^{1}) \otimes_{\mathcal{O}'} \mk(\omega^{2})$, and the proof is complete.
\end{proof}

For general $N$, we also consider the functor $\Delta_k: \mk(\omega^{\bullet})\rightarrow \mk(\omega^{k-\bullet})\otimes_{\mo'} \mk(\omega^k)$ induced by $\oplus_{\bfV_1,\bfV_2}\operatorname{Res}_{\bfV_1\oplus\bfW^{k-\bullet},\bfV_2\oplus\bfW^k}^{\bfV\oplus\bfW^{\bullet}}$, where $\bfW^{k-\bullet}=(\bfW^1,\cdots,\bfW^{k-1},\bfW^{k+1},\cdots,\bfW^n)$\\ and $\omega^{k-\bullet}=(\omega^1,\cdots,\omega^{k-1},\omega^{k+1},\cdots,\omega^n)$. The proof above adapts directly to $\Delta_N$, and we now indicate the corresponding argument for a general $\Delta_k$.
\begin{proposition}
  $\Delta_k: \mk(\omega^{\bullet})\rightarrow \mk(\omega^{k-\bullet})\otimes _{\mathcal{O}'}\mk(\omega^k)$ is a well-defined $\mathcal{O}'$-module surjective morphism for $1\leq k\leq n$.
\end{proposition}
\begin{proof}
  The same argument as in Proposition \ref{right} shows that this map is well-defined.

	It remains only to prove that $\Delta_k: _{\mathcal{O}'}\mk(\omega^{\bullet})\rightarrow _{\mathcal{O}'}\mk(\omega^{k-\bullet})\otimes _{\mathcal{O}'}\mk(\omega^k)$ is surjective. We show that $[(\ml_{\boldsymbol{\nu_1}\boldsymbol{d_1}\cdots\boldsymbol{\nu_{k-1}}\boldsymbol{d_{k-1}}\boldsymbol{\nu_{k+1}}\boldsymbol{d_{k+1}}\cdots\boldsymbol{\nu_n}\boldsymbol{d_n}},\phi_0)]\boxtimes[(\ml_{\boldsymbol{\nu_k}\boldsymbol{d_k}},\phi_0)]$ belongs to the image of $\Delta_k$, for any $\boldsymbol{\nu_i}$.

  Now fix $\nu=\dim \bfV$. First, we show that any element of the form $[(\ml_{\boldsymbol{d_1}\cdots\boldsymbol{d_{k-1}}\boldsymbol{\nu_{k+1}}\boldsymbol{d_{k+1}}\cdots\boldsymbol{\nu_n}\boldsymbol{d_n}},\phi_0)]\boxtimes[(\ml_{\boldsymbol{\nu_k}\boldsymbol{d_k}},\phi_0)]$ belongs to the image.

  By considering $\Delta_k([(\ml_{\boldsymbol{d_1}\cdots\boldsymbol{d_{k-1}}\boldsymbol{\nu_{k}}\boldsymbol{d_{k}}\boldsymbol{\nu_{k+1}}\boldsymbol{d_{k+1}}\cdots\boldsymbol{\nu_n}\boldsymbol{d_n}},\phi_0)])$ and applying the same argument as in Proposition \ref{right}, we obtain the desired assertion.

  Then, by induction on $\sum_{i=1}^{k-1}\dim \boldsymbol{\nu_i}$, we prove that $[(\ml_{\boldsymbol{\nu_1}\boldsymbol{d_1}\cdots\boldsymbol{\nu_{k-1}}\boldsymbol{d_{k-1}}\boldsymbol{\nu_{k+1}}\boldsymbol{d_{k+1}}\cdots\boldsymbol{\nu_n}\boldsymbol{d_n}},\phi_0)]\boxtimes[(\ml_{\boldsymbol{\nu_k}\boldsymbol{d_k}},\phi_0)]$ belongs to the image of $\Delta_k$.

  The assertion is clear for $\sum_{i=1}^{k-1}\dim \boldsymbol{\nu_i}=0$. Assume that the assertion holds for $\sum_{i=1}^{k-1}\dim \boldsymbol{\nu_i}=r$, We prove the induction step for $\sum_{i=1}^{k-1}\dim \boldsymbol{\nu_i}=r+1$ by considering
  \begin{align*}
    &\Delta_k([(\ml_{\boldsymbol{\nu_1}\boldsymbol{d_1}\cdots\boldsymbol{\nu_{k-1}}\boldsymbol{d_{k-1}}\boldsymbol{d_k}\boldsymbol{\nu_{k+1}}\boldsymbol{d_{k+1}}\cdots\boldsymbol{\nu_n}\boldsymbol{d_n}},\phi_0)])=v^M[(\ml_{\boldsymbol{\nu_1}\boldsymbol{d_1}\cdots\boldsymbol{\nu_{k-1}}\boldsymbol{d_{k-1}}\boldsymbol{\nu_{k+1}}\boldsymbol{d_{k+1}}\cdots\boldsymbol{\nu_n}\boldsymbol{d_n}},\phi_0)]\boxtimes[(\ml_{\boldsymbol{d_k}},\phi_0)]\\
    &+\sum_{\sum_{i=1}^{k-1}\dim \boldsymbol{\nu'_i}<r+1}v^{M_{\boldsymbol{\nu'}}}[(\ml_{\boldsymbol{\nu'_1}\boldsymbol{d_1}\cdots\boldsymbol{\nu'_{k-1}}\boldsymbol{d_{k-1}}\boldsymbol{\nu_{k+1}}\boldsymbol{d_{k+1}}\cdots\boldsymbol{\nu_n}\boldsymbol{d_n}},\phi_0)]\boxtimes[(\ml_{\boldsymbol{\nu'_k}\boldsymbol{d_k}},\phi_0)]
  \end{align*}
  By induction, $[(\ml_{\boldsymbol{\nu_1}\boldsymbol{d_1}\cdots\boldsymbol{\nu_{k-1}}\boldsymbol{d_{k-1}}\boldsymbol{\nu_{k+1}}\boldsymbol{d_{k+1}}\cdots\boldsymbol{\nu_n}\boldsymbol{d_n}},\phi_0)]\boxtimes[(\ml_{\boldsymbol{d_k}},\phi_0)]$ belongs to the image.

  We then use induction on $\dim \boldsymbol{\nu_k}$. Since the proposition has been proved when $\dim\boldsymbol{\nu_k}=0$, it remains to show that the case $\dim\boldsymbol{\nu_k}=s+1$ follows from the case $\dim\boldsymbol{\nu_k}=s$.

  Consider \begin{align*}
    &\Delta_k([(\ml_{\boldsymbol{\nu_1}\boldsymbol{d_1}\cdots\boldsymbol{\nu_{k-1}}\boldsymbol{d_{k-1}}\boldsymbol{\nu_k}\boldsymbol{d_k}\boldsymbol{\nu_{k+1}}\boldsymbol{d_{k+1}}\cdots\boldsymbol{\nu_n}\boldsymbol{d_n}},\phi_0)])=v^M[(\ml_{\boldsymbol{\nu_1}\boldsymbol{d_1}\cdots\boldsymbol{\nu_{k-1}}\boldsymbol{d_{k-1}}\boldsymbol{\nu_{k+1}}\boldsymbol{d_{k+1}}\cdots\boldsymbol{\nu_n}\boldsymbol{d_n}},\phi_0)]\boxtimes[(\ml_{\boldsymbol{\nu_k}\boldsymbol{d_k}},\phi_0)]\\
    &+\sum_{\sum_{i=1}^{k-1}\dim \boldsymbol{\nu'_i}<r+1}v^{M_{\boldsymbol{\nu'}}}[(\ml_{\boldsymbol{\nu'_1}\boldsymbol{d_1}\cdots\boldsymbol{\nu'_{k-1}}\boldsymbol{d_{k-1}}\boldsymbol{\nu_{k+1}}\boldsymbol{d_{k+1}}\cdots\boldsymbol{\nu_n}\boldsymbol{d_n}},\phi_0)]\boxtimes[(\ml_{\boldsymbol{\nu'_k}\boldsymbol{d_k}},\phi_0)]\\
    &+\sum_{\dim\boldsymbol{\nu''_k}<\dim\boldsymbol{\nu_k}}v^{M_{\boldsymbol{\nu''}}}[(\ml_{\boldsymbol{\nu_1}\boldsymbol{d_1}\cdots\boldsymbol{\nu_{k-1}}\boldsymbol{d_{k-1}}\boldsymbol{\nu''_{k+1}}\boldsymbol{d_{k+1}}\boldsymbol{\nu_{k+1}}\boldsymbol{d_{k+2}}\cdots\boldsymbol{\nu_n}\boldsymbol{d_n}},\phi_0)]\boxtimes[(\ml_{\boldsymbol{\nu''_k}\boldsymbol{d_k}},\phi_0)],
  \end{align*}
  By induction, it follows that $[(\ml_{\boldsymbol{\nu_1}\boldsymbol{d_1}\cdots\boldsymbol{\nu_{k-1}}\boldsymbol{d_{k-1}}\boldsymbol{\nu_{k+1}}\boldsymbol{d_{k+1}}\cdots\boldsymbol{\nu_n}\boldsymbol{d_n}},\phi_0)]\boxtimes[(\ml_{\boldsymbol{\nu_k}\boldsymbol{d_k}},\phi_0)]$ belongs to the image, and the proof is complete.
\end{proof}
Let $\cdot$ be the $\mo'$-linear map induced by the functor $\bigoplus\limits_{\bfV^{1},\bfV^{2}}\mathbf{Ind}^{\bfV\oplus\bfW^{\bullet}}_{\bfV^{1}\oplus\bfW^{k-\bullet}, \bfV^{2}\oplus\bfW^{k}}$,
$$ \cdot: \mv(\omega^{k-\bullet}) \otimes_{\mo'} \mv(\omega^{k}) \rightarrow \mv(\omega^{\bullet}). $$
It induces a well-defined surjection, still denoted by $\cdot$, $$ \cdot: \mv(\omega^{k-\bullet}) \otimes_{\mo'} \mk(\omega^{k}) \rightarrow \mk(\omega^{\bullet}). $$
\begin{proposition}\label{left}
	Regarding $\mk(\omega^{N-\bullet})$ as the $\mo'$-submodule of $\mv(\omega^{N-\bullet})$ generated by $[(L,\phi)]$ where $L$ is simple perverse sheaf and $L\notin \mn$, the map $\cdot$ restricts to a surjection $$ \cdot: \mk(\omega^{N-\bullet}) \otimes_{\mo'} \mk(\omega^{N}) \rightarrow \mk(\omega^{\bullet}). $$
\end{proposition}
\begin{proof}
	It suffices to show that $\mi(\omega^{N-\bullet})\cdot \mk(\omega^{N})\subseteq \mk(\omega^{N-\bullet})\cdot \mk(\omega^{N})$.

 For any simple perverse sheaf $L\in \mn_{\bfV},$ there exists $i\in I,$ $L\in \mf_{\Omega_i,\Omega}(\mn_{\bfV,i}).$ Given $[(L,\phi)]\in \mi(\omega^{N-\bullet})$, there exist $[(L',\phi')]$, $a^*(L')=L'$ and $L'\in \mq_{\bfV,\bfW^{N-\bullet}}$, $i\in I$ and $r>0$ such that $\mf_{\underline{i}}^{(r),\vee}([(L',\phi')])=[(L,\phi)]+\sum\limits_{L'',t_i^*(L'')>t_i*(L)}c_{L''}[(L'',\phi'')]$. 	
 
 We prove by induction on $\nu$ and descending induction of $t_i^*$. By associativity of the induction functor, it follows that $$\mf_{\underline{i}}^{(r),\vee}([(L',\phi')] )\cdot \mk(\omega^{N})=[(L',\phi')] )\cdot \mf_{\underline{i}}^{(r)} \mk(\omega^{N}) \subseteq [(L',\phi')]\cdot \mk(\omega^{N}).$$
 By the induction hypothesis, we can assume that
 $$[(L',\phi')]\cdot \mk(\omega^{N}) \subseteq \mk(\omega^{N-\bullet})\cdot \mk(\omega^{N})$$ and 
 $$ [(L'',\phi'')] )\cdot \mk(\omega^{N}) \subseteq \mk(\omega^{N-\bullet})\cdot \mk(\omega^{N}).$$
 This implies that $[(L,\phi)]\cdot \mk(\omega^{N}) \subseteq \mk(\omega^{N-\bullet})\cdot \mk(\omega^{N}).$
\end{proof}

By Proposition \ref{left}, and comparing the ranks of $\mk(\omega^{\bullet}),$ $\mk(\omega^{k-\bullet})\otimes _{\mo'}\mk(\omega^{k})$ and $\mk(\omega^{1})\otimes_{\mo'} \cdots \otimes_{\mo'} \mk(\omega^N)$, we have the following corollary.
\begin{corollary}\label{Oiso}
	The maps $\Delta_k$ are isomorphisms of free $\mathcal{O}'$-modules from $\mk(\omega^{\bullet})$ to $\mk(\omega^{k-\bullet})\otimes _{\mo'}\mk(\omega^{k})$. In particular, when $N=k=2$, $\Delta_k=\Delta$ is an isomorphism of free $\mathcal{O}'$-modules from $\mk(\omega^{\bullet})$ to $\mk(\omega^{1})\otimes _{\mo'}\mk(\omega^{2})$.
\end{corollary}

Let $\mathbf{P}:\mk(\omega^{1})\otimes _{\mo'}\mk(\omega^{2}) \rightarrow \mk(\omega^{2})\otimes _{\mo'}\mk(\omega^{1}) $ be the swapping map, we define $\Delta'=\mathbf{P}(\mathbf{D}\times \mathbf{D})\Delta \mathbf{D} :\mk(\omega^{\bullet})\rightarrow\mk(\omega^{2})\otimes _{\mo'}\mk(\omega^{1})$, $\Delta_k'=\mathbf{P}(\mathbf{D}\times \mathbf{D})\Delta_k \mathbf{D} :\mk(\omega^{\bullet})\rightarrow\mk(\omega^{k\bullet})\otimes _{\mo'}\mk(\omega^{k}).$ We have the following proposition, which determines the module structure of $\mk(\omega^{\bullet})$.
\begin{proposition}\label{OUmodiso}
	The morphism $\Delta_k'$ is an isomorphism of $ _{\mo'}\bfU$-modules. In particular, when $N=k=2$ the Grothendieck group $\mk(\omega^{\bullet})$ is canonically isomorphic to the tensor product $_{\mo'}L(\lambda^{2}) \otimes_{\mo'} {_{\mo'}L}(\lambda^{1})$.
\end{proposition}
\begin{proof}
	By Corollary \ref{Oiso}, the maps $\Delta_k'$ are isomorphisms of $\mo'$-modules. It remains to prove that $\Delta'$ is $_{\mo'}\mathbf{U}$-linear. Since $\mf^{(r)}_{\underline{i}}$ is isomorphic to the induction functor, 
by the main result in \cite{fang2024paritylusztigsrestrictionfunctor}, we have that \begin{align*}
  \oplus_{\bfV_1\oplus\bfV_2\cong \bfV}\operatorname{Res}^{\bfV\oplus\bfW^{\bullet}}_{\bfV_1\oplus\bfW^{k-\bullet},\bfV_2\oplus\bfW_k}\mf_{\underline{i}}&=(\mf_{\underline{i}}\times \operatorname{id})\oplus_{\bfV_1\oplus\bfV_2\cong \bfV-\bfV_{\underline{i}}}\operatorname{Res}^{\bfV\oplus\bfW^{\bullet}}_{\bfV_1\oplus\bfW^{k-\bullet},\bfV_2\oplus\bfW_k}\\
  &\oplus(\operatorname{id}\times \mf_i)\oplus_{\bfV_1\oplus\bfV_2\cong \bfV-\bfV_{\underline{i}}}\operatorname{Res}^{\bfV\oplus\bfW^{\bullet}}_{\bfV_1\oplus\bfW^{k-\bullet},\bfV_2\oplus\bfW_k}[-s_i\langle i',\sum_{j'\in I'}\nu^1_{\underline{j}}j'\rangle+s_i\sum_{j\not=k}\omega^j_{i'}],
\end{align*}
modulo traceless objects.
We can see that $\Delta_k'$ commutes with $\mf^{(r)}_{\underline{i}}$. By an argument similar to that in Theorem \ref{4.6}, it follows that $\mk(\omega^{\bullet})$ is an integrable module. In particular, the map $\Delta'$ is a $_{\mo'}\mathbf{U}^{-}$-linear isomorphism of two integrable modules, hence it must be an isomorphism of $_{\mo'}\mathbf{U}$-modules.
\end{proof}

\begin{remark}
	Indeed, one can use the commutation relations of $\mf^{(r)}_{\underline{i}}$ and $\me^{(r)}_{\underline{i}}$ and the reason that both of $\mk(\omega^{\bullet})$ and the tensor products of highest weight modules are integral modules to check that $\mf^{(r)}_{\underline{i}}\Delta_k' = \Delta_k'\mf^{(r)}_{\underline{i}}$ and $\me^{(r)}_{\underline{i}}\Delta_k' = \Delta_k'\me^{(r)}_{\underline{i}}$ hold for any $[(L_{\boldsymbol{\nu}^{1}\boldsymbol{e}^{1}\cdots\boldsymbol{\nu}^{n}\boldsymbol{e}^{n}},\phi_{0})]$. This also proves $\Delta_k'$ is a $_{\mo'}\mathbf{U}$-linear isomorphism.
\end{remark}

\begin{corollary}
  The composition of $(\operatorname{id}\boxtimes\cdots\operatorname{id}\boxtimes\Delta_2')\cdots(\operatorname{id}\boxtimes\Delta_{n-1}')\circ \Delta'_n$ gives an isomorphism between $_{\mo'}\bfU-$modules $\mk(\omega^{\bullet})$ and $\mk(\omega^n)\otimes_{\mo'}\cdots\otimes_{\mo'}\mk(\omega^1)$.
\end{corollary}

	We can swap $\bfW^{1},\bfW^{2}$ and consider the linear map $\Delta^{',\vee}$ defined by $P \bigoplus\limits_{\bfV^{1},\bfV^{2}} \mathbf{D}\mathbf{Res}^{\bfV\oplus\bfW^{\bullet}}_{\bfV^{1}\oplus\bfW^{2}, \bfV^{2}\oplus\bfW^{1}} \mathbf{D}$. By a similar proof, we have the following proposition.
\begin{proposition}\label{Oiso2}
	 The morphism $\Delta^{',\vee}$ is an isomorphism of $_{\mo'}\mathbf{U}$-modules, $$\Delta^{',\vee}:\mk(\omega^{\bullet}) \rightarrow \mk(\omega^{1})\otimes _{\mo'}\mk(\omega^{2}).$$
\end{proposition}

\subsection{Geometric pairing}
From this section, we only assume $N=2$.
Given $(L_{1},\phi_{1})$, $(L_{2},\phi_{2})$ in $\widetilde{\mathcal{Q}_{\mathbf{V},\mathbf{W}^{\bullet}}/\mathcal{N}_{\mathbf{V}}} $, $\phi_{1}$ and $\phi_{2}$ induce a linear map 
$$ L_{\phi_{1},\phi_{2}}: \Hom_{ \mathcal{D}^{b}_{G_{\bfV}}(\bfE_{\bfV, \bfW^{\bullet},\Omega})/\mn_{\bfV} } (a^{\ast} \mathbf{D}L_{1},a^{\ast} L_{2}) \longrightarrow \Hom_{ \mathcal{D}^{b}_{G_{\bfV}}(\bfE_{\bfV, \bfW^{\bullet},\Omega})/\mn_{\bfV} } ( \mathbf{D}L_{1}, L_{2}) $$
$$ f \mapsto \phi_{2} \circ f \circ \mathbf{D}(\phi_{1}), $$
then $\phi_{1}$ and $\phi_{2}$ induce an endomorphism $a_{\phi_{1},\phi_{2}}$ of $\Hom_{ \mathcal{D}^{b}_{G_{\bfV}}(\bfE_{\bfV, \bfW^{\bullet},\Omega})/\mn_{\bfV} } ( \mathbf{D}L_{1}, L_{2})$ by \begin{equation*}
	\begin{split}
		a_{\phi_{1},\phi_{2}}:&\Hom_{ \mathcal{D}^{b}_{G_{\bfV}}(\bfE_{\bfV, \bfW^{\bullet},\Omega})/\mn_{\bfV} } ( \mathbf{D}L_{1}, L_{2}) \cong \Hom_{ \mathcal{D}^{b}_{G_{\bfV}}(\bfE_{\bfV, \bfW^{\bullet},\Omega})/\mn_{\bfV} } (a^{\ast} \mathbf{D}L_{1},a^{\ast} L_{2})\\
		& \xrightarrow{ L_{\phi_{1},\phi_{2}}} \Hom_{ \mathcal{D}^{b}_{G_{\bfV}}(\bfE_{\bfV, \bfW^{\bullet},\Omega})/\mn_{\bfV} } ( \mathbf{D}L_{1}, L_{2}).
	\end{split}
\end{equation*}
Similarly, $\phi_{1}$ and $\phi_{2}$ also induce linear maps on $\Ext^{n}_{ \mathcal{D}^{b}_{G_{\bfV}}(\bfE_{\bfV, \bfW^{\bullet},\Omega})/\mn_{\bfV} } ( \mathbf{D}L_{1}, L_{2})$ for any $n$, we still denote these endomorphisms by $a_{\phi_{1},\phi_{2}}$.

\begin{definition}
	Define an $\mathcal{O'}$-bilinear form $(-,-)$ on $ \mk(\nu,\omega^{\bullet}) $ by 
	$$ ([(L_{1},\phi_{1})],[(L_{2},\phi_{2})])^{\lambda^{\bullet}}= \sum\limits_{n} tr (a_{\phi_{1},\phi_{2}},\Ext^{n}_{ \mathcal{D}^{b}_{G_{\bfV}}(\bfE_{\bfV, \bfW^{\bullet},\Omega})/\mn_{\bfV} } ( \mathbf{D}L_{1}, L_{2}))v^{-n}. $$
	We also extend $(-,-)^{\lambda^{\bullet}}$ to $(-,-)^{\lambda}:\mk(\omega^{\bullet})\otimes \mk(\omega^{\bullet}) \rightarrow \mathcal{O}'$ by setting $(\mk(\nu,\omega^{\bullet}),\mk(\nu',\omega^{\bullet}) )^{\lambda^{\bullet}}=0$ for $\nu \neq \nu'$. 
\end{definition}

Observe that if $(L_{1},\phi_{1})$ or $(L_{2},\phi_{2})$ is traceless, the map $a_{\phi_{1},\phi_{2}}$ acts by permutation and has trace zero, so $(-,-)^{\lambda^{\bullet}}$ is well-defined. In particular, for $m,k\in \bbZ$, $$(\zeta^k[(L_{1},\phi_{1})],\zeta^m[(L_{2},\phi_{2})])^{\lambda^{\bullet}}=\zeta^{m+k}([(L_{1},\phi_{1})],[(L_{2},\phi_{2})])^{\lambda^{\bullet}}.$$

\begin{proposition}\label{bilinear}
	The $\mathcal{A}$-bilinear form $(-,-)^{\lambda^{\bullet}}$ is contravariant with respect to $E_{i'}$ and $F_{i'}$ for any $i' \in I'$,
	$$ ( F_{i'} x,y )^{\lambda}=(x,v_{i'}\tilde{K}_{-i'}E_{i'}y)^{\lambda}.$$
	Moreover, for any $(L_{1},\phi_{1})$ and $(L_{2},\phi_{2})$ satisfying the condition in Proposition \ref{4.5}, we have\\
	(1)If $[(L_{1},\phi_{1})] \neq \mp[(L_{2},\phi_{2})]$, then $([(L_{1},\phi_{1})],[(L_{2},\phi_{2})])^{\lambda^{\bullet}} \in v^{-1} \mathcal{O}[[v^{-1}]]$ ;\\
	(2)If $[(L_{1},\phi_{1})] = [(L_{2},\phi_{2})]$, then $([(L_{1},\phi_{1})],[(L_{2},\phi_{2})])^{\lambda^{\bullet}} \in 1+v^{-1} \mathcal{O}[[v^{-1}]]$.
\end{proposition}
\begin{proof}
	Since $(L_{1},\phi_{1})$ and $(L_{2},\phi_{2})$ satisfying the condition in Proposition \ref{4.5} are self-dual, almost orthogonality follows from the perverse $t$-structure of $\mathcal{D}^{b}_{G_{\mathbf{V}}}(\mathbf{E}_{\mathbf{V},\mathbf{W}^{\bullet},\Omega})/ \mathcal{N}_{\mathbf{V}}$. By the proof of \cite[Proposition 3.31]{fang2023lusztigsheavesintegrablehighest}, the functor $\mf_{i}$ is left adjoint to $\me_{i}\mk_{i}[-1]$, so $\mf_{\underline{i}}$ is left adjoint to $\me_{\underline{i}}\mk_{\underline{i}}[-s_{i}]$. By definition, $(-,-)^{\lambda^{\bullet}}$ is contravariant with respect to $E_{i'}$ and $F_{i'}$.
\end{proof}

\begin{proposition}
	When $N=1$ and $\boldsymbol{\nu},\boldsymbol{\nu'} \in \mathcal{S}'$, we have $$([(L_{\boldsymbol{\nu}\boldsymbol{e}},\phi_{0}) ], [ (L_{\boldsymbol{\nu'}\boldsymbol{e}},\phi_{0})])^{\lambda} \in \mathbb{Z}((v)) \cap \mathbb{Q}(v).$$
\end{proposition}
\begin{proof}
	If one of $\boldsymbol{\nu},\boldsymbol{\nu'} \in \mathcal{S}'$ is empty, the statement holds trivially. Using the contravariant property, we can prove the statement by induction on the length of $\boldsymbol{\nu},\boldsymbol{\nu'}$.
\end{proof}

\begin{proposition}\label{preserve}
	When $N=2$, $\lambda^{\bullet}=(\lambda^1,\lambda^2)$, define a bilinear form on $\mk(\lambda_2)\otimes\mk(\lambda_1)$ by $$(-,-)^{\otimes}:=(-,-)^{\lambda_2}\cdot(-,-)^{\lambda_1},$$ then the isomorphism $\Delta'$ preserves the bilinear form, i.e. $(x,y)^{\lambda^{\bullet}}=(\Delta'(x),\Delta'(y))^{\otimes}$.
\end{proposition}
\begin{proof}
	Since both $(-,-)^{\otimes}$ and $(-,-)^{\lambda^{\bullet}}$ have the same contravariant property, we only need to prove the statement for a set of $\mathbf{U}^{-}$- generators.
	
	We note that
	$\Delta'([(L_{\boldsymbol{e^{1}}\boldsymbol{\nu^{2}}\boldsymbol{e^{1}}},\phi_0)])=[(L_{\boldsymbol{\nu^{2}}\boldsymbol{e^{1}}},\phi_0)]\otimes [(L_{\boldsymbol{e^{1}}},\phi_0)].$ Since the map $i_1:\bfE_{\bfV,\bfW^{2},\Omega}\rightarrow \bfE_{\bfV,\bfW^{\bullet},\Omega}$ is a closed embedding, and $(i_1)_!(L_{\boldsymbol{\nu^2}\boldsymbol{e}^{2}},\phi_0)=(L_{\boldsymbol{e}^{1}\boldsymbol{\nu^2}\boldsymbol{e}^{2}},\phi_0)$, it follows that $$(\Delta'(L_{\boldsymbol{e}^{1}\boldsymbol{\nu_2}\boldsymbol{e}^{2}},\phi_0),\Delta'(L_{\boldsymbol{e}^{1}\boldsymbol{\nu'_2}\boldsymbol{e}^{2}},\phi_0))^{\otimes}=([(L_{\boldsymbol{e}^{1}\boldsymbol{\nu}^{2}\boldsymbol{e}^{2}},\phi_0)],[(L_{\boldsymbol{e}^{1}\boldsymbol{\nu'}^{2}\boldsymbol{e}^{2}},\phi_0)])^{\omega^{\bullet}}.$$ This completes the proof.
\end{proof}

Combining the propositions above, we obtain the following corollary.
\begin{corollary}
		When $N=2$ and $\boldsymbol{\nu}^{i},\boldsymbol{\nu'}^{i} \in \mathcal{S}'$ for $i=1,2$, we have $$([(L_{\boldsymbol{\nu}^{1}\boldsymbol{e}^{1}\boldsymbol{\nu}^{2}\boldsymbol{e}^{2}},\phi_{0}) ], [ (L_{\boldsymbol{\nu'}^{1}\boldsymbol{e}^{1}\boldsymbol{\nu'}^{2}\boldsymbol{e}^{2}},\phi_{0})])^{\lambda^{\bullet}} \in \mathbb{Z}((v)) \cap \mathbb{Q}(v).$$
\end{corollary}

\subsection{Signed bases}
In this subsection, we prove that those $\mathbf{D}$-invariant objects $(L,\phi)$ form the signed bases of irreducible integrable highest weight modules and their tensor products.

We recall that a subset $\mathcal{B}$ of a module $M$ is called a signed basis if there exists a basis $\mathcal{B}'$ such that $\mathcal{B}=\mathcal{B}' \cup -\mathcal{B}'$. We denote the signed basis of $L(\lambda)$ by $\mathcal{B}(\lambda)$.

\subsubsection{$N=1$ case}
In this subsection, we assume $N=1$ and construct an $\mathcal{A}$-signed basis of $_{\mathcal{A}}L(\lambda). $

\begin{proposition}\label{span}
	Let $_{\mathcal{A}}\mathcal{M}(\nu,\omega)$ be the $\mathcal{A}$-submodule of $\mk(\nu,\omega)$ spanned by those $[(L_{\boldsymbol{\nu}\boldsymbol{e}},\phi_{0})],\boldsymbol{\nu} \in \mathcal{S}'$, and $_{\mathcal{A}}\mathcal{K}(\nu,\omega)$ be the $\mathcal{A}$-submodule of $\mk(\nu,\omega)$ spanned by those $[(L,\phi)]$ in Proposition \ref{4.5}. Then $_{\mathcal{A}}\mathcal{M}(\nu,\omega)={ _{\mathcal{A}}\mathcal{K}(\nu,\omega)}.$
\end{proposition}
\begin{proof}
	(1) We first show that $_{\mathcal{A}}\mathcal{M}(\nu,\omega)$ contains $_{\mathcal{A}}\mathcal{K}(\nu,\omega).$ It suffices to show that any $[(L,\phi)]$ in Proposition \ref{4.5} belongs to $_{\mathcal{A}}\mathcal{M}(\nu,\omega)$. By an argument similar to that in Theorem 4.6, we can find $(K,\phi')$ satisfying the following equation $\mf^{(r)}_{\underline{i}}[(K,\phi')]=[(L,\phi)]+ \sum\limits_{t_{i}(L')>r} P_{L'} [(L',\phi_{L'})]$. For $P_{L'}\in \bbZ[\zeta][v,v^{-1}]$, it may be regarded as a rational polynomial in the variables $\zeta,v.$ 
    
	We claim that for any $t\geqslant 0$, the coefficient of $v^{t}$ in those $P_{L'}$ are integers. Otherwise, we take a maximal $t$ such that there exists a $P_{L'}$ such that its coefficient $c_{t}(L')$ of $v^{t}$ is not an integer.

	By induction on dimension vector of $\bfV$ and descending induction on $r$, we may assume $[(K,\phi')]$ and those $[(L'',\phi_{L''})] $ with $t_{i}(L'')>0$ belong to $\bigoplus\limits_{\nu \in \bbN[I]^{a}}{_{\mathcal{A}}\mathcal{M}(\nu,\omega)}$. Then $(\mf^{(r)}_{\underline{i}}[(K,\phi')],[(L',\phi_{L'})] )^{\lambda} \in \mathbb{Z}((v)) \cap \mathbb{Q}(v).$ 
	In particular, $\sum\limits_{t_{i}(L'')>r}{P_{L''}}( [(L'',\phi_{L''})] ,[(L',\phi_{L'})] )^{\lambda}+ ( [(L,\phi)] ,[(L',\phi_{L'})] )^{\lambda} \in \mathbb{Z}((v)) \cap \mathbb{Q}(v)$.
	
	We have $( [(L,\phi)] ,[(L',\phi_{L'})] )^{\lambda} \in v^{-1} \mathcal{O}[[v^{-1}]]$, then the coefficient of $v^{t}$ in this term must be zero. If $ [(L'',\phi_{L''})] \neq [(L',\phi_{L'})] $, then $( [(L'',\phi_{L''})] ,[(L',\phi_{L'})] )^{\lambda} \in v^{-1} \mathcal{O}[[v^{-1}]] \cap \mathbb{Z}((v)) = v^{-1} \mathbb{Z}[[v^{-1}]]$. By maximality of $t$, it follows that the coefficient of $v^{t}$ in $ {P_{L''}}( [(L'',\phi_{L''})] ,[(L',\phi_{L'})] )^{\lambda}$ are integers. Similarly, $( [(L',\phi_{L'})] ,[(L',\phi_{L'})] )^{\lambda} \in 1+ v^{-1} \mathbb{Z}[[v^{-1}]]$, hence it follows that the coefficient of $v^{t}$ in $ {P_{L'}}( [(L',\phi_{L'})] ,[(L',\phi_{L'})] )^{\lambda}$ is an integer plus $c_{t}(L')$. By considering the coefficient of $v^{t}$ in $(\mf^{(r)}_{\underline{i}}[(K,\phi')],[(L',\phi_{L'})] )^{\lambda}$, we obtain a contradiction. 
	
	Applying Verdier duality $\mathbf{D}$ to Equation (\ref{OZ}) gives $$\mf^{(r)}_{\underline{i}}[(K,\phi')]=[(L,\phi)]+ \sum\limits_{t_{i}(L'')>r} \bar{P}_{L''} [(L'',\phi_{L''})],$$ so $\bar{P}_{L''}= P_{L''}$ for any $L''$. In particular, the coefficients of $v^{t}$ and $v^{-t}$ in $P_{L''}$ are the same and $P_{L''} \in \mathcal{A}$ for any $L''$. Hence $[(L,\phi)]$ belongs to $_{\mathcal{A}}\mathcal{M}(\nu,\omega)$.
	
	(2) Now we show that $_{\mathcal{A}}\mathcal{M}(\nu,\omega)$ is contained in $_{\mathcal{A}}\mathcal{K}(\nu,\omega).$ Take $[(L_{\boldsymbol{\nu}\boldsymbol{e}},\phi_{0})]$, it can be $\mathcal{O}'-$spanned by those $[(L,\phi)]$, so we may assume 
	$$ [(L_{\boldsymbol{\nu}\boldsymbol{e}},\phi_{0})]=\sum\limits_{(L,\phi_{L})} P_{L} [(L,\phi_{L})],$$
	where $P_{L} \in \mathcal{O}'$ and $(L,\phi_{L}) $ runs over all pairs satisfying the conditions in Proposition \ref{4.5}.
	
	We claim that $P_{L} \in \mathcal{A}$. Otherwise, there exists $t'$ such that the coefficient of $v^{t'}$ in some $P_{L}$ is not an integer. We take the maximal $t'$ with this property, and assume that the coefficient $c_{t'}(L')$ of $v^{t'}$ in $P_{L'}$ is not an integer. Observe that the coefficient of $v^{t'}$ in
	$ ([(L_{\boldsymbol{\nu}\boldsymbol{e}},\phi_{0})], [(L',\phi_{L'})] )^{\lambda}$ belongs to $\mathbb{Z}$, but the coefficient of $v^{t'}$ in $\sum\limits_{(L,\phi_{L})}  {P_{L}} ([(L,\phi_{L})], [(L',\phi_{L'})] )^{\lambda} $ is $c_{t'}(L')$ plus an integer by the maximality of $t'$, then we obtain a contradiction.
	
	In conclusion, $_{\mathcal{A}}\mathcal{M}(\nu,\omega)={_{\mathcal{A}}\mathcal{K}(\nu,\omega)}.$
\end{proof}

Combining Proposition 4.10 with Theorem 4.6 yields the following theorem.
\begin{theorem}
	The $\mathcal{A}$-submodule ${_{\mathcal{A}}\mathcal{K}(\omega)}=\bigoplus\limits_{\nu \in \bbN[I]^{a}}{_{\mathcal{A}}\mathcal{K}(\nu,\omega)}$ is canonically isomorphic to $_{\mathcal{A}}L(\lambda).$ We denote this isomorphism by $\eta^{\lambda}$. Moreover, the images of those $[(L,\phi)]$ in Proposition \ref{4.5} form a signed $\mathcal{A}$-basis of $_{\mathcal{A}}L(\lambda)$, and this basis is almost orthogonal with respect to the contravariant form $(-,-)^{\lambda}$.
\end{theorem}

\begin{remark}
	There is a natural vector bundle $\pi_{\bfV, \bfW^{\bullet}} : \bfE_{\bfV, \bfW^{\bullet},\Omega} \rightarrow \bfE_{\bfV,\Omega}$, whose pullback $(\pi_{\bfV, \bfW^{\bullet}})^{\ast}[\sum\dim \bfV_i\dim \bfW_i]$ defines a functor $(\pi_{\bfV, \bfW^{\bullet}})^{\ast}: \mq_{\bfV} \rightarrow \mq_{\bfV, \bfW^{\bullet}}$. Compose $(\pi_{\bfV, \bfW^{\bullet}})^{\ast}$ with the localization functor $L: \mq_{\bfV, \bfW^{\bullet}} \rightarrow \mq_{\bfV, \bfW^{\bullet}}/\mn_{\bfV} $, we obtain an additive functor $L' :\mq_{\bfV} \rightarrow \mq_{\bfV, \bfW^{\bullet}}/\mn_{\bfV}$. Observe that $L'$ commutes with $a^{\ast}$, then $L'$ induces a functor $\tilde{L}':\widetilde{\mq_{\bfV}} \rightarrow \widetilde{ \mq_{\bfV, \bfW^{\bullet}}/\mn_{\bfV}}$. This functor categorifies the canonical quotient $ \pi:\mathbf{U}^{-} \cong M_{\lambda} \rightarrow L(\lambda)$, where $M_{\lambda}$ is the Verma module. In particular, if $b=[(L,\phi)]$ is an element in the signed basis of $\mathbf{U}^{-}$ constructed by Lusztig in \cite{lusztig2010introduction}, then $\pi(b)\neq 0$ if and only if $L$ is nonzero in $\mq_{\bfV, \bfW^{\bullet}}/\mn_{\bfV}$, equivalently, $[(L,\phi)]$ is an element in the signed basis of $L(\lambda)$.
\end{remark}

\subsubsection{$N=2$ case}
In this subsection, we assume $N=2$ and construct an $\mathcal{A}$-signed basis of $_{\mathcal{A}}L(\lambda^{2}) \otimes_{\mathcal{A}} {_{\mathcal{A}}L}(\lambda^{1}). $ Recall that the signed basis $\{b_{1} \diamond b_{2}| b_{1} \in \mathcal{B}(\lambda^{2}),b_{2} \in \mathcal{B}(\lambda^{1}) \}$ of $_{\mathcal{A}}L(\lambda^{2}) \otimes_{\mathcal{A}} {_{\mathcal{A}}L}(\lambda^{1})$ is defined in \cite{lusztig1992canonical} and \cite{bao2016canonical}.

\begin{proposition}\label{span2}
	Let $_{\mathcal{A}}\mathcal{M}(\nu,\omega^{\bullet})$ be the $\mathcal{A}$-submodule of $\mk(\nu,\omega^{\bullet})$ spanned by those $[(L_{\boldsymbol{\nu}^{1}\boldsymbol{e}^{1}\boldsymbol{\nu}^{2}\boldsymbol{e}^{2}},\phi_{0})]$ such that $\boldsymbol{\nu}^{1},\boldsymbol{\nu}^{2} \in \mathcal{S}'$, and $_{\mathcal{A}}\mathcal{K}(\nu,\omega^{\bullet})$ be the $\mathcal{A}$-submodule of $\mk(\nu,\omega^{\bullet})$ spanned by those $[(L,\phi)]$ in Proposition \ref{4.5}. Then $_{\mathcal{A}}\mathcal{M}(\nu,\omega^{\bullet})={ _{\mathcal{A}}\mathcal{K}(\nu,\omega^{\bullet})}.$
\end{proposition}

\begin{proof}
	Let $_{\mathcal{A}}\mathcal{M}'(\nu,\omega^{\bullet})$ be the $\mathcal{A}$-submodule of $\mv(\nu,\omega^{\bullet})$ spanned by those $[(L_{\boldsymbol{\nu}^{1}\boldsymbol{e}^{1}\boldsymbol{\nu}^{2}\boldsymbol{e}^{2}},\phi_{0})]$ such that $\boldsymbol{\nu}^{1},\boldsymbol{\nu}^{2} \in \mathcal{S}'$, and $_{\mathcal{A}}\mv(\nu,\omega^{\bullet})$ be the $\mathcal{A}$-submodule of $\mv(\nu,\omega^{\bullet})$ spanned by those $[(L,\phi)]$ with the same properties as Proposition \ref{4.5}. It suffices to show $_{\mathcal{A}}\mathcal{M}'(\nu,\omega^{\bullet})={ _{\mathcal{A}}\mathcal{V}(\nu,\omega^{\bullet})}.$
	
	(1) When $N=1,$ the proof of $_{\mathcal{A}}\mathcal{M}'(\nu,\omega^{\bullet})={ _{\mathcal{A}}\mathcal{V}(\nu,\omega^{\bullet})}$ is the same as $_{\mathcal{A}}\mathcal{M}(\nu,\omega^{\bullet})={ _{\mathcal{A}}\mathcal{K}(\nu,\omega^{\bullet})}$ using the inner product defined by Lusztig in \cite[12.2]{lusztig2010introduction}.

	(2) We first show the statement for $\bfW^{2}=0$. Denote $_{\mathcal{A}}\mathcal{M}'(\nu,\omega^{\bullet})$ by $_{\mathcal{A}}\mathcal{M}'(\nu,\omega^{1,0})$ and denote ${ _{\mathcal{A}}\mathcal{V}(\nu,\omega^{\bullet})}$ by ${ _{\mathcal{A}}\mathcal{V}(\nu,\omega^{1,0})}$. By an argument similar to that in Proposition \ref{span} and \cite[Proposition 12.6.3]{lusztig2010introduction}, we can prove $_{\mathcal{A}}\mathcal{M}'(\nu,\omega^{1,0})={ _{\mathcal{A}}\mathcal{V}(\nu,\omega^{1,0})}$ by induction on $\nu$ and descending induction on $t_{i}^{*}$.
	
	(3) For general $\bfW^{2}$, notice that $\pi_{\bfV,\bfW^{2}}^{\ast}$ induces isomorphisms of $\mathcal{A}$-modules $_{\mathcal{A}}\mathcal{M}'(\nu,\omega^{\bullet}) \cong {_{\mathcal{A}}\mathcal{M}'}(\nu,\omega^{1,0})$ and ${ _{\mathcal{A}}\mathcal{V}(\nu,\omega^{\bullet})} \cong { _{\mathcal{A}}\mathcal{V}(\nu,\omega^{1,0})}.$ We get a proof.
\end{proof}

The following proposition shows that $\mathbf{D}$ on $\mq_{\bfV'', \bfW^{\bullet}}/\mn_{\bfV}$ realizes the $\Psi$-involution of the tensor product $\mk(\omega^{2}) \otimes_{\mo'} \mk(\omega^{1})$. The definition of $\Psi$ can be seen in \cite[24.3.2]{lusztig2010introduction}. This involution is defined by $\Psi(m\otimes m')= \Theta (\bar{m}\otimes\bar{m}')$, where $\Theta$ is the universal $\mathcal{R}$-matrix. 
\begin{proposition}\label{psi}
 Let $\Psi': \mk(\omega^{2}) \otimes_{\mo'} \mk(\omega^{1}) \rightarrow \mk(\omega^{2}) \otimes_{\mo'} \mk(\omega^{1})$ be the isomorphism defined by $\Delta' \mathbf{D} \Delta'^{-1}$, then $\Psi=\Psi'$. 
\end{proposition}

\begin{proof}
	Since $\Psi$ and $\Psi'$ satisfy the equation $\overline{u}\Psi=\Psi u$ and $\overline{u}\Psi'=\Psi' u$ for $u\in \mathbf{U}^{-}$, we only need to check $\Psi=\Psi'$ holds for a set of $\mathbf{U}^{-}$-generators.

	By direct calculation, we have $\Psi'([( L_{\boldsymbol{\nu}\boldsymbol{e}^{2} },\phi_{0} ) ] \otimes [(L_{\boldsymbol{e}^{1}},\phi_{0} )] ) = \Delta' \mathbf{D} [( L_{\boldsymbol{e}^{1}\boldsymbol{\nu}\boldsymbol{e}^{2} },\phi_{0} ) ] = [( L_{\boldsymbol{\nu}\boldsymbol{e}^{2} },\phi_{0} ) ] \otimes [(L_{\boldsymbol{e}^{1}},\phi_{0} )]. $ 
	
	Since $\Theta ( F^{(r_{1})}_{i_{1}'} F^{(r_{2})}_{i_{2}'} \cdots F^{(r_{s})}_{i_{s}'} v_{\lambda^{2}} \otimes v_{\lambda^{1}} )= F^{(r_{1})}_{i_{1}'} F^{(r_{2})}_{i_{2}'} \cdots F^{(r_{s})}_{i_{s}'} v_{\lambda^{2}} \otimes v_{\lambda^{1}} $, it follows that $\Psi([( L_{\boldsymbol{\nu}\boldsymbol{e}^{2} },\phi_{0} ) ] \otimes [(L_{\boldsymbol{e}^{1}},\phi_{0} )] ) = [( L_{\boldsymbol{\nu}\boldsymbol{e}^{2} },\phi_{0} ) ] \otimes [(L_{\boldsymbol{e}^{1}},\phi_{0} )] , $ and the proof is complete.

\end{proof}

\begin{theorem}
	The $\mathcal{A}$-module $_{\mathcal{A}} \mk(\omega^{\bullet})$ is canonically isomorphic to the tensor product $_{\mathcal{A}}L(\lambda^{2}) \otimes_{\mathcal{A}} {_{\mathcal{A}}L}(\lambda^{1})$ via the composition of $\eta^{\lambda^{2}} \otimes\eta^{\lambda^{1}}$ and $ \Delta'$. Moreover, the images of $\{[(L,\phi_{L})]|\mathbf{D}(L,\phi_{L})=(L,\phi_{L}),\ L\in \mP_{\bfV,\bfW^{\bullet}},a^{\ast}L \cong L \},$ when the order of the automorphism $a$ is even and choose a fixed sign, is $\mathcal{B}_0\cup -\mathcal{B}_0,$ where $\mathcal{B}_0$ forms an $\mathcal{A}$-basis, which coincides with the signed basis $\{b_{1} \diamond b_{2}| b_{1} \in \mathcal{B}(\lambda^{2}),b_{2} \in \mathcal{B}(\lambda^{1}) \}$ of $_{\mathcal{A}}L(\lambda^{2}) \otimes_{\mathcal{A}} {_{\mathcal{A}}L}(\lambda^{1})$.
\end{theorem}
\begin{proof}
By \cite[13.2.4]{lusztig2010introduction}, the isomorphism $\Delta'$ restricts to an isomorphism of $\mathcal{A}$-modules, so ${_{\mathcal{A}} \mk}(\omega^{\bullet}) \cong {_{\mathcal{A}}L}(\lambda^{2}) \otimes_{\mathcal{A}} {_{\mathcal{A}}L}(\lambda^{1})$. By Proposition \ref{preserve} and \ref{bilinear}, it follows that $\{[(L,\phi_{L})]|\mathbf{D}(L,\phi_{L})=(L,\phi_{L}),\ L\in \mP_{\bfV,\bfW^{\bullet}},a^{\ast}L \cong L \}$ is almost orthogonal. By Proposition \ref{psi}, the image of $[(L,\phi_{L})]$ is also $\Psi$-invariant. By \cite[Lemma 14.2.2]{lusztig2010introduction}, this basis coincides with the signed basis $\{b_{1} \diamond b_{2}| b_{1} \in \mathcal{B}(\lambda^{2}),b_{2} \in \mathcal{B}(\lambda^{1}) \}$.
\end{proof}

\section{Applications}
In this section, we present several applications of our realization.
\subsection{Yang--Baxter equation}
We assume $N=3$ and fix three dominant weights $\lambda^{1},\lambda^{2}$ and $\lambda^{3}$. For any $\mk(\omega^{i}) \otimes_{\mo'} \mk(\omega^{j})$, we let $\mathbf{R}_{i,j}$ be the isomorphism defined by 
$$\mathbf{R}_{ij}: \mk(\omega^{i}) \otimes_{\mo'} \mk(\omega^{j}) \xrightarrow{ (\Delta^{',\vee})^{-1} } \mk(\omega^{i},\omega^{j} ) \xrightarrow{\Delta'} \mk(\omega^{j}) \otimes_{\mo'} \mk(\omega^{i}) ,$$
where $\mk(\omega^{i},\omega^{j} )$ is the Grothendieck group associated with the $2$-framed quiver.
These $\mathbf{R}_{ij}$ can naturally be regarded as isomorphisms between tensor products of the form $_{\mathcal{A}}\mk(\omega^i)\otimes _{\mathcal{A}}\mk(\omega^j)\otimes_{\mathcal{A}}\mk(\omega^k)$ such that $\{i,j,k\}=\{1,2,3\}$.

\begin{proposition}\label{Yang}
The following equality of $_{\mathcal{A}}\mathbf{U}$-isomorphisms holds $$\mathbf{R_{23}}\mathbf{R_{13}}\mathbf{R_{12}}=\mathbf{R_{12}}\mathbf{R_{13}}\mathbf{R_{23}}:_{\mathcal{A}}\mk(\omega^1)\otimes _{\mathcal{A}}\mk(\omega^2)\otimes_{\mathcal{A}}\mk(\omega^3)\rightarrow _{\mathcal{A}}\mk(\omega^3)\otimes _{\mathcal{A}}\mk(\omega^2)\otimes_{\mathcal{A}}\mk(\omega^1).$$
\end{proposition}
\begin{proof}
	Denote $_{\mathcal{A}}\mk(\omega^i)\otimes _{\mathcal{A}}\mk(\omega^j)\otimes_{\mathcal{A}}\mk(\omega^k)$ by $ijk$, and $_{\mathcal{A}}\mk(\omega^i,\omega^j)\otimes_{\mathcal{A}}\mk(\omega^k)$ by $(ij)k$. Let $\mk(\omega^{\bullet})$ be the Grothendieck group associated with the $3$-framed quiver, then we have the following commutative diagram.
\[
\xymatrix{
	&    & 123 \ar[dll]_{R_{12}} \ar[drr]^{R_{23}}  &     &  \\
	213 \ar[dd]_{R_{13}}	& (12)3 \ar[l] \ar[ur] &  & 1(23) \ar[ul] \ar[r] &132 \ar[dd]^{R_{13}} \\
	& 2(13) \ar[ul] \ar[dl] & \mk(\omega^{\bullet}) \ar[ul] \ar[l]\ar[dl] \ar[ur] \ar[r] \ar[dr] & (13)2 \ar[ur]\ar[dr] &  \\
	231 \ar[drr]_{R_{23}}	& (23)1 \ar[l] \ar[dr] &    & 3(12) \ar[dl] \ar[r] & 312 \ar[dll] \\
	&    & 321  &     &.
}
\]
All arrows in the diagram are induced by restriction functors, and the assertion follows from the coassociativity of the restriction functor.
\end{proof}

\subsection{Symmetrizable crystal structure arising from quiver variety}
In this subsection, we deduce a crystal structure on the set of $a$-fixed irreducible components of Nakajima's quiver varieties and show that it agrees with the crystal structure on the signed basis of $L(\lambda)$. This recovers one of the main results of \cite{SA}.

In this subsection, all varieties are over the field $\mathbb{C}$. Consider the quiver $Q = (I, H, \Omega)$ with an admissible automorphism $a$, and $I$-graded spaces $\bfW$ and $\bfV$, where their dimensions are given by vectors $v$ and $w$ as described above. Let $\bar{\Omega}$ be the set of arrows $\bar{h}$ that have the opposite orientation of $h \in \Omega$. 

\begin{definition}
	Forgetting the admissible automorphism $a$, we define $\bfM=\bfM_{\bfV,\bfW}$ to be the vector space 
	\[
	\bigoplus_{h \in \Omega}\Hom_{\bbC}(\bfV_{s(h)}, \bfV_{t(h)}) \oplus \bigoplus_{h \in \bar{\Omega}}\Hom_{\bbC}(\bfV_{s(h)}, \bfV_{t(h)}) \oplus \bigoplus_{i \in I}\Hom_{\bbC}(\bfV_i, \bfW_i) \oplus \bigoplus_{i \in I}\Hom_{\bbC}(\bfW_i, \bfV_i),
	\]
	with $\bfG = \bfG_{v} := \prod_{k \in I} \text{GL}(\bfV_k)$ acting on it naturally. The moment map is defined by $$\mu: \bfM \rightarrow \bigoplus_{k \in I}\Hom_{\bbC}(\bfV_k, \bfV_k).$$ For $B \in \bigoplus_{h \in \Omega}\Hom_{\bbC}(\bfV_{s(h)}, \bfV_{t(h)}) \oplus \bigoplus_{h \in \bar{\Omega}}\Hom_{\bbC}(\bfV_{s(h)}, \bfV_{t(h)})$, $j \in \bigoplus_{k \in I}\Hom_{\bbC}(\bfV_k, \bfW_k)$, and $i \in \bigoplus_{k \in I}\Hom_{\bbC}(\bfW_k, \bfV_k)$, the moment map can be denoted by:
	\[
	\mu(B, i, j) := \sum_{h \in \Omega \cup \bar{\Omega}} \varepsilon(h)B_hB_{\bar{h}} + ij,
	\]
	where $\varepsilon(h) = 1$ if $h \in \Omega$ and $\varepsilon(h) = -1$ if $h \in \bar{\Omega}$.
	
	The affine variety $\mu^{-1}(0)$ gives rise to the affine GIT quotient $\mm_0 = \mm_0(v, w) := \textbf{Spec} A(\mu^{-1}(0))^{\bfG}$. Fix the character $\chi$ of $\bfG$ given by $\chi(g) = \prod_{k \in I}\det(g_k)^{-1}.$ The quiver variety is defined by $\mm = \mm(v, w) := \textbf{Proj}\bigoplus_{n \geq 0} A(\mu^{-1}(0))^{\bfG, \chi^n}$. 
	
\end{definition} 

\begin{remark}
		When $\bf W=0$ , the variety $\mu^{-1}(0) \subset \bfM_{\bfV}$ is the moduli space of the preprojective algebra. Lusztig's nilpotent variety $\Lambda_{\bfV}$ is the subset of $\mu^{-1}(0) \subset \bfM_{\bfV}$ consisting of nilpotent representations of the preprojective algebra. (See \cite[Section 12]{L01}.)
\end{remark}

Lift the action of $\bfG$ to $\mu^{-1}(0) \times \bbC$ by $g(B, i, j, z) = (gBg^{-1}, gi, jg^{-1}, \chi(g)^{-1}z)$, then $\mu^{-1}(0)^s := \{(B, i, j) \in \mu^{-1}(0) \mid \overline{\bfG(B, i, j, z)} \cap (\mu^{-1}(0) \times \{0\}) = \emptyset, z \neq 0\}$. According to Nakajima, $\bfG$ acts freely on $\mu^{-1}(0)^s$, and the GIT quotient $\mu^{-1}(0)^s / \bfG$ coincides with the geometric points of $\mm$. There is a natural projective morphism $\pi: \mm \rightarrow \mm_0$, with $\pi^{-1}(0)$ being a Lagrangian subvariety denoted by $\mathfrak{L} = \mathfrak{L}(v, w)$, which is homotopy equivalent to $\mm$. 

Nakajima's quiver variety admits a stratification $$\mm(v, w) = \bigcup_{r \geq 0} \mm_{k, r}(v, w), k\in I,$$ where $$\mm_{k, r}(v, w) := \{(B, i, j) \in \mm(v, w) \mid \text{codim}(\text{im}\oplus_{h \in H, t(h) = k}B_h + \text{im} i_k) = r\}.$$ There is a natural smooth map, $p: \mm_{k, r}(v, w) \rightarrow \mm_{k, 0}(v - rk, w)$ whose fiber is a connected Grassmannian. This gives a bijection between the irreducible components of $\mathfrak{L}_{k, r}(v, w)$ and $\mathfrak{L}_{k, 0}(v - rk, w)$, where $\mathfrak{L}_{k, r}(v, w) = \mathfrak{L} \cap \mm_{k, r}(v, w)$ and $\mathfrak{L}_{k, r}(v, w)$ is not empty. For an irreducible component $X \in \mathfrak{L}_{k, r}(v, w)$, the closure $\bar{X}$ of $X$ is an irreducible component of $\mathfrak{L}(v, w)$ and we denote $t_k(\bar{X}) = r$. Hence $p$ induces a bijection by $X \mapsto \overline{p(X \cap \mm_{k, r}(v, w) )} $ and we denote this bijection by $$\rho_{k, r}: \{X \in \text{Irr}(\mathfrak{L}(v, w)) \mid t_k(X) = r\} \rightarrow \{X \in \text{Irr}(\mathfrak{L}(v - rk, w)) \mid t_k(X) = 0, \langle h_{k},wt(X) \rangle\geq 0\}.$$

The admissible automorphism $a$ acts naturally on $\mathfrak{L}(v,w)$ and induces a permutation, which is still denoted by $a$, on the set of its irreducible components. Let $\text{Irr}^{a}\mathfrak{L}(v,w)$ be the set of $a$-fixed irreducible components. For any irreducible $X \in \text{Irr}^{a}\mathfrak{L}(v,w)$, if $t_{k}(X)=r$, then $t_{k'}=r$ for any $k' \in \underline{k}$. Then the map $\rho_{\underline{k}, r}(X) = \prod\limits_{k \in \underline{k}}\rho_{k, r}(X)$ is well-defined. Since each $\rho_{k, r}$ is bijective, it follows that $\rho_{\underline{k}, r}(X)$ is also fixed by $a$. Hence $\rho_{\underline{k}, r}$ is a bijection between the corresponding sets of $a$-fixed irreducible components.
 
These maps $\rho_{\underline{k}, r}$ induce a crystal operator on the set $\bigcup\limits_{v}\text{Irr}^{a}\mathfrak{L}(v,w)$ in the following way: 
\begin{equation*}
	\begin{split}
		\tilde{e}_{\underline{k}}(X) = 
		\begin{cases} 
			\rho^{-1}_{\underline{k}, r-1}\rho_{\underline{k}, r}(X) & \text{if } t_{k}(X)=r>0,\\ 
			0 & \text{otherwise};
		\end{cases}
	\end{split}
\end{equation*}
and
\begin{equation*}
		\tilde{f}_{\underline{k}}(X) = 	\rho^{-1}_{\underline{k}, r+1}\rho_{\underline{k}, r}(X), \text{if } t_{k}(X)=r.
\end{equation*}
(For the detailed definition of the crystal operators and crystal structures, see \cite{crystal}. )

For $X \in \text{Irr}^{a}\mathfrak{L}(v,w)$, we also set $\epsilon_{\underline{k}}(X)=t_{k}(X)$ and $\varphi_{\underline{k}}(X)=t_{k}(X)+ \langle h_{\underline{k}},wt(X) \rangle$, then we have the following theorem.

\begin{theorem}\label{4.15}
	Given a symmetrizable generalized Cartan matrix $C=(c_{ij})_{i,j\in I'}$, let $Q=(I,H,\Omega)$ be an associated quiver with an admissible automorphism $a$. We identify $\underline{I}$ with $I'$, then the crystal structure of $(\bigcup\limits_{v \in \mathbb{N}I^{a}}\text{Irr}(\mathfrak{L}(v, w))^a,\tilde{f}_{\underline{k}},\tilde{e}_{\underline{k}},\epsilon_{\underline{k}},\varphi_{\underline{k}} ,\underline{k} \in I' )$ is isomorphic to the crystal structure $B(\lambda)$ of the irreducible highest weight module $L(\lambda)$ of the quantum group associated to $C$. 
\end{theorem}
\begin{proof}
	Let $\mathcal{B}$ be the set of $[(L,\phi)]$, where $(L,\phi)\cong \mathbf{D}((L,\phi))$. We say $(L,\phi)$ is equivalent to $(L',\phi')$ if $[(L,\phi)]=\pm [(L',\phi')]$, and denote the set of equivalence classes of $\mathcal{B}$ by $\mathcal{B}'$. A crystal structure on $\mathcal{B}'$ can be described as follows:
	\begin{equation*}
		\begin{split}
			\tilde{e}_{\underline{i}}([(L,\phi)] ) = 
			\begin{cases} 
				\pi^{-1}_{\underline{i}, r-1}\pi_{\underline{i}, r}([(L,\phi)]) & \text{if } t_{i}(L)=r>0,\\ 
				0 & \text{otherwise};
			\end{cases}
		\end{split}
	\end{equation*}
	and
	\begin{equation*}
		\tilde{f}_{\underline{i}}([(L,\phi)]) = 	\pi^{-1}_{\underline{i}, r+1}\pi_{\underline{i}, r}([(L,\phi)]), \text{if } t_{i}(L)=r.
	\end{equation*}
Applying \cite[Theorem 18.3.8]{lusztig2010introduction} to the signed basis of ${_{\mathcal{A}}\mathcal{K}(\nu,\omega)} \cong {_{\mathcal{A}}L(\lambda)}$, it follows that there is an isomorphism of crystals $\mathcal{B}' \cong B(\lambda) $. Hence it suffices to show $\bigcup\limits_{v \in \mathbb{N}I^{a}}\text{Irr}(\mathfrak{L}(v, w))^a$ is isomorphic to $\mathcal{B}'$ as crystals.

 	Denote $\tilde{e}_{\underline{i}}^{max} ([(L,\phi)])=\tilde{e}_{\underline{i}}^{r} ([(L,\phi)])= [\pi_{\underline{i}, r}((L,\phi))] $ for those pairs with $t_{i}(L)=r$. Similarly, we can define $\tilde{e}_{\underline{i}}^{max}(X)=\rho_{\underline{i}, r}(X) $ for an irreducible component $X$ with $t_{i}(X)=r$. By \cite[Theorem 4.20]{fang2023lusztigsheavesintegrablehighest}, there is a bijection $\Phi$ from the set $\mathcal{P}_{\bfW}$ of nonzero simple perverse sheaves in $\coprod\limits_{\bfV}\mq_{\bfV, \bfW}/\mn_{\bfV}$ to the set $\bigcup\limits_{v \in \mathbb{N}I^{a}}\text{Irr}(L(v, w))$ of irreducible components such that $\pi_{i,r}(L)=K$ if and only if $\rho_{i,r}(\Phi(L))=\Phi(K).$
 	
 	If we forget the isomorphisms  $\phi$, then $\mathcal{B}'$ is in bijection with the set of $a^*-$invariant nonzero simple perverse sheaves in $ \coprod\limits_{\bfV}\mq_{\bfV, \bfW}/\mn_{\bfV}$ and 
 	$\pi_{\underline{i},r}(L) =\prod\limits_{i \in \underline{i}}\pi_{i,r} (L)$ for any simple perverse sheaf $L$ with $a^{\ast}L \cong L$. In particular, $\Phi$ restricts to a bijection $\Phi^{a}$ between the $a$-fixed points of $\mathcal{P}_{\bfW}$ and those of $\bigcup\limits_{v \in \mathbb{N}I^{a}}\text{Irr}(L(v, w))$, such that $\pi_{\underline{i},r}(L)=K$ if and only if $\rho_{\underline{i},r}(\Phi^{a}(L))=\Phi^{a}(K).$ Since the $a$-fixed points of $\mathcal{P}_{\bfW}$ are naturally in bijection with $\mathcal{B}'$ and the $a$-fixed points of $\bigcup\limits_{v \in \mathbb{N}I^{a}}\text{Irr}(L(v, w))$ are $\bigcup\limits_{v \in \mathbb{N}I^{a}}\text{Irr}(L(v, w))^{a}$, we obtain a bijection $\Phi^{a}: \mathcal{B}' \rightarrow \bigcup\limits_{v \in \mathbb{N}I^{a}}\text{Irr}(L(v, w))^{a} $ which intertwines $\pi_{\underline{i},r} $ and $\rho_{\underline{i},r}$. Since the crystal operators $\tilde{e}_{\underline{i}}$ and $\tilde{f}_{\underline{i}}$ are uniquely determined by $\pi_{\underline{i},r} $ and $\rho_{\underline{i},r}$, it follows that $\Phi^{a}: \mathcal{B}' \rightarrow \bigcup\limits_{v \in \mathbb{N}I^{a}}\text{Irr}(L(v, w))^{a} $ is an isomorphism of crystals in the sense of Kashiwara \cite{crystal}.
 	 
\end{proof}

 Choose the dimension vector of $\bfW$ sufficiently large so that $\lambda=\sum\limits_{i' \in I'}\omega_{i'}\beta_{i'}$ satisfies $\omega_{i'}>\nu_{i}$ for each $i$ in the $a$-orbit corresponding to $i'$ and so that $B(\lambda)_{\nu}$ is in bijection with $B(\infty)_{\nu}$ for a fixed weight $\nu$. Then we obtain the following corollary, which is another main result of \cite{SA}.
\begin{corollary}
	Given a symmetrizable Cartan matrix $C=(c_{ij})_{i,j\in I'}$, let $Q=(I,H,\Omega)$ be an associated quiver with an admissible automorphism $a$, let $\bigcup\limits_{\bfV}\text{Irr}(\Lambda_{\bf V})^{a}$ be the set of $a$-fixed irreducible components of Lusztig's nilpotent variety $\coprod\limits_{\bfV}\Lambda_{\bf V}$ of the preprojective algebra of $Q$. Then $\bigcup\limits_{\bfV}\text{Irr}(\Lambda_{\bf V})^{a}$ carries a natural crystal structure and is canonically isomorphic to $B(\infty)$ of the quantum group associated to $C$.
\end{corollary} 
\begin{proof}
	Recall that by \cite[Lemma 5.8]{N94}, the set of irreducible components of $L(v, w)$ is in bijection with a subset of irreducible components of $\Lambda_{\bf V}$ via the projection $(B,0,j) \mapsto B$. Observe that if $\bfW$ is large enough, the map $(B,0,j) \mapsto B$ induces a bijection from $\text{Irr}(L(v, w))$ to $\text{Irr}(\Lambda_{\bf V})$, which is compatible with the crystal operator. Restricting this bijection to the $a$-fixed irreducible components completes the proof.
\end{proof}
\begin{remark}
	When the automorphism $a$ is trivial, i.e, the Cartan matrix is symmetric, our results also agree with the main result of \cite{KS} and \cite{Saito2002}.
\end{remark}

\subsection{Symmetrizable crystal structure of tensor products}
With the notation of Section 5.2, fix a decomposition $\mathbf{W}=\mathbf{W}^{1} \oplus \mathbf{W}^{2}$ such that their dimension vectors $|\mathbf{W}^1|=\omega^1$ and $|\mathbf{W}^2|=\omega^2$, then $\omega=\omega^1+\omega^2$. Recall that the group $\mathrm{GL}_{\mathbf{W}}$ acts on $\mathbf{M}$ by conjugation. This induces a $\mathrm{GL}_{\mathbf{W}}$-action on $\mathcal{M}$. Then the one-parameter subgroup 
\begin{align*}
	\lambda:\mathbb{G}_m&\rightarrow \mathrm{GL}_{\mathbf{W}^1}\times \mathrm{GL}_{\mathbf{W}^2}\subset \mathrm{GL}_{\mathbf{W}}\\
	t&\mapsto ({\mathrm{Id}}_{\mathbf{W}^{1}},t\, {\mathrm{Id}}_{\mathbf{W}^{2}})
\end{align*}
acts on $\mathcal{M}$. Nakajima's (Lagrangian) tensor product variety is defined to be
$$\tilde{\mathfrak{Z}}(\nu,\omega)=\{ [x,y,z] \in \mathcal{M}(\nu,\omega)\mid \lim_{t \rightarrow 0} \lambda(t). [x,y,z] \in \bigsqcup_{\nu'+\nu''=\nu} \mathfrak{L} (\nu',\omega^{1}) \times \mathfrak{L}(\nu'',\omega^{2}) \}.$$

Similarly, the morphism $p$ also induces a bijection 
 $$\rho_{k, r}: \{X \in \text{Irr}(\tilde{\mathfrak{Z}}(\nu,\omega) \mid t_k(X) = r\} \rightarrow \{X \in \text{Irr}(\tilde{\mathfrak{Z}}(\nu-rk,\omega)) \mid t_k(X) = 0,\langle h_{k},wt(X) \rangle\geq r\}.$$
 
 The admissible automorphism $a$ acts naturally on $\tilde{\mathfrak{Z}}(\nu,\omega)$ and induces a permutation, which is still denoted by $a$, on the set of its irreducible components. Let $\text{Irr}^{a}\tilde{\mathfrak{Z}}(\nu,\omega)$ be the set of $a$-fixed irreducible components. For any irreducible $X \in \text{Irr}^{a}\tilde{\mathfrak{Z}}(\nu,\omega)$, if $t_{k}(X)=r$, then $t_{k'}=r$ for any $k' \in \underline{k}$. Then the map $\rho_{\underline{k}, r}(X) = \prod\limits_{k \in \underline{k}}\rho_{k, r}(X)$ is well-defined. Since each $\rho_{k, r}$ is bijective, it follows that $\rho_{\underline{k}, r}(X)$ is also fixed by $a$. Hence $\rho_{\underline{k}, r}$ is a bijection between the corresponding sets of $a$-fixed irreducible components.
 
 These maps $\rho_{\underline{k}, r}$ induce crystal operators on the set $\bigcup\limits_{v}\text{Irr}^{a}\tilde{\mathfrak{Z}}(\nu,\omega)$ in the following way: 
 \begin{equation*}
 	\begin{split}
 		\tilde{e}_{\underline{k}}(X) = 
 		\begin{cases} 
 			\rho^{-1}_{\underline{k}, r-1}\rho_{\underline{k}, r}(X) & \text{if } t_{k}(X)=r>0, \rho^{-1}_{\underline{k}, r-1}\textbf{ exists}\\ 
 			0 & \text{otherwise};
 		\end{cases}
 	\end{split}
 \end{equation*}
 \begin{equation*}
 	\tilde{f}_{\underline{k}}(X) = 	\rho^{-1}_{\underline{k}, r+1}\rho_{\underline{k}, r}(X), \text{if } t_{k}(X)=r;
 \end{equation*}
 $$\epsilon_{\underline{k}}(X)=t_{k}(X);$$ 
 and $$\varphi_{\underline{k}}(X)=t_{k}(X)+ \langle h_{\underline{k}},wt(X) \rangle.$$ Then we have the following theorem.

 \begin{theorem}\label{5.7}
 	Given a symmetrizable generalized Cartan matrix $C=(c_{ij})_{i,j\in I'}$, let $Q=(I,H,\Omega)$ be an associated quiver with an admissible automorphism $a$. Identifying $\underline{I}$ with $I'$, the crystal structure of $(\bigcup\limits_{v \in \mathbb{N}I^{a}}\text{Irr}(\tilde{\mathfrak{Z}}(\nu,\omega))^a,\tilde{f}_{\underline{k}},\tilde{e}_{\underline{k}},\epsilon_{\underline{k}},\varphi_{\underline{k}} ,\underline{k} \in I' )$ is isomorphic to the crystal structure $B(\lambda_{1})\otimes B(\lambda_2)$ of the tensor product $L(\lambda_{1}) \otimes L(\lambda_2)$. Here $\lambda_1$ and $\lambda_2$ are the dominant weights of the quantum group associated to $C$ determined by $\mathbf{W}^{1}$ and $\mathbf{W}^{2}$ respectively. 
 \end{theorem}
 
 \begin{proof}
 	We have the following facts.
 	
 	(1) 	In \cite{fang2025lusztigsheavescharacteristiccycles}, there is a bijection $\Phi=\Phi_{Q^{(2)}}$ between the set $\bigcup\limits_{v \in \mathbb{N}}\text{Irr}(\tilde{\mathfrak{Z}}(\nu,\omega))$ and the set of nonzero simple perverse sheaves in $\mathcal{Q}_{\mathbf{V},\mathbf{W}^{\bullet}}/ \mathcal{N}_{\mathbf{V}}$ such that $\Phi$ intertwines the bijections $\rho_{i,r}$ and $\pi_{i,r}$.
 	
 	(2) The bijections $\pi_{i,r}$ induce operators 
 		\begin{equation*}
 			\begin{split}
 				\tilde{e}_{i}(L) = 
 				\begin{cases} 
 					\pi^{-1}_{i, r-1}\pi_{i, r}(L) & \text{if } t_{i}(L)=r>0, \pi^{-1}_{i, r-1}\textbf{ exists}\\ 
 					0 & \text{otherwise};
 				\end{cases}
 			\end{split}
 		\end{equation*}
 		and
 		\begin{equation*}
 			\tilde{f}_{i}(L) = 	\pi^{-1}_{i, r+1}\pi_{i, r}(L), \text{if } t_{i}(L)=r.
 		\end{equation*}
 		By \cite[Lemma 5.4]{FLalten}, the operators defined above can be identified with Kashiwara's crystal operators on the tensor product crystals for the symmetric quantum group associated to $Q$ via the canonical isomorphism $\Delta$. (That is, we forget the automorphism $a$.)

 	Restricting (1) and (2) to the $a$-fixed subsets of $\text{Irr}(\tilde{\mathfrak{Z}}(\nu,\omega))$ and $\mathcal{P}_{\mathbf{V},\mathbf{W}^{\bullet}} \backslash (\mathcal{N}_{\mathbf{V}} \cap \mathcal{P}_{\mathbf{V},\mathbf{W}^{\bullet}})$, the proof follows by the same argument as in Theorem \ref{4.15}. 	
 \end{proof}
\subsection*{Author Contributions}
Lan, Wu, and Xiao contributed equally to this manuscript.
\subsection*{Funding}
Y. Lan, Y. Wu and J. Xiao are supported by National Natural Science Foundation of China [Grant No. 12031007] and [Grant No. 12471030].
\subsection*{Acknowledgement} This paper is a continuation of a collaborative work \cite{fang2023lusztigsheavesintegrablehighest} with Jiepeng Fang, and we are very grateful to him for many helpful discussions.
\subsection*{Data Availability}
No datasets were generated or analyzed during the current study.
\section*{Declarations}
\subsection*{Competing interests}
The authors declare no competing interests.
\end{spacing}

\bibliography{ref}

\end{document}